\documentclass{amsart}
\usepackage{amsfonts,amssymb,amsmath, graphicx, bm, amscd, bbm}
\usepackage[all]{xy}

\usepackage%[backref=page]
{hyperref}
\hypersetup{colorlinks=true,citecolor=blue,linkcolor=blue,urlcolor=blue}
\def\cal{\mathcal}

\def\Li{{\rm Li_2}}
\def\dd{\,{\mathrm d}}
\newtheorem{thm}{Theorem}[section]

\newtheorem{prop}[thm]{Proposition}
\newtheorem{lem}{Lemma}[section]

\numberwithin{equation}{section}

\def\NN{{\mathbb N}}
\def\ZZ{{\mathbb Z}}

\def\QQ{{\mathbb Q}}
\def\CC{{\mathbb C}}

\def\eps{\varepsilon}
\def\al{\alpha}

\def\rho{\varrho}
\def\phi{\varphi}

\renewcommand{\Re}{\mathrm{Re}\,}

\newcommand{\f}[2]{\frac{#1}{#2}}
\newcommand{\be}[1]{\begin{equation}\label{#1}}
\newcommand{\ee}{\end{equation}}
\newcommand{\multsum}[2]{\sum_{{\scriptstyle #1}\atop {\scriptstyle #2}}}

\newcommand{\rf}[1]{{\rm (\ref{#1})}}

\def\ii{{\mathrm i}}

\begin{document}

\author{Valentin Blomer}
\author{J\"org Br\"udern}
\author{Per Salberger}
\address{Mathematisches Institut, Bunsenstr. 3-5, 37073 G\"ottingen, Germany} \email{vblomer@math.uni-goettingen.de}\email{jbruede@mathematik.uni-goettingen.de}
\address{Mathematical Sciences, Chalmers University of Technology, SE-412 96 G\"oteborg, Sweden} \email{salberg@chalmers.se}

\title{The Manin-Peyre formula for a certain biprojective threefold}

\thanks{The first author was supported  by the Volkswagen Foundation, a  Starting Grant of the European Research Council and a grant from Deutsche Forschungsgemeinschaft. The second author was supported by a grant from Deutsche Forschungsgemeinschaft.}


\begin{abstract}  The conjectures of Manin and Peyre  are confirmed for a certain threefold. 
 \end{abstract}


\maketitle

%\tableofcontents

 \section{Introduction}

\subsection{The main result}
In a recent memoir \cite{BBS} we confirmed the predictions of Manin and Peyre  for the distribution of rational points on the cubic fourfold in $\Bbb{P}^5$ defined by 
\begin{equation}\label{1}
x_1y_2y_3 + x_2y_1y_3 + x_3y_1y_2 = 0.
\end{equation}
Here, we continue our study of this equation, now viewing the polynomial on the left hand side
as a linear form in $\mathbf x = (x_1,x_2,x_3)$ and a quadratic form in $\mathbf y = (y_1,y_2,y_3)$. With $\mathbf x$ and $\mathbf y$ interpreted as homogeneous coordinates, the equation
\rf{1} defines a variety $V$ in $\Bbb{P}^2 \times \Bbb{P}^2$. For a $\QQ$-rational point on $V$ there
are representatives  $\mathbf x$, $\mathbf y \in\ZZ^3$ with $(x_1;x_2;x_3)=(y_1;y_2;y_3)=1$,
both unique up to sign. An anticanonical height function on $V$ is then given by
\begin{equation}\label{height}
  H(\mathbf{x}, \mathbf{y}) = \max_{1\le i,j\le 3} |x_i|^2  |y_j|.
\end{equation}   
Rational points on $V$ ordered with respect to this height accumulate on the subvariety
cut out from $V$ by the additional equation $x_1x_2x_3y_1y_2y_3=0$. To see this, note that
the choices $\mathbf{x} = (0, 1, 1)$ and $\mathbf{y} = (y_1, y_2, -y_2)$ with $(y_1; y_2) = 1$
produce more than
$ B^2$ rational points of height at most $B$ on this subvariety, while on the 
Zariski-open subset $V^\circ$ of $V$ where $x_1x_2x_3y_1y_2y_3 \not= 0$ the rational points are much sparser.
This is a consequence of the following asymptotic formula.

\begin{thm}\label{thm1} Let $N(B)$ denote the number of rational points  on  $V^\circ$ with height not exceeding $B$, and let
\be{EPC} C=\prod_p \Big(1-\frac{1}{p}\Big)^5 \Big(1 + \frac{5}{p} + \frac{5}{p^2} + \frac{1}{p^3}\Big). \ee
Then
\be{asym1} 
N(B) =  \frac{\pi^2 - 3 + 24\log 2}{ 144}\, C B(\log B)^4 + O\big(B (\log B)^{4 - \frac{1}{480}}\big). 
\ee
\end{thm}

Hitherto it was only known \cite{BB1} that $N(B) \asymp B (\log B)^4$. No effort has been made to optimize the error term in \rf{asym1}. With more work it is possible to show
that there is a polynomial $P$ of degree four and a positive number $\delta$ with the
property that
\be{asymp2} N(B) = BP(\log B) + O(B^{1-\delta}), \ee
but in order to keep the paper at reasonable length we have to content ourselves with
a detailed proof of \rf{asym1}. The reader is referred  to Section 1.2 for a brief summary of refinements needed to establish \rf{asymp2}.  

The shape of the asymptotic formula \rf{asym1} is in line with a general conjecture of Manin
(see \cite{FMT}) concerning the distribution of rational points on smooth Fano varieties.
However, $V$ has three singularities located at $x_i = x_j = y_i = y_j$ for $1 \leq i < j \leq 3$. A resolution has been obtained in \cite[Theorem 4]{BBS}: Let $X \subset \Bbb{P}^2 \times \Bbb{P}^2 \times \Bbb{P}^2$ be the triprojective variety defined in trihomogeneous coordinates $(\mathbf{x}, \mathbf{y}, \mathbf{z})$  by
\begin{equation}\label{tri}
  x_1z_1 + x_2z_2 + x_3z_3 = 0 \quad \text{and}\quad 
   y_1z_1=y_2z_2 = y_3z_3.
\end{equation}
Then the restriction to $X$ of the projection $\Bbb{P}^2 \times \Bbb{P}^2 \times \Bbb{P}^2 \rightarrow \Bbb{P}^2 \times \Bbb{P}^2$ onto the first two factors is a \emph{crepant} resolution of $V$, and one has $\text{\rm rk } \text{\rm Pic}(X) = 5$. In this situation, 
Batyrev and Tschinkel \cite{BT} predict that $N(B)/(B(\log B)^4)$ tends to a limit as $B\to\infty$, and Peyre \cite{Pe2} suggested a formula for this limit. At the end of this paper,
in Sections \ref{geometry} and \ref{peyre-constant}, we show that our findings in Theorem \ref{thm1} agree with Peyre's formula. Notice in particular that  the $p$-adic factor of the constant $C$ in \eqref{EPC} is the $p$-adic density of the universal torsor over $X$, which in our case is  a $\Bbb{G}_m^5$-torsor over a $\Bbb{P}^1$-bundle over a del Pezzo surface of degree six.

\medskip

The Manin-Peyre conjectures for the distribution of rational points on algebraic varieties
have received considerable attention in recent years. Powerful techniques are available for surfaces (see e.g.\ the references in \cite{Br}). Moreover, the circle method typically produces
asymptotic relations that confirm the conjectures, but this requires the dimension be large in terms of the degree.
If the variety carries additional structure, further tools can be brought into play.
For example, when the variety is
an equivariant compactification of certain linear algebraic groups, Tschinkel and his collaborators applied adelic Fourier analysis to prove Manin's conjecture in some generality (see e.g. \cite{CLT, STBT}). 

The variety under consideration is not covered by the cases just described.  Definitive results for Fano threefolds are very rare besides the remarkable paper of de la Bret\`eche \cite{dlB} on the Segre cubic $(x_1+\ldots + x_5)^3 = x_1^3 + \ldots x_5^3$. Le Boudec \cite{Bou} determined the order of magnitude for the number of rational points of bounded height on the    biprojective threefold 
$x_1y_1^2 + x_2y_2^2 + x_3y_3^2 = 0$, and we agree with him that a refinement to an asymptotic formula seems
``far out of reach".

Although we are concerned here with just one concrete example, the methods that underpin the 
the proof of Theorem \ref{thm1} are by no means restricted to the case at hand. The family of varieties
defined by
$$ \f{x_1}{y_1} + \f{x_2}{y_2} + \cdots + \f{x_n}{y_n} = 0 $$
springs to mind of which we treat the case $n=3$ here. Larger $n$ should be within reach for the techniques described herein, and we intend to return to the theme in a broader setting in due course. 
We  hope that the present example spurs further work
on higher dimensional cases of the Manin-Peyre formula.

\subsection{The methods} The 
ideas behind the proof of Theorem \ref{thm1} have some similarity with our earlier work \cite{BBS} where  the cubic form \eqref{1} was studied as a fourfold in $\Bbb{P}^5$. Yet, there are several fundamental differences. 
The initial step is the same as in \cite{BBS}. An elementary argument transfers the original counting problem to one on the universal torsor. The latter is given by  
\begin{equation}\label{T} 
u_1v_1+u_2v_2+u_3v_3 = 0,
\end{equation}
and it is the simple shape of this bilinear equation what makes the variety $V$ accessible to analytic methods.
It is typical that box-like conditions on the original equation transform to regions with hyperbolic spikes on the torsor. In the situation considered here, the anticanonical height function \rf{height}
involves a product, resulting in very narrow spikes where integral points are difficult to count. This forced us to waive the strategy followed in \cite{BBS} where the solutions of \rf{T} were parametrized in the obvious way, leading to a hyperbolic lattice point problem that was then approached through multiple Dirichlet series. Instead, we use Fourier analysis directly to count points on the
torsor. At the core of the method, we then require an asymptotic formula for the number $N_{\mathbf{r}}(\mathbf{X}, \mathbf{Y})$ of solutions 
to  the equation
\be{T1}r_1x_1y_1 + r_2x_2y_2 + r_3x_3y_3 = 0\ee
in
${\mathbf x}\in{\mathbb Z}^3$, ${\mathbf y}\in{\mathbb Z}^3$ within the region 
$${\cal B}(\mathbf{X}, \mathbf{Y}) = \left\{(\mathbf{x}, \mathbf{y}) \in \Bbb{R}^{3} \times \Bbb{R}^3 \mid \textstyle \frac{1}{2}X_j < |x_j| \leq X_j, \frac{1}{2}Y_j < |y_j| \leq Y_j, \, 1 \leq j \leq 3\right\}. $$
 Here ${\mathbf r} = (r_1, r_2, r_3) \in{\mathbb N}^3$ are given coefficients, $\mathbf{X}, \mathbf{Y}$ are triples of positive real numbers, and we need the asymptotic formula uniformly with respect to $\mathbf r$. 
Note that ${\cal B}(\mathbf X, \mathbf Y)$ consists of $2^6$  possibly very lopsided boxes. Nonetheless 
this counting problem is within the competence of the circle method. The following  proposition %\ref{asymp} in Section \ref{heart} 
 delivers the desired 
asymptotics, and we shall save a small power of the smallest side of the box. The asymptotic formula involves the singular series
\begin{equation}\label{defer}
  {\cal E}_\mathbf{r} = \sum_{q = 1}^{\infty} \frac{\phi(q) (q; r_1)(q; r_2)(q; r_3)}{q^3} 
\end{equation}
and the singular integral
  \begin{equation}\label{defI}
  {\cal I}_\mathbf{r}(\mathbf{X}, \mathbf{Y}) =  \int_{\Bbb{R}} \int_{{\cal B}(\mathbf{X}, \mathbf{Y})} %\int_{\substack{X_2 \leq |x_2| \leq 2X_2\\ Y_2 \leq |y_2| \leq 2Y_2}}\int_{\substack{X_3 \leq |x_3| \leq 2X_3\\ Y_3 \leq |y_3| \leq 2Y_3}} 
  e(\alpha( r_1x_1y_1 + r_2x_2y_2 + r_3x_3y_3) )  \dd(\mathbf{x} ,  \mathbf{y})  \dd\alpha.
\end{equation}

\begin{prop}\label{asymp} Let $X_1, X_2, X_3,  Y_1, Y_2, Y_3 \geq 1$ and  $r_1, r_2, r_3 \in \Bbb{N}$.   
Then
$$N_{\mathbf{r}}(\mathbf{X}, \mathbf{Y}) = {\cal E}_{\mathbf{r}}  {\cal I}_{\mathbf{r}}(\mathbf{X}, \mathbf{Y}) + \Theta_{\mathbf{r}}(\mathbf{X}, \mathbf{Y})$$
where for  any fixed positive value of  $\varepsilon$ one has   
$$\Theta_{\mathbf{r}}(\mathbf{X}, \mathbf{Y}) \ll \frac{ (r_1X_1Y_1 \cdot r_2X_2Y_2 \cdot r_3X_3Y_3)^{1+\varepsilon} }{\max(r_1X_1Y_1, r_2X_2Y_2, r_3X_3Y_3)  \min(X_1, X_2, X_3, Y_1, Y_2, Y_3)^{1/6}}.$$
The implicit constant depends at most on $\varepsilon$. 
\end{prop}

This asymptotic formula is what the circle method predicts, although our proof uses a different argument in which the key ingredient is a non-trivial bound for Kloosterman sums. 
%Lemma \ref{mellin} gives an alternative expression for the ``singular integral" ${\cal I}_{\mathbf{r}}(\mathbf{X}, \mathbf{Y})$. 
The dependence on $r_1, r_2, r_3$ in the error term $\Theta_{\mathbf{r}}(\mathbf{X}, \mathbf{Y})$ can be improved. 

\smallskip

With this result in hand, one can  apply a simple version of the patchwork method developed in \cite{BB2}. For some small $\delta>0$, we
glue together
the contributions from boxes where 
\begin{displaymath}
  \min(X_1, X_2, X_3, Y_1, Y_2, Y_3) \geq  \max(X_1, X_2, X_3, Y_1, Y_2, Y_3)^{\delta}.
\end{displaymath}
This keeps us away from the spikes, and then we send $\delta$ to $0$ at an appropriate speed. In this way, the ideas underpinning the proof \cite[Lemma 2.8]{BB2} deliver the conclusions recorded in  Theorem \ref{thm1}. The factorization of the leading constant in \eqref{asym1} is imported from similar properties of the main term in the asymptotics announced in Proposition \ref{asymp}. We are fortunate that all local factors can be computed explicitly.

Once Theorem \ref{thm1} is established, we turn to the task of comparing the result with the predictions made by Manin and Peyre. In this endeavour,
we need to compute some invariants of
the crepant resolution $X$, cf.\ \eqref{tri}.   To compute Peyre's alpha
invariant $\alpha(X)$% in Section \ref{geometry}
, we endow the invertible
$O_X$-modules with $\Bbb{G}_m^3$-linearizations compatible with a (non-faithful) 
$\Bbb{G}_m^3$-action on $X$ induced  by the  embedding of $X$ in
$\Bbb{P}^2 \times \Bbb{P}^2 \times \Bbb{P}^2$. This implies (see Lemma
\ref{lem21}) that any effective divisor on $X$ is linearly equivalent to a
$\Bbb{G}_m^3$-invariant effective divisor. It is then not too hard to show
that the pseudo-effective cone in $\text{Pic }X$  is spanned by nine
$\Bbb{G}_m^3$-invariant prime divisors and to calculate $\alpha(X)$.

To compute the adelic volume in Peyre's constant,  we proceed as in
\cite{BBS} and relate the volume forms on the non-singular locus of $V$ to Poincar\'e residues of meromorphic forms on $\Bbb{P}^2\times \Bbb{P}^2$  with poles along $V$. We then see that the Euler product  in 
Theorem \ref{thm1} is the expected one, and also that the product of the $\al$-invariant
and the archimedean volume agrees with the first factor on the right hand side of \eqref{asym1}.

As we have pointed out before, our arguments can be refined further. Indeed, one may develop Proposition  \ref{asymp} so as to cover the case
where some of the variables in \rf{T1} are fixed, and equipped with this, one can use the machinery
from \cite{BB2} in full, to cope with the cuspidal portion of the counting more precisely. 
This analysis provides error terms that save a power of the largest side of the box. Then, by the main result of \cite{BB2} one obtains \rf{asymp2}. Also, one may sum by parts to  obtain an analytic 
 continuation of the 6-fold Dirichlet series
\be{zeta}
   \left.\sum \right. ' %\sum_{\substack{x_1, x_2, x_3, y_1, y_2, y_3 \in \Bbb{Z} \setminus \{0\}\\r_1x_1y_1 + r_2x_2y_2 + r_3x_3y_3 = 0}} %
   \frac{1}{|x_1|^{s_1}|x_2|^{s_2}|x_3|^{s_3}|y_1|^{t_1}|y_2|^{t_2}|y_3|^{t_3}},
\ee
where the prime indicates summation over $(\mathbf{x}, \mathbf{y}) \in (\Bbb{Z} \setminus \{0\})^6$ satisfying \eqref{T1}. 
If one specialises to $r_1=r_2=r_3=1$ and then restricts to the diagonal $s_1=s_2=s_3$,
$t_1=t_2=t_3$,  this series 
is essentially a minimal parabolic Eisenstein series for $\text{SL}_3(\Bbb{Z})$, see \cite{Bu}.  
The series \rf{zeta} is a far-reaching generalization of this well-understood Eisenstein series that,
apparently, does no longer memorize the group theoretic information carried by its ancestor. 
Perhaps this is the reason for the considerable complexity of
our analysis of the threefold defined by
 \eqref{1}.  

Details of the arguments outlined in the  preceding paragraph will 
not be  worked out in this paper.  Armed with this refinement of  Proposition \ref{asymp}
it would be straightforward but elaborate to do so.  

\medskip

{\em Notation}. Most of the notation used in this paper is either standard or
otherwise explained at the appropriate stage of the argument. However, traditional notation in
the various branches in mathematics on which our work is built resulted in clashes, and
 the desire for entire consistency
in this respect turned out to be impracticable. The following guide may help the reader to clarify
the symbolism in the work to follow.

Throughout, we apply the following convention  concerning the letter $\eps$. 
Whenever $\eps$ occurs in a
statement, may it be explicitly or implicitly, then it is asserted that the statement is true for any
fixed positive real number in the role of $\eps$. Constants implicit in the use of Landau's or Vinogradov's well-known symbols may then depend on the  value assigned to $\eps$. Note that this  allows us 
to conclude from the inequalities $A\ll X^\eps$ and $B\ll X^\eps$ that $AB\ll X^\eps$, for example.

Frequently, we use vector notation $\mathbf x = (x_1,\ldots,x_n)$ where the underlying field and the 
dimension $n$ is usually clear from the context. If the coordinates $x_j$ are complex numbers, we write
\be{abso}|\mathbf x|=\max_j |x_j|, \quad  |\mathbf x|_1= \sum_{j=1}^n |x_j|. \ee
Further, when  $\mathbf{x} \in \Bbb{R}^{n}$ and $\mathbf{a} \in \Bbb{C}^n$ we write  
\begin{equation}\label{power}
\mathbf{x}^{\mathbf{a}} = |x_1|^{a_1}   |x_2|^{a_2}\ldots |x_n|^{a_n}.
\end{equation}

We will often have to integrate over vertical lines in the compex plane. In this context, when $c$ is a real number,
the parametrized line $(-\infty,\infty)\to \CC$, $ t\mapsto c+i t$ is denoted by $(c)$.

The number of divisors of the natural number $n$ is $\tau(n)$.  The M\"obius function is denoted by
$\mu(n)$, and $\phi(n)$ is Euler's totient. The greatest common divisor of the non-zero integers
$a_1,\ldots,a_n$ is denoted by $(a_1;\ldots;a_n)$, and their least common multiple is  
 $[a_1;\ldots;a_n]$. When $f:\NN\to\CC$ is an arithmetical function and $\mathbf a\in\NN^n$, we put, by slight abuse of notation, 
$$f({\mathbf a})=f(a_1)f(a_2)\cdots f(a_n).$$

We put $\ZZ_0=\ZZ\setminus\{0\}$. The cardinality of a finite set $\cal S$ is $|{\cal S}|$. For a real number $\theta$ we put $e(\theta) = \exp (2\pi\ii\theta)$.

\section{Auxiliary tools}

\subsection{Smoothing}\label{smoothing} 
 
To accelerate convergence of certain integrals, we need a smooth approximation to 
the characteristic function of the unit interval $f_0 : [0, \infty) \rightarrow [0, 1]$. This is achieved via a conventional  convolution argument. For $0 < \Delta < 1$ one can find a smooth function $\rho_{\Delta} : [0, \infty) \rightarrow [0,\infty)$ with
\begin{equation}\label{supprho}
  \text{supp}(\rho_{\Delta}) \subset (1, 1+\Delta)\quad 
\mbox{and} 
\quad
\int_0^{\infty} \rho_{\Delta}(x) \frac{\mathrm d x}{x} = 1
\end{equation}
 such that
$
  \rho_{\Delta}^{(j)} (x) \ll_j \Delta^{-1-j}
$
holds for all $j\in \Bbb{N}_0$. Its Mellin transform 
$$\widehat{\rho}_{\Delta}(s) = \int_0^\infty  \rho_{\Delta}(x) x^{s-1}\dd x   $$ is entire and satisfies \begin{equation}\label{mellinrho}
\frac{\dd^j}{\dd s^j}  \widehat{\rho}_{\Delta}(s) \ll_{\Re s, A}   (\Delta |s|)^{-A}
 \end{equation}
for all $s\in\mathbb C$,  all $j \in \Bbb{N}_0$ (in fact uniformly in $j$, although we will only apply it for fixed $j$) and all integers $A \in \Bbb{N}_0$, as one confirms by   integration by parts and differentiation under the integral sign. 
For $x\in[0,\infty)$ we define
\begin{equation}\label{deff}
  f_{\Delta}(x) = \int_0^{\infty} \rho_{\Delta}(z) f_0\left(\frac{x}{z}\right) \frac{\mathrm d z}{z} = \int_x^{\infty} \rho_{\Delta}(z) \frac{\mathrm{d}z}{z}. 
\end{equation}
Then, by \eqref{supprho}  and \eqref{deff},  
\begin{equation}\label{suppfsimple}
   0\le f_{\Delta}(x) \le 1, \quad  f_{\Delta} = 1 \text{ on } [0, 1], \quad \text{supp}(f_{\Delta}) \subset [0, 1+\Delta], \quad f^{(j)}_{\Delta}(x) \ll_j \Delta^{-j}
 \end{equation}  
for $x \in [0, \infty)$ and $j \in \Bbb{N}_0$. Further, we have
\begin{equation}\label{difff}
%f_{\Delta}^{(j)} \ll_j \Delta^{-j}
\widehat{f}_{\Delta}(s) = \widehat{\rho}_{\Delta}(s)/s, \qquad  \widehat{f}_0(s) =1/s .
\end{equation}
 We also note that
$ \text{supp}(f'_{\Delta}) \subset [1, 1+\Delta]$.
 Thus, $f_{\Delta}$ is indeed a smooth approximation 
to $f_0$, and as in \cite[Lemma 24(i)]{BBS} one shows that
\begin{equation}\label{approx}
  \widehat{f}_{\Delta}(s) - \widehat{f}_0(s) \ll \min(\Delta, |s|^{-1}) \qquad (1/100 \leq \Re s \leq 2).
\end{equation}
From \rf{difff} we now see that %It follows from \eqref{mellinrho} and \eqref{difff} that 
\begin{equation}\label{approx2}
  \max\big(\widehat{f}_{\Delta}(s), \, \widehat{f}_{0}(s)\big) \ll |s|^{-1} \qquad (1/100 \leq \Re s \leq 2). 
\end{equation}

Let ${\tt D}$ be the differential operator given by $({\tt D}f)(x) = x f'(x)$ for differentiable functions $f$. Then the Mellin transforms of ${\tt D}f$ and $f$ (for, say, Schwartz class functions $f$) are related by 
\begin{equation}\label{mellin-related}
\widehat{{\tt D}f}(s) = s \widehat{f}(s).
\end{equation}

Let $X\geq 1$ be a parameter. We will also need a smooth approximation to the characteristic function of $\frac{1}{2} X\leq |x| \leq X$. To this end let $1/10 \leq P \leq X/10$ be another parameter and let $v$ be a non-negative smooth function with $v(x) = 1$  for $|x| \in [\frac{1}{2} X - P, X + P]$, $v(x) = 0$ for $|x| \not\in [\frac{1}{2} X - 2P, X + 2P]$ and $\|v^{(j)}\|_{\infty} \ll_j P^{-j}$ for all fixed $j \in \Bbb{N}_0$. We call such a function \emph{of type} $(X, P)$. The Mellin transform $\widehat{v}(s)$ of $v$ is entire, and by partial integration one confirms easily the bound
\begin{equation}\label{mellin-v}
\widehat{v}(s) \ll X^{\Re s}  (1 + |s|P/X)^{-2} 
\end{equation}
in fixed vertical strips.

 \subsection{Certain sum transforms}\label{gcdsec} In this section we consider certain multiple sums with
coprimality constraints on the variables of summation. Such sums occur in the counting process on the torsor, and we wish to remove the coprimality conditions by M\"obius inversion.

Let  ${\cal B}$ denote the
set of all $({\mathbf a,\mathbf d, \mathbf z}) \in   \Bbb{Z}_0^3 \times \Bbb{Z}_0^3 \times \Bbb{Z}_0^3 $ 
that satisfy the coprimality constraints
\begin{equation}\label{22}
(a_1 z_1; a_2 z_2; a_3 z_3 ) =(d_i; d_j) = (z_i; z_j) = (d_k; z_k) =1 \quad  (1\le i < j \le 3,\, 
1\le k\le 3), 
\end{equation}
Note that these conditions  may also be written in the equivalent form
\begin{equation}\label{24}
\begin{array}{ll}
(d_i; d_j) = (z_i; z_j) = (d_k; z_k) =   1 & (1\le i < j \le 3, 
1\le k\le 3), \\
(a_1; a_2; a_3) = (a_i; a_j; z_k) =1 & (\{i, j, k\} = \{1, 2, 3\}).
\end{array}
\end{equation}
The significance of the set $\cal B$ is that it appears naturally in the parametrization of the universal torsor (see Section \ref{sec4}). 

\begin{lem}\label{gcd} Let  $G : \ZZ_0^{3\times 3}  \rightarrow \Bbb{C}$ be a function of compact support. Then
 \begin{equation}\label{magic}
\sum_{(\mathbf a, \mathbf d, \mathbf z)\in\cal B} G(\mathbf a, \mathbf d, \mathbf z)
 =  \sum_{\mathbf{b}, \mathbf{c}, \mathbf{f}, \mathbf{g}\in \NN^3 } \sum_{h=1}^\infty 
\mu((\mathbf{b}, \mathbf{c}, \mathbf{f}, \mathbf{g}, h)) \underset{ \mathbf{a}, \mathbf{d}, \mathbf{z} \in\ZZ_0^3}{\left.\sum \right.^{\sharp}}G(\mathbf a, \mathbf d, \mathbf z),
\end{equation}
in which $\sum^{\sharp}$ denotes that the sum is restricted to values $\mathbf{a}, \mathbf{d}, \mathbf{z} \in\ZZ_0^3$ satisfying
\begin{equation}\label{divi}
 [g_i; g_j; h] \mid a_k, \quad [b_i;b_j;f_k]  \mid d_k, \quad [c_i;c_j;f_k;g_k]  \mid z_k \quad
(\{i, j, k\} = \{1, 2, 3\}).
\end{equation}
\end{lem}

\begin{proof} 
Note that the simultaneous conditions \eqref{divi} are equivalent to  the  divisibility conditions
\begin{equation}\label{divi1}
b_k \mid (d_i; d_j), \quad c_k \mid (z_i; z_j), \quad f_k \mid (z_k; d_k), \quad g_k \mid (a_i; a_j; z_k), \quad h \mid (a_1; a_2; a_3) \quad (\{i, j, k\} = \{1, 2, 3\}).
\end{equation}  
Hence, on applying M\"obius inversion to dissolve all 13 coprimality conditions in \eqref{24}, 
one obtains the desired identity.
 \end{proof}

In the sum on the right hand side of \eqref{magic} it is often desirable to truncate all sums
over $b_j,c_j,f_j,g_j$ and $h$ to an interval $[1,T]$, say. We wish to control the error in doing so,
and for a discussion of this matter, some notation is in order. Suppose that the 13 variables $b_j,c_j,f_j,g_j,h$
$(1\le j\le 3)$ are labelled 1 to 13 in some fixed way, and let $\cal S$ be a non-empty subset
of $\{1,2,\ldots, 13\}$. If the label of some variable is in $\cal S$, then we say that the variable
belongs to $\cal S$. If $c_1$ belongs to $\cal S$ then, by abuse of language, we write $c_1\in\cal S$,
and likewise for other variables.  

 We now claim that the inequality
$$\Bigl|\sum_{\substack{\mathbf{b}, \mathbf{c}, \mathbf{f}, \mathbf{g}, h\\ x > T \text { if } x \in {\cal S} }}
\mu((\mathbf{b}, \mathbf{c}, \mathbf{f}, \mathbf{g}, h)) \underset{ \mathbf{a}, \mathbf{d}, \mathbf{z} \in\ZZ_0^3}{\left.\sum \right.^{\sharp}}G(\mathbf a, \mathbf d, \mathbf z)\Bigr| \leq 
\sum_{\substack{x \in {\cal S}\\ {x>T}}}\; \underset{ \hspace{-0.3cm}\mathbf{a}, \mathbf{d}, \mathbf{z}}{\left.\sum \right.^{{\cal S}}} |G(\mathbf{a}, \mathbf{d}, \mathbf{z})|$$
holds, in which 
 $\sum^{{\cal S}}$ indicates that the sum is restricted to tuples $(\mathbf{a}, \mathbf{d}, \mathbf{z})$ satisfying the divisibility conditions \eqref{divi1} for those variables that belong to ${\cal S}$. Note here that the outer sum consists of $|\cal S|$ independent summations. For a proof of this inequality, we merely have to carry out all summations on the left hand side related to variables that
do {\em not} belong to $\cal S$. Reversing the M\"obius inversion formula, we then import one of the
conditions \eqref{24} from each such sum. After this step, we are left with the summations over variables belonging to $\cal S$, we apply the triangle inequality and then drop the imported coprimality constraints to confirm the inequality as claimed above.

Equipped with this last inequality, we may indeed truncate all outer sums on the right hand side
of \eqref{magic}. The inclusion-exclusion principle then allows us to conclude as follows.

\begin{lem}\label{kor2}  Let $T \geq 1$. In the notation introduced in the preamble to this lemma, 
one has
 \begin{displaymath}%\label{gcd}
 \begin{split}
\sum_{(\mathbf a, \mathbf d, \mathbf z)\in\cal B}  G(\mathbf a, \mathbf d, \mathbf z)
 =  & \sum_{|\mathbf{b}|, |\mathbf{c}|, |\mathbf{f}|, |\mathbf{g}|, h \leq T }  \hspace{-0.5cm}
\mu((\mathbf{b}, \mathbf{c}, \mathbf{f}, \mathbf{g}, h)) \underset{ \mathbf{a}, \mathbf{d}, \mathbf{z} \in\ZZ_0^3}{\left.\sum \right.^{\sharp}}G(\mathbf a, \mathbf d, \mathbf z) + O\Bigl(\sum_{{\cal S}} \sum_{\substack{ x \in {\cal S} \\ x>T}}\; \underset{ \hspace{-0.3cm}\mathbf{a}, \mathbf{d}, \mathbf{z}}{\left.\sum \right.^{{\cal S}}} |G(\mathbf{a}, \mathbf{d}, \mathbf{z})| \Bigr). 
\end{split}
\end{displaymath}
\end{lem}
The conditions \eqref{divi} turn out to be significant in the future analysis, and we  introduce the $3 \times 3$-tuple  ${\bm \alpha} = ({\bm \alpha}_1, {\bm \alpha}_2, {\bm \alpha}_3) \in \Bbb{N}^9$ with $\bm \alpha_j = (\alpha_{j1}, \alpha_{j2}, \alpha_{j3})$, where whenever $\{i, j, k\} = \{1, 2, 3\}$, one takes 
\begin{equation}\label{defalpha}
   \alpha_{1k} = [g_i; g_j; h],   \quad  \alpha_{2k} =  [b_i;b_j;f_k], \quad    \alpha_{3k} = [c_i;c_j;f_k;g_k].
\end{equation}

\subsection{An exponential sum}

In this section we examine an exponential sum of Kloosterman type. When $a,b\in\ZZ$ and $q\in\NN$, the classical Kloosterman sum is defined by
$$ S(a,b;q) = \multsum{x=1}{(x;q)=1}^q e\left(\frac{ax+b\bar x}{q}\right) $$
where, here and later, the bar denotes the multiplicative inverse with respect to a modulus that is always clear from the context; currently this modulus is $q$. Recall Weil's classical estimate 
$|S(a, b; q)| \leq (a, b, q)^{1/2} \tau(q) q^{1/2}$.

For $r, x \in \Bbb{N}$, $h, h_1, h_2 \in \Bbb{Z}$ we define the sum
\be{Kloo}
 S_{r, h}(h_1, h_2;x) = \sum_{\substack{\xi, \eta =1\\ r\xi\eta \equiv -h \,(\text{mod } x)}}^{x} e\left(\frac{h_1\xi + h_2\eta}{x}\right).
\ee
%Notice that even though the right hand side of \eqref{Kloo} depends only on the product $r_2x_2$, the assumption  $(r_2;r_3) = 1$ depends on $r_2$ alone. 
We evaluate the sum \eqref{Kloo}  in terms of Kloosterman sums.

\begin{lem}\label{expsum} Let $r,h, h_1, h_2, x$ be as above. Then one has $S_{r, h}(h_1, h_2;x)=0$ except when
$(r; x) \mid (h; h_1; h_2)$, in which case  $S_{r, h}(h_1, h_2;x) $ equals
$$ 
\sum_{d(r;x) \mid (x; h; h_2)}  d(r;x)^2 S\left(\frac{h_1}{(r; x)}, -\frac{h_2}{d(r; x)} \overline{\frac{r}{(r; x)}} \frac{h}{d(r; x)}, \frac{x}{d(r; x)}\right).$$
\end{lem}

\begin{proof} The sum \rf{Kloo} is empty unless $(r; x) \mid h$, which we assume from now on. We write $r' = r/(r; x)$, $h' = h/(r; x)$ and $x' = x/(r; x)$. Then
\begin{displaymath}
\begin{split}
 S_{r, h}(h_1, h_2;x) &= \sum_{\substack{\xi, \eta =1 \\ \xi\eta \equiv -\overline{r'}h' \, (\text{mod } x')}}^{x} e\left(\frac{h_1\xi + h_2\eta}{x}\right) = \sum_{d \mid (x'; h')} \sum_{\substack{ \xi =1  \\ (\xi; x/d) = 1}}^{x'/d} 
\sum_{\substack{ \eta=1 \\ \eta \equiv - \overline{r'} \frac{h'}{d} \overline{\xi} \,(\text{mod } x'/d)}}^{x}   e\left(\frac{h_1\xi d + h_2\eta}{x}\right). 
\end{split}
\end{displaymath} 
The sum over $\eta$ vanishes unless $d(r; x) \mid h_2$, and in the latter case we find that
\begin{displaymath}
\begin{split}
 S_{r, h}(h_1, h_2;x) & = \sum_{d(r;x) \mid (x; h; h_2)}d(r;x) \sum_{\substack{ \xi =1 \\ (\xi; x'/d) = 1}}^{x/d}     e\left(\frac{h_1\xi d - h_2 \overline{r'} \frac{h'}{d} \overline{\xi}}{x}\right). 
\end{split}
\end{displaymath} 
The sum over $\xi$ vanishes unless $(r;x) \mid h_1$, and we obtain the lemma. 
\end{proof}

\begin{lem}\label{korexp} 
 Let $r,h, h_1, h_2, x$ be as in 
 {\rm (\ref{Kloo})}. Then $  S_{r, h}(0, 0;x)=0$ unless $(r; x) \mid h$, 
in which case
\begin{equation}\label{00}
   S_{r, h}(0, 0;x) = \sum_{d(r; x) \mid (x; h)} d(r; x)^2 \phi\left(\frac{x}{(r; x)d}\right).
\end{equation}
Further, when $h_1h_2\neq 0$, one has the inequalities 
\begin{equation}\label{01}
   |S_{r, h}(0, h_2;x)| \leq \tau(h)(x; hh_2), 
 \end{equation}
\begin{equation}\label{10}
   |S_{r, h}(h_1, 0;x)| \leq \tau(h) (x; hh_1),  
 \end{equation}
\begin{equation}\label{11}
   |S_{r, h}(h_1, h_2;x)| \leq \tau^2(x)  (r; x) x^{1/2}\Big(\frac{hh_1}{(r;x)}; \frac{hh_2}{(r;x)}; x\Big)^{1/2}. 
 \end{equation}
\end{lem}

\begin{proof} The statements concerning $S_{r, h}(0, 0;x)$ are immediate from Lemma \ref{expsum}.  By symmetry it is enough to prove one of the bounds \eqref{01} and \eqref{10}, and we show the latter. Since $h_1 \not= 0$, the standard bound for Ramanujan sums
$|S(a, 0;q)| \leq (a; q)$ suffices to conclude that
\begin{displaymath}
 \begin{split}
  |S_{r, h}(h_1, 0;x)| &\leq \sum_{d \mid (\frac{x}{(r; x)}, \frac{h}{(r; x)})} d(r; x)^2 \left(\frac{h_1}{(r; x)}; \frac{x}{d(r; x)}\right)\\
  & \leq \tau(h) \left(\frac{x}{(r; x)}; \frac{h}{(r; x)}\right)(r; x)^2 \left(\frac{h_1}{(r; x)}; \frac{x}{(r; x) (\frac{x}{(r; x)}; \frac{h}{(r; x)})}\right)\\
  & = \tau(h) (x; hh_1). 
 \end{split} 
\end{displaymath}
Finally, for  $h_1h_2 \not= 0$,  Weil's bound for Kloosterman sums  yields
\begin{displaymath}
 \begin{split}
    |S_{r, h}(h_1, h_2;x)|& \leq \tau(x) \sum_{d(x; r) \mid (x; h; h_2)}  d^{1/2}(x; r) x^{1/2} \left( h_1 ; \frac{h_2h}{d(r; x)}  ;  \frac{x}{d }\right)^{1/2}\\
    & \leq  \tau^2(x) (x; h; h_2)^{1/2} (x; r)^{1/2}  x^{1/2} \left( h_1 ; \frac{h_2h}{(x; h; h_2)}  ;  \frac{x(x; r)}{(x; h; h_2) }\right)^{1/2},
   % & = \tau^2(r_2x_2) (x_2; r_3)^{1/2}  (r_2x_2)^{1/2} \left( h_1 (r_2x_2; h; h_2) ;  h_2h   ;   r_2x_2(x_2; r_3) \right)^{1/2},
      \end{split} 
\end{displaymath}
and \eqref{11} follows. 
\end{proof}

\subsection{Euler products}

\begin{lem}\label{eulerkor}
 Let $a\in \NN$ and  $X \geq 1$, $1/10 \leq P \leq X/10$. Further, let  $v$ be a function of type $(X, P)$ as in Section {\rm \ref{smoothing}}.  Then
\begin{equation}\label{lemma5}
\sum_{n= 1}^\infty \frac{\phi(an)}{n^2} v(n) = \frac{\phi(a)}{\zeta(2)}  \prod_{p \mid a} \left(1 - \frac{1}{p^2}\right)^{-1} \int_0^{\infty} v(x) \frac{\mathrm d x}{x} + O\left(a X^{1/2} P^{-1} \log X\right).
\end{equation}
\end{lem}

\begin{proof} Comparing Euler products, one easily confirms the formula
$$\sum_{n= 1}^\infty \frac{\phi(an)/\phi(a)}{n^s} =  \frac{\zeta(s-1)}{\zeta(s)} \prod_{p \mid a}  \left(1 - \frac{1}{p^s}\right)^{-1} $$
in $\Re s > 2$.  By Mellin inversion we conclude that 
$$\sum_{n\geq 1} \frac{\phi(an)}{n^2} v(n)  =  \phi(a) \int_{(1)} \frac{\zeta(s+1)}{\zeta(s+2)} \prod_{p \mid a} \left(1 - \frac{1}{p^{s+2}}\right)^{-1} \widehat{v}(s) \frac{\mathrm d s}{2\pi \ii}.$$
We shift the contour to $\Re s = -1/2$. The pole at $s=0$ contributes the main term on the right hand side of \eqref{lemma5}. Using \eqref{mellin-v} and Cauchy's inequality, we bound the remaining integral by
$$\ll \frac{a}{X^{1/2}} \int_{(-1/2)} \frac{|\zeta(s+1)|}{ (1 + |s|P/X)^{2}}\,  |\mathrm d s| \ll \frac{a}{P^{1/2}}  \left(\int_{(-1/2)} \frac{|\zeta(s+1)|^2}{ (1 + |s|P/X)^{2}}\,  |\mathrm ds|\right)^{1/2}.$$
 The standard bound $\int_0^T |\zeta(1/2 + \ii t)|^2 \dd t \ll T \log T$ \cite[Section 2.15]{Ti} provides 
the bound $(X/P) \log X/P$ for the last integral, which completes the proof. 
 \end{proof}

 For $\mathbf{r} = (r_1, r_2, r_3) \in \Bbb{N}^3$ and a prime $p$  let $\mathbf{r}(p) = (r_1^{v_p(r_1)}, r_2^{v_p(r_3)}, r_3^{v_p(r_3)})$, where $v_p$ is the usual $p$-adic valuation. With the shorthand notation $r_2' = r_2/(r_2; r_3)$, $r_3' = r_3/(r_2; r_3)$ 
we define
\begin{equation}\label{def-F}
{\cal F}_{\mathbf{r}} = \frac{1}{\zeta(2)} \sum_{abc = r_3'}\frac{\mu(a)}{ab} \sum_{(d; r'_3/b) = 1} \frac{(d; r_2')(db; r_1) }{ d^2} \sum_{fgh = \frac{db}{(db; r_1)}(r_2; r_3) } \frac{\mu(g)}{g}    \prod_{p \mid \frac{r_2'a}{(d; r_2')}}\left(1- \frac{1}{1+p}\right).
\end{equation}
The function $\mathbf{r} \mapsto {\cal F}_{\mathbf{r}}/{\cal F}_{\mathbf{1}}$ is multiplicative in $\mathbf{r}$. We define the corresponding Euler factors 
\begin{equation}\label{f-euler}  {\cal F}_{\mathbf{1}}(p) = \left(1 - \frac{1}{p^2}\right) \sum_{\delta = 0}^{\infty} \frac{1}{p^{2\delta}}\sum_{\substack{\varphi+\gamma \leq \delta\\ \gamma \leq 1}} \frac{(-1)^{\gamma}}{p^{\gamma}}, \quad {\cal F}_{\mathbf{r}}(p) = {\cal F}_{\mathbf{1}}(p) \frac{{\cal F}_{\mathbf{r}(p)}}{{\cal F}_{\mathbf{1}}},
\end{equation}
so that ${\cal F}_{\mathbf{r}} = \prod_{p} {\cal F}_{\mathbf{r}}(p).$ Similarly, for the quantity ${\cal E}_{\mathbf{r}}$ defined in \eqref{defer}, we define its Euler factors
 $${\cal E}_{\mathbf{1}}(p) = \sum_{k=0}^{\infty} \frac{\phi(p^k)}{p^{3k}}, \quad {\cal E}_{\mathbf{r}}(p) = {\cal E}_{\mathbf{1}}(p) \frac{{\cal E}_{\mathbf{r}(p)}}{{\cal E}_{\mathbf{1}}}.$$
  With this notation, we have the following. 
%We write $$\alpha = v_p(r_1), \quad \beta = v_p(r_2), \quad \gamma = v_p(r_3).$$
%The expression $ {\cal E}_\mathbf{r}$ factors into an Euler product
%$$ {\cal E}_\mathbf{r} = \prod_p E_{\alpha, \beta, \gamma}(p) =  \prod_p\left( 1+\sum_{k=1}^{\infty} %\frac{p^{k-1}(p-1) p^{\min(k, \alpha) + \min(k, \beta) + \min(k, \gamma)}}{p^{3k}}\right).$$
 
\begin{lem}\label{scaling} 
{\rm a)} Let $d\in\NN$ and ${\mathbf r}\in\NN^3$. Then 
  ${\cal E}_{d\mathbf r} = d {\cal E}_{\mathbf r}$ and ${\cal E}_{\mathbf r} \ll (r_1; r_2;r_3)(r_1r_2r_3)^{\varepsilon}.$ \\ 
{\rm b)} Let $p$ be a prime and $\mathbf{r} = (p^{\alpha}, p^{\beta}, 1)$ with $0 \leq \beta \leq \alpha$. Then ${\cal F}_{\mathbf{r}}(p) = {\cal E}_{\mathbf{r}}(p).$
\end{lem}

\begin{proof} Put $\alpha = v_p(r_1)$, $\beta = v_p(r_2)$ and $ \gamma = v_p(r_3)$. By symmetry, we may suppose that $\alpha\ge \beta\ge \gamma$, and then find that
 \begin{equation}\label{eulerE}
\begin{split}
{\cal E}_{\mathbf r}(p) &= 1+\sum_{k=1}^{\infty} \frac{p^{k-1}(p-1) p^{\min(k, \alpha) + \min(k, \beta) + \min(k, \gamma)}}{p^{3k}}\\
& =  \frac{p^{\gamma - \alpha -1}(p^{\alpha}(p+1) (1+\gamma - \beta + p(1-\gamma+ \beta)) - p^{\beta +1})}{p+1}.  
\end{split}
\end{equation}
This formula shows on the one hand $|{\cal E}_{\mathbf r}(p)| \leq p^{\gamma}(2+\beta)$, on the other hand we see ${\cal E}_{\mathbf r}(p) = d^{-1} {\cal E}_{d\mathbf r}(p)$ for $d = p^{\delta}$ a power of $p$. 
Part (a)  follows. 

For (b), we note that \begin{displaymath}
\begin{split}
 {\cal F}_{(p^\alpha, p^{\beta}, 1)}  &=  \left(1-\frac{1}{p^2}\right) \sum_{d=0}^{\infty} \frac{p^{\min(d, \beta) + \min(d, \alpha)}}{p^{2d}} \sum_{\substack{f + g \leq \max(0, d-\alpha)\\ g \leq 1}} \frac{(-1)^{g}}{p^g} \sum_{k \leq \min(1, \max(0, \beta-d))} \frac{(-1)^k}{1+p^k}\\
  & =  \left(1-\frac{1}{p^2}\right)\left( \sum_{d=0}^{\beta-1}   \sum_{k =0}^1 \frac{(-1)^k}{1+p^k} + \sum_{d = \beta}^{\alpha-1} \frac{1}{p^{d-\beta}}    + \sum_{d = \alpha}^{\infty}  \frac{1}{p^{2d-\beta-\alpha}} \sum_{\substack{f+g \leq d-\alpha\\ g \leq 1}}  \frac{(-1)^g}{p^g}\right)\\
   &=  \left(1-\frac{1}{p^2}\right)\left(\frac{\beta p}{p+1} + \frac{p^{1-\alpha}(p^{\alpha} - p^{\beta} ) }{p - 1}+ \frac{p^{1-\alpha+\beta}(1  +p+p^2)}{(p-1)(p+1)^2} \right)   = 1 + \beta + \frac{1-\beta}{p} -\frac{p^{\beta-\alpha}}{1+p},
     \end{split}
\end{displaymath}
 which coincides with \eqref{eulerE} if $\gamma = 0$. \end{proof}

 Our final   lemma in this section investigates a multiple sum of multiplicative functions that comes up in the computation of the main  term. We recall the definitions  \eqref{defalpha} and \eqref{EPC}.  %As before, we use the convention that for fixed $k \in \{1, 2, 3\}$, the numbers $i, j$ are chosen such that $\{i, j, k\} = \{1, 2, 3\}$.
 \begin{lem}\label{supereuler}% {\rm a)} 
 In the range $T \geq 1$ one has 
 %The absolutely convergent sum
\begin{equation}\label{eulerproduct}
 \sum_{ |\mathbf{b}|, |\mathbf{c}|, |\mathbf{f}|, |\mathbf{g}|, h \leq T } \hspace{-0.5cm} \mu((\mathbf{b}, \mathbf{c}, \mathbf{f}, \mathbf{g}, h)) \sum_{q \in \Bbb{N} }      \frac{\phi(q)  }{q^3} \prod_{k=1}^3 \frac{(q; \alpha_{1k} \alpha_{2k})}{\alpha_{1k} \alpha_{2k} \alpha_{3k}} = 
 C + O(T^{\eps-1}). 
\end{equation} 
 %, as well as the notation \eqref{muvector}.  
%{\rm b)} For any $ \alpha < 1$ % and $\eta$ sufficiently small, 
%the sum
 %\begin{equation}\label{multiple}
%\begin{split}
%  & \sum_{ \mathbf{b}, \mathbf{c}, \mathbf{f}, \mathbf{g} \in \Bbb{N}^3}  \sum_{h, q \in \Bbb{N} }      \frac{|  \mu(\mathbf{b})\mu(\mathbf{c})\mu(\mathbf{f})\mu(\mathbf{g})\mu(h)|}{q^2 h^{3} \prod_{j=1}^3 b_j^2c_j^2g_j^3 f_j^2%)^{1-\eta} 
%   } \left(h^{\alpha} + \sum_{j=1}^3(b_j^{\alpha} + c_j^{\alpha} + f_j^{\alpha} + g_j^{\alpha})\right)\\
% & \times  \prod_{k=1}^3 \Bigl(q(h(g_i; g_j); g_ig_j)(f_k(b_i; b_j); b_ib_j);  g_ig_jh b_ib_jf_k\Bigr)%^{1+\eta}    \
% \big(f_kg_k(c_i; c_j); c_ic_jg_k; c_ic_jf_k\big)%^{1+\eta} 
 % \end{split}
%\end{equation}
%is absolutely convergent. 
 \end{lem} 

\begin{proof} 
The product $C$ in \eqref{EPC} equals the 
completed sum on the left of \eqref{eulerproduct}, with all $b_j,c_j,f_j,g_j$ and $h$ running over all natural numbers. We first establish this claim. 

The completed sum
can be written as an Euler product where the Euler $p$-factor is given (formally) by the same sum, but with 
all variables of summation running over powers of $p$. The main observation is that there is no contribution from terms where $p^2 \mid q$. Indeed, for squarefree variables $\mathbf{b}$,   $\mathbf{f}$, $\mathbf{g}$, $h$, the numbers $\alpha_{1k}$ and $\alpha_{2k}$ are squarefree, and hence, $\alpha_{1k} \alpha_{2k}$ is cubefree. Then, whenever 
$p^2 \mid q$, we see that
$$v_p\left(\frac{(q; \alpha_{1k} \alpha_{2k})}{\alpha_{1k} \alpha_{2k} \alpha_{3k}}\right) = v_p\left(\frac{1}{\alpha_{3k}}\right)   $$
is independent of $v_p(h)$, and consequently,  the contribution from $h=1$ and $h=p$ cancel out.  
Hence, we may introduce the multiplicative factor $\mu(q)^2$ in the expression defining the completed sum.
After this simplification, a mundane computation shows that the $p$-th Euler factor of this sum coincides with that
of the product \eqref{EPC}, as we have claimed.

It remains to estimate the error term introduced by completing the sum on the left. To this end we use Rankin's trick and  bound  the characteristic function on $x \geq T$ by $(x/T)^{\xi}$, for some $0 < \xi < 1$.  Thus it suffices to show that 
\begin{equation}\label{multiple}
\begin{split}
  &  \sum_{ |\mathbf{b}|, |\mathbf{c}|, |\mathbf{f}|, |\mathbf{g}|, h  } \hspace{-0.2cm} |\mu((\mathbf{b}, \mathbf{c}, \mathbf{f}, \mathbf{g}, h))| \sum_{q \in \Bbb{N} }      \frac{\phi(q)  }{q^3} \prod_{k=1}^3 \frac{(q; \alpha_{1k} \alpha_{2k})}{\alpha_{1k} \alpha_{2k} \alpha_{3k}} \Bigl(h^{\xi} + \sum_{j=1}^3(b_j^{\xi} + c_j^{\xi} + f_j^{\xi} + g_j^{\xi})\Bigr) 
   \end{split}
\end{equation}
is absolutely convergent. To see this, first note that the rightmost factor in the preceding display is a
sum of 13 summands, and it is then sufficient to show absolute convergence with only one of these
summands present. Irrespective of which summand is present, we are reduced to multiple sum of
multiplicative terms that we may formally rewrite as an Euler product. As before, its $p$-th
Euler factor arises from letting all variables run through powers of $p$. Again as before, since $\alpha_{1k} \alpha_{2k}$ is cubefree, it is clear that  terms affecting convergence  in the Euler $p$-factor come from $q \mid p^2$. 
Another mundane computation then shows that the $p$-th Euler factor under consideration is of the form
  $1 + O(p^{\xi-2})$.  
We take $\xi=1-\eps$ to ensure absolute convergence of the Euler product. This completes the proof. \end{proof}

\subsection{Mellin inversion formulae}

Our first lemma in this section expresses the Fourier integral \eqref{defI} as a Mellin integral. This features the meromorphic function
\begin{equation}\label{defK}
  K(s) = \frac{\Gamma(s) \cos(\pi s/2) (1-2^{s-1})^2}{(1-s)^2}.
\end{equation}

\begin{lem}\label{mellin} Let $\mathbf{r} \in \Bbb{N}^3$ and $X_1, X_2, X_3, Y_1, Y_2, Y_3 \geq 1$. Then, whenever
the positive numbers $c_1$, $c_2$ satisfy $c_1+c_2<1$, one has 
   %define 
%\begin{displaymath}
%  {\cal I}_\mathbf{r}(\mathbf{X}, \mathbf{Y}) =   \int_{\Bbb{R}} \int_{\substack{\frac{1}{2}X_j \leq |x_j| \leq X_j\\ \frac{1}{2}Y_j \leq |y_j| \leq Y_j}}% \int_{\substack{X_2 \leq |x_2| \leq 2X_2\\ Y_2 \leq |y_2| \leq 2Y_2}}\int_{\substack{X_3 \leq |x_3| \leq 2X_3\\ Y_3 \leq |y_3| \leq 2Y_3}} 
 % e(\alpha( r_1x_1y_1 + r_2x_2y_2 + r_3x_3y_3) )  \dd\mathbf{x} \, d \mathbf{y} \, \dd\alpha.
%\end{displaymath}
%Then
\begin{displaymath}
  {\cal I}_\textbf{r}(\mathbf{X}, \mathbf{Y}) = \frac{64}{\pi } \int_{(c_2)} \int_{(c_1)} \frac{(X_1Y_1)^{1-s} (X_2Y_2)^{1-t} (X_3Y_3)^{s+t}}{r_1^{s} r_2^tr_3^{1-s-t}} K(s) K(t) K(1-s-t) \frac{\mathrm d s \,\mathrm d t}{(2\pi \ii)^2}.
\end{displaymath}
\end{lem}

In particular, choosing $c_1 = c_2 = 1/3$, we see that
\begin{equation}\label{bound-I}
 \mathcal{I}_\textbf{r}(\textbf{X}, \textbf{Y}) \ll \frac{(X_1X_2X_3Y_1Y_2Y_3)^{2/3}}{(r_1r_2r_3)^{1/3}}.
 \end{equation}

\begin{proof} Let $B(X, Y)$ denote the region $\frac{1}{2} X \leq |x| \leq X$, $\frac{1}{2} Y \leq |y| \leq Y$. For $\alpha, r \in \Bbb{R}$, one has
\begin{equation}\label{sinint}
  \int_{B(X, Y)} e(\alpha rxy  ) \dd(x, y)  = \frac{2({\rm Si}(\frac{1}{2}\pi \alpha r XY) - 2 {\rm Si}(\pi \alpha rXY) + {\rm Si}(2\pi \alpha r X Y))}{\pi \alpha r},
\end{equation}
where
$${\rm Si}(x) = \int_0^x \frac{\sin(t)}{t}\dd t$$
is the integral sine.  By   \cite[6.246.1, 8.230.1]{GR}, the identity
\begin{equation}\label{sinint1}
  \int_0^{\infty}  \frac{2({\rm Si}(x/4) - 2 {\rm Si}(x/2) + {\rm Si}(x))}{x/2} x^{s-1}\dd x= 
4 \int_{( c)} K(s) \frac{\mathrm d s}{2\pi \ii}
\end{equation}
holds for  $0 < c< 1$, and hence, for the same $c$, Mellin inversion yields
\begin{displaymath}
  \int_{B(X, Y)} e(\alpha rxy  ) \dd(x, y) = 4 \int_{( c)} K(s)(2\pi r|\alpha|)^{-s} (XY)^{1-s} \frac{\mathrm d s}{2\pi \ii}.
\end{displaymath}
 We use this formula twice for the integration over $x_1, y_1, x_2, y_2$ with contours $(c_1)$, $(c_2)$ such that $c_1, c_2 > 0$, $c_1 + c_2 < 1$. Then we integrate over $x_3, y_3$ using \eqref{sinint} and finally integrate over $\alpha$ by \eqref{sinint1}. This gives the desired formula.\end{proof}

%\begin{lem}\label{inv} Let $f  (0,\infty) \rightarrow \Bbb{C}$ be a smooth, compactly supported function. Let $g(x ) = \sum_{k  = 0}^{\infty} f(2^{k}x).$ 
%Then one has 
%$$g'(x) = \frac{-1}{2\pi i} \int_{(1)}  \frac{2^{s}s}{2^s-1} \widehat{f}(s) x^{-s}ds.$$
%\end{lem}

%\mathbf{Proof.} A simple computation shows
%$$\int_0^{\infty} g'(x) x^{s-1} dx =  -\frac{2^{s-1}(s-1)}{2^{s-1}-1}\widehat{f}(s-1)$$
%and the result follows by Mellin inversion and a change of variables.  

The following lemma   computes explicitly a certain multiple Mellin integral whose integrand is a rational function. 

\begin{lem}\label{supermellin1} Fix $\mathbf{z} = (z_1, z_2) \in \Bbb{C}^2$ with $\Re z_1 = \Re z_2 = 1/3$, and for $\mathbf{y} = (y_6, y_7, y_8, y_9) \in \Bbb{C}^4$ let 
\begin{displaymath}
\begin{split}
F_{\mathbf{z}}(\mathbf{y}) =  & y_6 y_7 \left(\frac{1-z_1}{2}  - y_6 - y_7\right)    2  y_8 2 y_9  \left(\frac{1-z_2}{2} - 2y_8 - 2y_9\right) \left(1 - z_1-z_2 - y_6 - 2y_8 \right) \\
&\times \left(z_2- y_7 - 2  y_9 \right) \left(\frac{3z_1 + z_2 - 2}{2} + y_6 + y_7 +2y_8 + 2y_9 \right). 
\end{split}
\end{displaymath}
Then
$$
 \frac{1}{(2\pi \ii)^4} \int_{(\f1{15})} \int_{(\f1{15})} \int_{(\f1{15})} \int_{(\f1{15})} (F_{\mathbf{z}}(\mathbf{y}))^{-1}\dd y_6 \dd y_7 \dd y_8 \dd y_9 = \frac{2}{(1-z_1)(1-z_2)z_1z_2(z_1+z_2)(1-z_1-z_2)}.$$
\end{lem}

\begin{proof} This can be obtained by straightforward contour shifts. We shift all contours successively to the far left (the opposite direction would be possible, too). Each time we pick up two poles, and
the remaining integral vanishes in the limit. 
If we first compute the innermost integral over $y_6$ in this way and then divide by $2\pi \ii$, we arrive at
  \begin{displaymath}
 \begin{split}
  & \frac{4(4y_9 - z_2 + 1)}{(2y_7 + z_1 - 1)(2y_8 + z_1 + z_2 - 1)(2y_7 + 4y_9 + z_1 - z_2)(4y_8 + 4y_9 + 2z_1 + z_2 - 1)}\\
 & \times \left[y_7  2  y_8 2 y_9  \left(\frac{1-z_2}{2} - 2y_8 - 2y_9\right)  \left(z_2- y_7 - 2  y_9 \right)\right]^{-1}.\\
\end{split}
\end{displaymath}
We integrate this over $y_7$ and again divide by $2\pi \ii$, then obtaining 
\begin{displaymath}
\begin{split}
   &\frac{4y _9 - z_2 + 1}{(2y _8 + z_1 + z_2 - 1)(4y _8 + 4y _9 + 2z_1 + z_2 - 1)y _8y _9 \left(\frac{1-z_2}{2} - 2y _8 - 2y _9\right)}\\
  & \times \frac{1+z_2}{(1-z_1)(z_1+z_2)(4y _9 - z_2 + 1)(2y _9 - z_2)}.
\end{split}
\end{displaymath}
Again, integrating this over $y_8$ and dividing by $2\pi \ii$, one gets the function
\begin{displaymath}
%\begin{split}
   \frac{1+z_2}{(1-z_1)(z_1+z_2) y _9(2y _9 - z_2)} \cdot \frac{2(z_2-1)}{z_1(1-z_1-z_2)(4y _9 + z_2 - 1)(z_2 - 1 - 4y _9)}.
%\end{split}
\end{displaymath}
Finally, 
\begin{displaymath}
%\begin{split}
  \int(\ldots) \frac{\dd y_9}{2\pi \ii}  = \frac{(1+z_2)2(z_2-1)}{(1-z_1)(z_1+z_2)z_1(1-z_1-z_2)}  \cdot \frac{-1}{(z_2-1)^2(1+z_2)z_2}, %= \frac{2}{(1-z_1)(1-z_2)z_1z_2(z_1+z_2)(1-z_1-z_2)}
 % \end{split}
\end{displaymath}
and the claim follows. \end{proof}

The double Mellin integral in the next lemma is related to the archimedean density of our algebraic variety, cf.\  Lemma \ref{lem31}. 

\begin{lem}\label{supermellin2} We have
\begin{displaymath}
\begin{split}
  &  \int_{(\frac{1}{3})} \int_{(\frac{1}{3})} \frac{ \Gamma(z_1)\Gamma(z_2)\Gamma(1-z_1-z_2) \cos(\frac{\pi z_1}{2})\cos(\frac{\pi z_2}{2})\cos(\frac{\pi(1-z_1-z_2)}{2})}{(1-z_1)(1-z_2)z_1z_2(z_1+z_2)(1-z_1-z_2)} \frac{\mathrm d z_1\, \mathrm d z_2}{(2\pi \ii)^2} = \frac{\pi}{8}(\pi^2 - 3 + 24\log 2). 
\end{split}
\end{displaymath}
\end{lem}

\begin{proof} We call the left hand side ${\cal I}$. First we note the Mellin formula
\begin{displaymath}
  \int_0^{\infty}\Bigl( \int_y^{\infty} \frac{\sin t}{t^2} \dd t\Bigr) y^{u-1} \dd y = \int_0^{\infty} \Bigl(\int_y^{\infty} \frac{\cos t}{t} \dd t  + \frac{\sin y}{y}\Bigr) y^{u-1} \dd y =  \frac{\Gamma(u) \cos(\pi u/2)}{u(1-u)} 
\end{displaymath}
that holds for $0 < \Re u < 1$  (\cite[6.246.2]{GR} and \cite[3.761.4]{GR}). Hence, by Mellin inversion, we have
 \begin{displaymath} 
 {\cal I}  = \int_0^{\infty} \Bigl(\int_y^{\infty} \frac{\sin t}{t^2} \dd t\Bigr)^3 \dd y = \int_0^{\infty}  \Bigl(\int_1^{\infty} \frac{\sin y t}{t^2} \dd t\Bigr)^3 \frac{\mathrm d y}{y^3}.
\end{displaymath}
We apply the Fubini-Tonelli theorem and see that we may exchange the order of integrations. This yields the formula
$${\cal I}  =  \int_{[1, \infty]^3} \frac{T(\mathbf t)}{t_1^2t_2^2t_3^2}\dd \mathbf{t} $$
where
$$T({\mathbf t})  = \int_0^\infty\frac{\sin yt_1 \sin yt_2 \sin yt_3}{y^3} \dd y.  $$
Let $\text{sgn}\,\beta$ denote the sign of the real number $\beta$.
Then, on
writing the sin-function in terms of exponentials, a standard application of the residue theorem shows that 
 \begin{displaymath}
 \begin{split}
 T({\mathbf t}) =   \frac{\pi}{16}& \left((t_1+t_2+t_3)^2 - (t_1+t_2-t_3)^2\text{sgn}(t_1+t_2-t_3)\right. \\
   &\left.-  (t_1-t_2+t_3)^2\text{sgn}(t_1-t_2+t_3) -  (-t_1+t_2+t_3)^2\text{sgn}(-t_1+t_2+t_3)\right). 
   \end{split}
 \end{displaymath} 
If $t_1>0$, $t_2>0$ and $|t_1-t_2| \geq 1$, a straightforward computation shows (split at $|t_1-t_2|$ and $t_1+t_2$)
   \begin{displaymath}
 \begin{split}
 \int_1^{\infty} \frac{T(\mathbf t)}{t_1^2t_2^2t_3^2}  \dd t_3& = \frac{\pi}{16t^2_1t^2_2}\left(8 \min(t_1, t_2)\log|t_1-t_2| - 8\frac{t_1t_2}{t_1+t_2} + 4(t_1+t_2) \log\frac{t_1+t_2}{|t_1-t_2|} + 8\frac{t_1t_2}{t_1+t_2}\right)\\
 & = \pi\frac{(t_1+t_2)\log(t_1+t_2) - |t_1-t_2| \log|t_1-t_2|}{4t_1^2t_2^2},
 \end{split}
 \end{displaymath} 
while for $t_1>0$, $t_2>0$ and $|t_1 - t_2| < 1$, a slightly simpler computation gives
  \begin{displaymath}
 \begin{split}
 \int_1^{\infty} \frac{T(\mathbf t)}{t_1^2t_2^2t_3^2}  \dd t_3& = \frac{\pi}{16t^2_1t^2_2}\left(4(t_1+t_2)\log(t_1+t_2) - \frac{2(t_1+t_2 - 1)(t_1+t_2 +(t_1-t_2)^2)}{t_1+t_2}   + 8\frac{t_1t_2}{t_1+t_2}\right)\\
 & = \frac{\pi(1 - (t_1-t_2)^2 + 2(t_1+t_2)\log(t_1+t_2))}{8t_1^2t_2^2}. 
 \end{split}
 \end{displaymath} 
Let
$$ {\mathcal T}_1 = \{(t_1,t_2)\in(0,\infty)^2: |t_1-t_2|\ge 1\}, \quad  {\mathcal T}_2 = 
\{(t_1,t_2)\in(0,\infty)^2: |t_1-t_2|< 1\}.$$
We then have a natural decomposition ${\cal I} = {\cal I}_1 + {\cal I}_2$, where
\begin{displaymath}
\begin{split}
  {\cal I}_1 & = \int_{{\mathcal T}_1} \pi\frac{(t_1+t_2)\log(t_1+t_2) - |t_1-t_2| \log|t_1-t_2|}{4t_1^2t_2^2} \dd\mathbf t,\\
{\cal I}_2 & =  \int_{{\mathcal T}_2} \frac{\pi(1 - (t_1-t_2)^2 + 2(t_1+t_2)\log(t_1+t_2))}{8t_1^2t_2^2}  \dd\mathbf t. 
\end{split}
\end{displaymath}
Obvious substitutions deliver that
$$ 
 {\cal I}_1
 =  2\pi \int_1^{\infty} \int_1^{\infty} \frac{(r+2t)\log(r+2t) - r \log r}{4(r+t)^2 t^2} \dd t\dd r = 2 \pi \int_0^1 \int_1^{\infty} \frac{(1+2rt)\log(1+2 rt) -2rt\log r}{4t^2 r(1+rt)^2} \dd t\dd r, $$
and a similar computation produces 
%The $t$-integral can be computed elementarily and equals
% \begin{displaymath}
%\frac{  r^2(\log(r+2) - \log r) + 2(r+1)(2\log(r+2) + \log r \log(r+1) - \log 4 - \log(r+1)) - 2 r \log r}{4r^2(1+r)}
 %  \frac{ (  2 (1 + r) \log(1 + r) - r^2)\log r - 
% 2 \log (4(1+r))  - 2r \log(4 r (1 + r))  + (2 + r)^2 \log(
 %  2 + r)}{4 r^2 (1 + r)} 
% \end{displaymath}
$$
  {\cal I}_2  =  2\pi \int_{0}^{1} \int_1^{\infty} \frac{ 1 - r^2 + 2(r+2t)\log(r+2t)}{8(r+t)^2t^2}   \dd t\dd r.
$$
%The $t$-integral equals
%$$\frac{2r(2+r)^2 \log(2+r) + 2(1+r)(r^2 - 2r - 1)\log(1+r) + (1+r)r(2 - 8\log 2 - r - r^2)}{8r^3(1+r)}$$
The  $t$-integrals in the final expression for ${\mathcal I}_j$ have an elementary primitive, and a tedious computation yields
$${\cal I} = \frac{\pi}{4} \int_0^1 f(r) \dd r$$
where $f$ is defined by
\begin{equation*}%\label{r-int}
\begin{split}
-8r^3(1+r) f(r) =  &  8 r \log 2 + r (1 + r) (-2 + r + r^2 + 8 r \log 2) + 
 4 r^3 (1 + r) \log(1 + 1/r) \log r \\
 & + 4 r^4 \log(4 r) + 2 (1 + r (3 + r + r^2 + 2 r^3)) \log(1 + r) - 
 2 r (2 + r)^2 \log(2 + r) \\
 & - 2 r^2 (1 + 2 r)^2 \log(1 + 2 r) .
 \end{split}
\end{equation*}
Now let
$$\Li(z) = -\int_0^z \frac{\log(1-t)}{t} \dd t = \sum_{n=1}^{\infty} \frac{z^n}{n^2}$$%\sum_{n=1}^{\infty} \frac{z^n}{n^2}$$
be the Dilogarithm, see \cite{Le} for basic and more advanced properties of this function.
By brute force one then checks that a primitive of $f$ is given by $F$, where 
\begin{displaymath}
\begin{split}
8F(r)= &-\frac{1}{r} - r + \frac{8 \log 2}{r} -8 r \log 2 + 4 \log\left(1 + \frac{1}{r}\right) + 4r\log\left(1 + \frac{1}{r}\right) + 4 \log r - 4r \log r\\
&- 2 \log r \log\left(1 + \frac{1}{r}\right) - 4 r\log r \log\left(1 + \frac{1}{r}\right) - 2(\log r)^2- \log(1+r) + \frac{\log(1+r)}{r^2}\\
&+ \frac{4\log(1+r)}{r} - 4r \log(1+r)+2\log r \log(1+r) - 4\log(2+r) - \frac{8\log(2+r)}{r} \\
& + 4\log(1+2r) - 2\log 2 \log(1+2r)+ 8r\log(1+2r) - 2\log(1+r)\log(1+2r)\\
&  - 2\,\Li(-1-r) - 2\, \Li(-r) - 2\, \Li(-2r) - 2\, \Li(-1-2r).
%& + \int \frac{4\log(1+r)}{1+2r} dr. 
 %  \frac{-r (1 + r^2 + 8 ( r^2-1) \log 2) + \log(1+r) +  4r (\log(1 + r) - 2 \log(2 + r))- r^2  ( 4 r-4 + \log r)\log r}{8r^2}\\
  % \frac{-r^2(2 \log(1 + 1/r) (-2 - 2 r + \log r + 2 r \log r) + \log(1 + r) -2 (-2 r + \log(-1 - 2 r) + \log r) \log(1 + r) + 4 \log (2 + r))}{8r^2}
\end{split}
\end{displaymath}

  %It satisfies $Li_2'(z) = - \log(1-r)/r$ and can be continued analytically to $\Re z < 1$. 
In order to evaluate $F(1) - F(0)$, we need to evaluate   $\Li(-3) + 2\, \Li(-2) - \Li(-1)$. It is well-known %\cite[p.4]{Le} 
that $\Li(-1) = -\pi^2/12$. Moreover, one confirms by differentiation that the function 
$$x \mapsto 2\,\Li(x) + \Li(1-x^2) - 2\,\Li(-1/x) +2 \log(1 - x^2) \log x - (\log x)^2$$
is constant, and it takes  the value   $\pi^2/6$, as can be seen by substituting $x=1$.  For $x = -2$ we use   $\Li(-1/2) = \frac{1}{12}\pi^2  - \frac{1}{2} (\log 2)^2$ \cite[(1.16)]{Le} to conclude that
$$2\,\Li(-2) + \Li(-3) =  - \frac{\pi^2}{3} - 2\log 2 \log 3. $$
Altogether, this gives 
% $${\cal I} = \pi(1 + \log 2) + \frac{\pi^3}{8}.$$
 ${\cal I} =  \frac{\pi}{4}(F(1) - F(0)) = \frac{\pi}{8}(\pi^2 - 3 + 24\log 2)$ 
as required. \end{proof}

 \section{An asymptotic formula}\label{heart}
 
This section is devoted to a proof of Proposition \ref{asymp}. By symmetry we can assume that
\begin{equation}\label{sym}
  r_1X_1Y_1 \leq  r_2X_2Y_2 \leq r_3X_3 Y_3, \quad X_2 \leq Y_2. 
\end{equation}
To begin with, we make the two  additional assumptions 
\begin{equation}\label{assump}
(r_1; r_2; r_3) = 1, \quad r_2X_2X_3 \asymp r_3X_3Y_3 = Z,
\end{equation}
say.

We start our argument by smoothing the summation conditions: let $P_1, P_2, P_3,  Q_1, Q_2, Q_3$ satisfy  $1/10 \leq P_j \leq X_j/10$, $1/10 \leq Q_j \leq Y_j/10$, and for $1 \leq j \leq 3$ let $v_j$ be a function of type $(X_j, P_j)$ and $w_j$ be a function of type $(Y_j, Q_j)$, cf.\ Section \ref{smoothing}.  %We define smooth even weight functions $v_j$ satisfying $v_j(x) = 1$ for $|x| \in [\frac{1}{2}X_j-P_j,  X_j+P_j]$, $v_j(x) = 0$ for $|x| \not\in (\frac{1}{2}X_j-2P_j, X_j+2P_j)$ and $v_j^{(i)}(x) \ll_i P_j^{-i}$ for $i \in \Bbb{N}_0$ and all $x$.  Similarly we define smooth even weight functions $w_j$ satisfying $w_j(y) = 1$ for $|y| \in [\frac{1}{2}Y_j-Q_j,  Y_j+Q_j]$, $w_j(y) = 0$ for $y \not\in (\frac{1}{2}Y_j-2Q_j, Y_j+2Q_j)$ and $w_j^{(i)}(y) \ll_i Q_j^{-i}$ for $i \in \Bbb{N}_0$ and all $y$. 
Let  
\begin{displaymath}
  N^{(1)}_{\mathbf{r}}(\mathbf{X}, \mathbf{Y}) = \left.  \sum \right. ' %_{r_1x_1y_1 +r_2 x_2y_2 + r_3x_3y_3=0}  
  v_1(x_1)v_2(x_2)v_3(x_3) w_1(y_1)w_2(y_2)w_3(y_3),
\end{displaymath}
where the prime indicates summation over $(\mathbf{x}, \mathbf{y}) \in \Bbb{Z}_0^6$ satisfying \eqref{T1}. We   choose 
$$P_j = \frac{X_j}{\Xi^{\delta}} , \quad Q_j = \frac{Y_j}{\Xi^{\delta}}, \quad \Xi = \min(X_1, X_2, X_3, Y_1, Y_2, Y_3) $$
for some $0 < \delta < 1$ to be specified later. A simple divisor argument shows
\begin{equation}\label{e1}
  N_{\mathbf{r}}(\mathbf{X}, \mathbf{Y}) =  N^{(1)}_{\mathbf{r}}(\mathbf{X}, \mathbf{Y}) + O\left( X_1Y_1Z^{1+\varepsilon} \Xi^{-\delta}%\Bigl(\frac{P_1}{X_1} + \frac{Q_1}{Y_1} +  \frac{P_2}{X_2} +   \frac{Q_2}{Y_2} + \frac{P_3}{X_3} + \frac{Q_3 }{ Y_3}\Bigr)
   \right). 
\end{equation}
Since $(r_1; r_2; r_3) = 1$, we may write
\begin{equation}\label{6to4}
N_{\mathbf{r}}(\mathbf{X}, \mathbf{Y})  = \sum_{(r_2; r_3) \mid x_1y_1}  v_1(x_1)w_1(y_1) M\left(\frac{r_2}{(r_2, r_3)}, \frac{r_3}{(r_2, r_3)}, \frac{r_1x_1y_1}{(r_2, r_3)}\right)
\end{equation}
where
\begin{displaymath}
  M(r'_2, r'_3, h) = \sum_{h + r'_2x_2y_2 + r'_3x_3y_3 = 0}  v_2(x_2)v_3(x_3)  w_2(y_2)w_3(y_3). 
\end{displaymath}
We now manipulate $M(r'_2, r'_3, h)$ for $h \not= 0$, $(r'_2; r'_3) = 1$. %Let us assume without loss of generality 
%\begin{equation}\label{wlog}
 % X_2 \leq Y_2,
%\end{equation}
%otherwise we simply exchange the roles of $x_2$ and $y_2$ in the following argument. 
We have
\begin{displaymath}
\begin{split}
 M(r'_2, r'_3, h)& = \sum_{x_2} v_2(x_2)  \sum_{r_3x_3y_3 \equiv -h \, (\text{mod }r'_2|x_2|)}   v_3(x_3) w_2\Bigl( \frac{-r'_3x_3y_3-h}{r'_2x_2} \Bigr)w_3(y_3)\\
&  = \sum_{x_2} v_2(x_2)  \sum_{\substack{\xi, \eta \, (\text{mod }r'_2|x_2|)\\ r'_3\xi\eta \equiv -h \, (\text{mod }r'_2|x_2|)}} \sum_{\substack{ x_3 \equiv \xi \, (\text{mod } r'_2|x_2|) \\ y_3 \equiv \eta \, (\text{mod } r'_2|x_2|)} } W_{r'_2, r'_3, h}(x_3, y_3; x_2),
  \end{split}
\end{displaymath} 
where
\begin{displaymath}
 W_{r'_2, r'_3, h}(x_3, y_3;x_2) =   v_3(x_3) w_2\Bigl( \frac{-r'_3x_3y_3-h}{r'_2x_2} \Bigr)w_3(y_3). 
\end{displaymath}
By the Poisson summation formula we obtain
\begin{equation}\label{poisson}
  M(r'_2, r'_3, h)=  \sum_{x_2} \frac{v_2(x_2)}{(r'_2x_2)^2 }    \sum_{h_1, h_2 \in \Bbb{Z}} {\cal W}_{r'_2, r'_3, h}\left(\frac{h_1}{r'_2|x_2|}, \frac{h_2}{r'_2|x_2|}; x_2\right) S_{ r'_3, h}(h_1, h_2; r_2'|x_2|),
\end{equation} 
where the exponential sum $S_{  r'_3, h}(h_1, h_2;r_2'|x_2|) $ was defined in \eqref{Kloo},  and where  
%\begin{displaymath}
% S_{r'_2, r'_3, h}(h_1, h_2;x_2) = \sum_{\substack{\xi, \eta \, (r'_2|x_2|)\\ r'_3\xi\eta \equiv -h \, (r'_2|x_2|)}} e\left(\frac{h_1\xi + h_2\eta}{r'_2x_2}\right)
%\end{displaymath}
%as in Lemma \ref{expsum} 
$${\cal W}_{r_2', r_3', h} (\xi, \eta; x_2) = \int_{-\infty}^{\infty}  \int_{-\infty}^{\infty}  W_{r_2', r_3', h}(x_3, y_3; x_2) e(-x_3\xi - y_3\eta)\dd x_3\dd y_3$$
is the Fourier transform  with respect to the first two variables. By partial integration it is easy to see that
$${\cal W}_{r'_2, r'_3, h}(\xi, \eta; x_2) \ll_{A, B} X_3Y_3 \left(\left(\frac{1}{P_3} + \frac{r'_3Y_3}{r'_2Q_2X_2}\right)\frac{1}{|\xi|}\right)^A\left( \left(\frac{1}{Q_3} + \frac{r'_3X_3}{r'_2Q_2X_2}\right)\frac{1}{|\eta|}\right)^B$$
holds for any $A, B \geq 0$ and $x_2 \asymp X_2$. By \eqref{11}, the contribution of the terms $h_1h_2 \not= 0$ to \eqref{poisson} is therefore at most
\begin{equation}\label{error1}
  \ll Z^{\varepsilon} (r'_2; h)^{1/2}  (r'_2)^{1/2}  X_3Y_3X_2^{3/2} \left(\frac{1}{P_3} + \frac{r'_3Y_3}{r'_2Q_2X_2}\right) \left(\frac{1}{Q_3} + \frac{r'_3X_3}{r'_2Q_2X_2}\right).
\end{equation}
Similarly, by \eqref{01} and \eqref{10} the contribution of the terms $h_1h_2 = 0$, $(h_1, h_2) \not= (0, 0)$ is at most
\begin{equation}\label{error2}
  \ll Z^{\varepsilon} \frac{(r'_2; h)}{r'_2}     X_3Y_3  \left(\frac{1}{P_3} +\frac{1}{Q_3}+ \frac{r'_3(X_3 +Y_3)}{r'_2Q_2X_2}\right) .
\end{equation}
Moreover, by \eqref{00}, the contribution of the central term %, $\tilde{M}(r_2, r_3, h)$ say, 
equals 
\begin{displaymath}
   %\tilde{M}(r_2, r_3, h)=  
   \sum_{(r'_3; x_2) \mid h} \frac{v_2(x_2)}{(r'_2x_2)^2}    {\cal W}_{r'_2, r'_3, h}\left(0, 0; x_2\right)   \sum_{d(r'_3; x_2) \mid (r'_2x_2; h)} d(r'_3; x_2)^2 \phi\left(\frac{r'_2x_2}{(r'_3; x_2)d}\right).
\end{displaymath} 
We substitute this back into \eqref{6to4} and sum the error terms \eqref{error1} and \eqref{error2} over $x_1\asymp X_1$ and $y_1\asymp Y_1$ (recall \eqref{sym}). We continue to write $r_2' = r_2/(r_2; r_3)$ and $r_3' = r_3/(r_2; r_3)$ and see in this way  that $ N^{(1)}_{\mathbf{r}}(\mathbf{X}, \mathbf{Y}) $ equals the expression
\begin{displaymath}
\begin{split}
 N^{(2)}_{\mathbf{r}}(\mathbf{X}, \mathbf{Y}) = \sum_{(r_2; r_3) \mid x_1y_1} &v_1(x_1)w_1(y_1)\sum_{(r'_3; x_2) \mid \frac{r_1x_1y_1}{(r_2; r_3)}} \frac{v_2(x_2)}{(r'_2x_2)^2}    \sum_{d(r'_3; x_2) \mid (r'_2x_2; \frac{r_1x_1y_1}{(r_2; r_3)})} d(r'_3; x_2)^2 \\
  & \times  \phi\left(\frac{r'_2x_2}{(r'_3; x_2)d}\right)  \int_{\Bbb{R}} \int_{\Bbb{R}} v_3(x_3) w_2\Bigl(\frac{-r_3x_3y_3-r_1x_1y_1}{r_2x_2}\Bigr)w_3(y_3) \dd x_3 \dd y_3, 
  \end{split}
  \end{displaymath}
up to an error that does not exceed
\begin{equation}\label{step1}
\begin{split}
\ll &Z^{\varepsilon} \frac{(r_2'; r_1)X_1Y_1X_3Y_3   }{r'_2}\left(\frac{1}{P_3} +\frac{1}{Q_3}+ \frac{r'_3(X_3 +Y_3)}{r'_2Q_2X_2}\right) \\
&\quad + Z^{\varepsilon} X_1Y_1X_3Y_3   (r_2'; r_1)^{1/2}  ( r'_2)^{1/2}   X_2^{3/2} \left(\frac{1}{P_3} + \frac{r'_3Y_3}{r'_2Q_2X_2}\right) \left(\frac{1}{Q_3} + \frac{r'_3X_3}{r'_2Q_2X_2}\right)  \\
 \ll &  X_1Y_1Z^{1+\varepsilon} (\Xi^{\delta-1} + \Xi^{2\delta - \frac{1}{2}}). 
\end{split}
\end{equation}
In the sum defining $ N^{(2)}_{\mathbf{r}}(\mathbf{X}, \mathbf{Y})$, we pull the $d$-sum outside and introduce a new variable $b = (x_2; r_3')$. Then the summation conditions for $x_1, y_1, x_2$ become
$$\frac{db}{(db; r_1)}(r_2; r_3) \mid x_1y_1, \quad (x_2; r_3') = b, \quad \frac{db}{(db; r_2')} \mid x_2.$$
Writing $x_2 = x_2' db/(d; r_2')$, the last two conditions are equivalent to $(x'_2d; r_3'/b)= 1$. Hence we can rewrite the main term as
 \begin{displaymath}
 \begin{split}
N^{(2)}_{\mathbf{r}}(\mathbf{X}, \mathbf{Y}) =  \sum_{b \mid r_3'}&  \sum_{(d; r'_3/b) = 1} \frac{(d; r_2')^2 }{(r_2')^2d} \sum_{\frac{db}{(db; r_1)}(r_2; r_3)  \mid x_1y_1} \sum_{(x_2; r_3'/b) = 1}   \phi\left(\frac{r_2' x_2}{(d; r_2') }\right) \frac{v_2(dbx_2/(d; r_2'))}{x_2^2}     \\
&\times v_1(x_1)w_1(y_1)   \int_{\Bbb{R}} \int_{\Bbb{R}} v_3(x_3) w_2\Bigl( \frac{-r_3x_3y_3-r_1x_1y_1}{r_2dbx_2/(d, r_2')}\Bigr)w_3(y_3) \dd x_3 \dd y_3 .
  \end{split}
 \end{displaymath} 
By M\"obius inversion this equals
   \begin{displaymath}
 \begin{split}
 \sum_{abc = r_3'}&\frac{\mu(a)}{a^2} \sum_{(d; r'_3/b) = 1} \frac{(d; r_2')^2 }{(r_2')^2d} \sum_{fgh = \frac{db(r_2; r_3) }{(db; r_1)}} \mu(g) \sum_{x_1, y_1}\sum_{x_2 }   \phi\left(\frac{r_2' ax_2}{(d; r_2')  }\right) \frac{v_2(dbax_2/(d; r_2'))}{x_2^2}     \\
&\times v_1(fgx_1)w_1(hgy_1)   \int_{\Bbb{R}} \int_{\Bbb{R}} v_3(x_3) w_2\Bigl( \frac{-r_3x_3y_3-r_1f g^2hx_1y_1}{r_2dbax_2/(d; r_2')} \Bigr)w_3(y_3) \dd x_3 \dd y_3.
  \end{split}
 \end{displaymath} 
 We execute the $x_2$-sum by Lemma \ref{eulerkor}. This introduces an error not exceeding 
 \begin{equation}\label{err1}
 \begin{split}
 &\ll Z^{\varepsilon}  \sum_{abc = r_3'}\frac{1}{a^2} \sum_{d} \frac{(d; r_2')^2 }{(r_2')^2d}\sum_{fgh = \frac{db(r_2; r_3) }{(db; r_1)}}  \frac{X_1Y_1}{fg^2h} X_3Y_3 \frac{r_2' ax_2}{(d; r_2')  }%\\
 %&\quad\quad\quad\quad\quad \times 
 \left(\frac{X_2 (d; r_2')}{dba}\right)^{\frac{1}{2}} \left(\frac{dba}{ (d; r_2')P_2} + \frac{dbaY_2}{(d; r_2')Q_2X_2}\right)\\
   & \ll Z^{\varepsilon} \frac{(r_3'; r_1)^{1/2}}{r_2'} X_1Y_1X_3Y_3X_2^{1/2}   \left(\frac{1}{  P_2} + \frac{ Y_2}{ Q_2X_2}\right) \ll X_1Y_1Z^{1+\varepsilon} \Xi^{\delta - \frac{1}{2}}. 
   \end{split}
 \end{equation}
 Next we execute the $x_1$-sum by Poisson summation and keep only the central term. This introduces an error no larger than  
  \begin{equation}\label{err2}
 \begin{split}
 & \ll  \sum_{abc = r_3'}\frac{1}{a^2} \sum_{d} \frac{(d; r_2')^2 }{(r_2')^2d}\sum_{fgh = \frac{db(r_2; r_3) }{(db; r_1)}}  \frac{r_2'a}{(d; r_2')} \frac{ Y_1}{ gh} X_3Y_3   \left(\frac{X_1}{P_1} + \frac{r_1Y_1X_1}{r_2X_2Q_2}\right)\\
   & \ll  Z^{\varepsilon} \frac{Y_1X_3Y_3}{r_2'}  \left(\frac{X_1}{P_1} + \frac{r_1Y_1X_1}{r_2X_2Q_2}\right) \ll X_1Y_1Z^{1+\varepsilon} \Xi^{\delta - 1};
   \end{split}
 \end{equation}
here we applied \eqref{sym}. Finally we execute the $y_1$-sum by Poisson summation and keep only the central term. This introduces an error of 
 \begin{equation}\label{err3}
  \ll Z^{\varepsilon} \frac{X_1X_3Y_3}{r_2'}  \left(\frac{Y_1}{Q_1} + \frac{r_1Y_1X_1}{r_2X_2Q_2}\right)  \ll X_1Y_1Z^{1+\varepsilon} \Xi^{\delta - 1}. 
 \end{equation}
 Hence, up to an error described by \eqref{err1} -- \eqref{err3}, we can write $N^{(2)}_{\mathbf{r}}(\mathbf{X}, \mathbf{Y})$ in the form
 %If we call $E_{\mathbf{r}}(\mathbf{X}, \mathbf{Y})$ the joint contribution of \eqref{err1} -- \eqref{err3}, we obtain
%  \begin{displaymath}%{equation}\label{finally}
% \begin{split}
% &  N^{(2)}_{\mathbf{r}}(\mathbf{X}, \mathbf{Y}) =  N^{(3)}_{\mathbf{r}}(\mathbf{X}, \mathbf{Y}) +E_{\mathbf{r}}(\mathbf{X}, \mathbf{Y}) % O\left( Z^{\varepsilon}  \frac{(r_3', r_1)^{1/2}}{r_2'}  \frac{X_3Y_3X_1Y_1X_2^{1/2}}{P_2}  + Z^{\varepsilon} \frac{ X_3Y_3}{r_2'}  \left(\frac{X_1Y_1}{P_1} + \frac{X_1Y_1}{Q_1}\right)\right). 
% \end{split}  
 %\end{displaymath} 
%with
 \begin{displaymath}
 \begin{split}
   N^{(3)}_{\mathbf{r}}(\mathbf{X}, \mathbf{Y}) =  & \sum_{abc = r_3'}\frac{\mu(a)}{a^2} \sum_{(d; r'_3/b) = 1} \frac{(d; r_2')^2 }{(r_2')^2d} \sum_{fgh = \frac{db}{(db; r_1)}(r_2; r_3) } \mu(g)  \frac{\phi(r_2'a/(d; r_2'))}{\zeta(2)} \\
& \times \prod_{p \mid \frac{r_2'a}{(d; r_2')}} \left(1- \frac{1}{p^{2}}\right)^{-1} \int_{\Bbb{R}^5}   v_2\left(\frac{dbax_2}{(d; r_2')}\right)    v_1(fgx_1)w_1(hgy_1)    v_3(x_3) \\
& \quad\quad\quad\quad \times w_2\Bigl( \frac{-r_3x_3y_3-r_1f g^2hx_1y_1}{r_2dbax_2/(d; r_2')} \Bigr)w_3(y_3)
 {|x_2|}^{-1} \dd(x_1 ,x_2  , x_3 ,y_1, y_3 ) 
 .
      % v_2(|x_2|) v_3(|x_3|) w_2\Bigl(\Bigl|\frac{r_3x_3y_3+r_1fg^2hx_1y_1}{r_2dbax_2/(d, r_2')}\Bigr|\Bigr)w_3(|y_3|)   v_1(fg|x_1|)w_1(hg|y_1|) {|x_2|}^{-1} \dd(x_1 ,x_2  , x_3 ,y_1, y_3 ) 
 \end{split}
 \end{displaymath} 
 A change of variables yields
 \begin{displaymath}
 N^{(3)}_{\mathbf{r}}(\mathbf{X}, \mathbf{Y}) =  {\cal F}_{\mathbf{r}}  {\cal J}_{\mathbf{r}}(\mathbf{X}, \mathbf{Y} ),  
  \end{displaymath}   
where
\begin{displaymath}
\begin{split}
   {\cal F}_{\mathbf{r}} & 
 = r_2\sum_{abc = r_3'}\frac{\mu(a)}{a^2} \sum_{(d; r'_3/b) = 1} \frac{(d; r_2')^2 }{(r_2')^2d} \sum_{fgh = \frac{db}{(db; r_1)}(r_2; r_3) } \frac{\mu(g)}{fg^2h}  \frac{\phi(r_2'a/(d; r_2'))}{\zeta(2)} \prod_{p \mid \frac{r_2'a}{(d; r_2')}}  \left(1- \frac{1}{p^{2}}\right)^{-1}\\
   &  =  \frac{1}{\zeta(2)} \sum_{abc = r_3'}\frac{\mu(a)}{ab} \sum_{(d; r'_3/b) = 1} \frac{(d; r_2')(db; r_1) }{ d^2} \sum_{fgh = \frac{db}{(db; r_1)}(r_2; r_3) } \frac{\mu(g)}{g}    \prod_{p \mid \frac{r_2'a}{(d; r_2')}}\left(1- \frac{1}{1+p}\right)  
        \end{split}
\end{displaymath}  
is seen to  coincide with the definition \eqref{def-F}, and where 
\begin{displaymath}
  {\cal J}_{\mathbf{r}}(\mathbf{X}, \mathbf{Y} )   =  \frac{1}{r_2}\int_{\Bbb{R}^5}   v_2( x_2 )    v_1(x_1)w_1(y_1)    v_3(x_3) w_2\Bigl( \frac{-r_3x_3y_3-r_1 x_1y_1}{r_2 x_2} \Bigr)w_3(y_3) 
 {|x_2|}^{-1} \dd(\mathbf{x} ,y_1, y_3 )
. 
\end{displaymath}
Note that ${\cal F}_{\mathbf{r}} \ll (r_1r_2r_3)^{\varepsilon}$. 
Turning to $ {\cal J}$, we apply Fourier inversion to see that
\begin{displaymath}
\begin{split}
w_2\Bigl( \frac{-r_3x_3y_3-r_1 x_1y_1}{r_2 x_2} \Bigr) &  = \int_{\Bbb{R}} \int_{\Bbb{R}} w_2(y_2) e(y_2\alpha) \dd y_2\,  e\left(\alpha  \frac{r_3x_3y_3+r_1 x_1y_1}{r_2 x_2} \right) \dd\alpha\\
& =r_2|x_2| \int_{\Bbb{R}} \int_{\Bbb{R}} w_2(y_2) e(\alpha( r_1x_1y_1 + r_2x_2y_2 + r_3x_3y_3) ) \dd y_2   \dd\alpha.
\end{split}
\end{displaymath}
This double integral is not absolutely convergent, but the integral over $\alpha$  is absolutely convergent, and this is all we need to justify the following interchange of integrals:
\begin{displaymath}
   {\cal J}_{\mathbf{r}}(\mathbf{X}, \mathbf{Y} )   =  \int_{\Bbb{R}} \int_{\Bbb{R}^6}      v_1(x_1)w_1(y_1) v_2(x_2) w_2(y_2)   v_3(x_3) w_3(y_3) e(\alpha( r_1x_1y_1 + r_2x_2y_2 + r_3x_3y_3) )  \dd(\mathbf{x} ,  \mathbf{y})  \dd\alpha. 
   \end{displaymath}

  Finally we remove the smooth weight functions. To this end we observe that the estimate
\begin{equation}\label{observe}
\int_{X/2}^X\int_{Y/2}^Y e(\alpha rxy) \dd x \dd y  \ll \min\left(XY, \frac{1}{r |\alpha|}\right)   
\end{equation}
(cf.\ \eqref{sinint}) holds uniformly for $r, X, Y \geq 1$, $\alpha \in \Bbb{R}$, and one also has
\begin{displaymath}
\begin{split}
\int_{B(X, Y)} e(\alpha rxy) \dd (x,y)  & - \int_{\Bbb{R}^2} v(x) w(y) e(\alpha rxy) \dd(x,y) \ll \min\left( PY +  QX, \frac{1}{r |\alpha|}\right),
\end{split}
\end{displaymath}
where, as before, $B(X, Y)$ denotes the region $\frac{1}{2}X \leq |x| \leq X$, $\frac{1}{2}Y \leq |y| \leq Y$. This shows that
\begin{equation}\label{err-final}
\begin{split}
  {\cal J}_{\mathbf{r}}(\mathbf{X}, \mathbf{Y} )& = {\cal I}_{\mathbf{r}}(\mathbf{X}, \mathbf{Y} )
  %  \int_{\Bbb{R}} \int_{\substack{X_1 \leq |x_1| \leq 2X_1\\ Y_1 \leq |y_1| \leq 2Y_1}} \int_{\substack{X_2 \leq |x_2| \leq 2X_2\\ Y_2 \leq |y_2| \leq 2Y_2}}\int_{\substack{X_3 \leq |x_3| \leq 2X_3\\ Y_3 \leq |y_3| \leq 2Y_3}} e(\alpha( r_1x_1y_1 + r_2x_2y_2 + r_3x_3y_3) )  \dd\mathbf{x} \, d \mathbf{y} \, \dd\alpha\\
   + O\left(X_1Y_1Z^{1+\varepsilon}    \Xi^{-\delta} \right),
  \end{split}
\end{equation}
with ${\cal I}_{\mathbf{r}}$ as in \eqref{defI}.  Collecting the error terms \eqref{e1}, \eqref{step1}, \eqref{err1}, \eqref{err2}, \eqref{err3} and \eqref{err-final} and choosing $\delta = 1/6$,   we have now proved the asymptotic relation
\begin{equation}\label{prelim1}
  N_{\mathbf{r}}(\mathbf{X}, \mathbf{Y}) = {\cal F}_{\mathbf{r}}  {\cal I}_{\mathbf{r}}(\mathbf{X}, \mathbf{Y}) + O\left(\frac{X_1Y_1(r_2 X_2Y_2)^{1+\varepsilon}}{\min(X_1, X_2, X_3, Y_1, Y_2, Y_3)^{1/6}}\right),
\end{equation}
yet subject to the additional assumptions \eqref{assump}. 

For fixed $\mathbf{r}$ and $X_1=X_2=X_3=Y_1=Y_2=Y_3 = W$,  it follows easily from Lemma \ref{mellin} that ${\cal I}_{\mathbf{r}}(\mathbf{X}, \mathbf{Y}) \asymp_{\mathbf{r}} W^4$, while the error term is $O_{\mathbf{r}}(W^{23/6+\varepsilon})$. Moreover, both $ N_{\mathbf{r}}(\mathbf{X}, \mathbf{Y}) $ and ${\cal I}_{\mathbf{r}}(\mathbf{X}, \mathbf{Y})$ are symmetric in $r_1, r_2, r_3$. Thus letting $W\rightarrow \infty$, we conclude that ${\cal F}_{\mathbf{r}}$ is symmetric in $r_1, r_2, r_3$, provided $(r_1; r_2; r_3) = 1$. By \eqref{f-euler}, also the Euler factors ${\cal F}_{\mathbf{r}}(p)$ are symmetric, for all primes $p$.  (This can be checked directly, too, but requires some computation.) By Lemma \ref{scaling}b we now infer that ${\cal E}_{\mathbf{r}} = {\cal F}_{\mathbf{r}}$ if $(r_1; r_2; r_3) = 1$, and we have proved Proposition \ref{asymp} under the  assumptions \eqref{assump}.

It remains to remove these extra assumptions.  First,  should it be the case that 
\begin{equation}\label{contra}
10(r_1X_1Y_1 + r_2X_2Y_2) \leq r_3X_3Y_3,
\end{equation}
 then clearly $N_{\mathbf{r}}(\mathbf{X}, \mathbf{Y}) = 0$. We proceed to show that
\rf{contra} also implies that ${\cal J}_r(\mathbf{X}, \mathbf{Y}) = 0$. Indeed, formally integrating by parts in the $\alpha$-integral, we obtain that
\begin{displaymath}
\begin{split}
   {\cal J}_{\mathbf{r}}(\mathbf{X}, \mathbf{Y} )   =  \int_{\Bbb{R}} \int_{\Bbb{R}^6}  &    v_1(x_1)w_1(y_1) v_2(x_2) w_2(y_2)   v_3(x_3) w_3(y_3) \left(-\frac{r_1x_1y_1 + r_2x_2y_2}{r_3x_3y_3}\right)^n \\
   &\hspace{2em} \times e(\alpha( r_1x_1y_1 + r_2x_2y_2 + r_3x_3y_3) )  \dd(\mathbf{x} ,  \mathbf{y})  \dd\alpha 
  \end{split} 
   \end{displaymath}
for any positive integer $n$. In particular, we conclude that ${\cal J}_{\mathbf{r}}(\mathbf{X}, \mathbf{Y}) = 0$ whenever \eqref{contra} holds.  To justify this formal manipulation, we observe that (by partial integration in any of the $x$ or $y$-variables) the $\alpha$-integral is rapidly decaying at $\pm \infty$. Hence we can truncate it (smoothly) with an arbitrarily small error, pull it inside and integrate by parts, pull it outside and complete the range of integration again with an arbitrarily  small error.  This argument shows that the  proposition holds trivially  under the assumption \eqref{contra}, and hence we can drop our initial assumption $r_2X_2Y_2 \asymp r_3X_3Y_3$. 

By \eqref{observe} we see that the $\alpha$-integral in the definition of ${\cal I}_{\mathbf{r}}(\mathbf{X}, \mathbf{Y})$ is absolutely convergent, hence we can make a change of variables to conclude that
$${\cal I}_{d\mathbf r}(\mathbf{X}, \mathbf{Y}) = d {\cal I}_{\mathbf r}(\mathbf{X}, \mathbf{Y})$$
holds for all $d \in \Bbb{N}$. Together with Lemma \ref{scaling}a we see
$${\cal E}_{d\mathbf r}  {\cal I}_{d\mathbf r}(\mathbf{X}, \mathbf{Y}) = {\cal E}_{\mathbf r} {\cal I}_{\mathbf r}(\mathbf{X}, \mathbf{Y})$$
But $N_{d\mathbf r}(\mathbf{X}, \mathbf{Y})  = N_{\mathbf r}(\mathbf{X}, \mathbf{Y})$, whence we may dismiss the assumption that $(r_1; r_2; r_3) = 1$. The proof of the proposition is complete.

\section{The elementary part of the argument}\label{sec4}

\subsection{The universal torsor}

We keep the notation introduced in Section \ref{gcdsec}. 
%For notational simplicity we write
%$$\Bbb{Z}_0 := \Bbb{Z} \setminus \{0\}.$$
%In contrast, $\Bbb{N}_0 := \Bbb{N} \cup \{0\}$.  
Let ${\cal A}$ denote the
set of all $({\bf a,d, z}) \in   \Bbb{Z}_0^3 \times \Bbb{N}^3 \times \Bbb{Z}_0^3 $ 
  that
satisfy the lattice equation  
\begin{equation}\label{21}
a_1 d_1 + a_2 d_2 + a_2 d_3 =0
\end{equation}
and the coprimality constraints \eqref{22} or equivalently \eqref{24}. 
%\begin{equation}\label{22}
%\begin{array}{ll}
%(d_i; d_j) = (z_i; z_j) = (d_k; z_k) =1 & (1\le i < j \le 3, 
%1\le k\le 3), \\
%(a_1 z_1; a_2 z_2; a_3 z_3 ) =1.&
%\end{array}
%\end{equation}
We recall that the four six-tuples $(\pm \mathbf{x}, \pm \mathbf{y})$ satisfying \eqref{1} are representatives of the same point on $V^{\circ}$. The following result from   \cite[Section 2]{BB1} provides a parametrization of the points on $V^{\circ}$. 
\begin{lem}\label{torsor} The mapping ${\cal A} \to V^{\circ}$ defined by
\begin{equation}\label{23}
\left.
\begin{array}{lll}
x_1 = a_1 z_1, & x_2 = a_2 z_2, & x_3 = a_3 z_3,  \\
y_1 = d_2 d_3 z_1, & y_2 = d_1 d_3 z_2, & y_3 = d_1 d_2 z_3
\end{array}
\right. 
\end{equation}
is $4$-to-$1$.
\end{lem}
%We can rewrite the conditions \eqref{22} as 
%\begin{equation}\label{24}
%\begin{array}{ll}
%(d_i; d_j) = (z_i; z_j) = (d_k; z_k) =   1 & (1\le i < j \le 3, 
%1\le k\le 3), \\
%(a_1, a_2, a_3) = (a_i, a_j, z_k) =1 & (\{i, j, k\} = \{1, 2, 3\}).
%\end{array}
%\end{equation}

\subsection{Upper bounds}
  
We will use frequently the following lattice point count \cite[Lemma  3]{HB}. 

\begin{lem}\label{lattice} Let $\mathbf{v}\in \Bbb{Z}^3$ be  primitive and let $H_i >0$ $(1\leq i \leq 3)$. Then the number of primitive ${\mathbf u} \in \Bbb{Z}^3$ that satisfy $u_1v_1 + u_2v_2 + u_3v_3 = 0$ 
and that lie in the box $|u_i| \leq H_i$ $(1 \leq i \leq 3)$, is $O(1+H_1H_2|v_3|^{-1})$.
\end{lem}
   
We introduce the following notation. For   $\mathbf{X} = (X_1, X_2, X_3)$, $\mathbf{Y} = (Y_1, Y_2, Y_3)$ with $X_j, Y_j \geq 1$, $H \geq 1$ and $\mathbf{r}, \bm \alpha, \bm \delta, \bm \zeta \in \Bbb{N}^3$  let ${\cal V}_{\mathbf{r};  (\bm \alpha, \bm \delta, \bm \zeta)}(\mathbf{X}, \mathbf{Y}, H)$   be the set of $9$-tuples $(\mathbf{a}, \mathbf{d}, \mathbf{z}) \in \Bbb{Z}_0^9$ satisfying
\begin{align}
%\begin{split}
&  \notag |a_j z_j| \leq X_j \quad (1 \leq j \leq 3),\quad\quad   |d_id_jz_k| \leq Y_k \quad (\{i, j, k\} = \{1, 2, 3\}),\\
&\label{43} \min(|a_1|, |a_2|, |a_3|, |d_1|, |d_2|, |d_3|, |z_1|, |z_2|, |z_3|) \leq H,\\
& \label{44}  r_1a_1d_1 + r_2a_2d_2 + r_3a_3d_3=0,\\
&\label{45} \alpha_j \mid a_j, \quad \delta_j \mid d_j, \quad \zeta_j \mid z_j \quad (1\leq j \leq 3).
 %&(a_1, a_2, a_3) = (d_1, d_2, d_3) = 1.
%  \end{split}
  \end{align}
   
\begin{lem}\label{lem3} Let $H, X_j,  Y_j, \mathbf{r}, \bm \alpha, \bm \delta, \bm \zeta $ be as in the preceding paragraph, and write $Z = |{\mathbf X}|_1 + |{\mathbf Y}|_1 $. Then
% the number of $(\mathbf{a}, \mathbf{d}, \mathbf{z}) \in {\cal V}_{\mathbf{r}}(\mathbf{X}, \mathbf{Y})$ with  $$\min\big(|a_1|, |a_2|, |a_3|, |d_1|, |d_2|, |d_3|, |z_1|, |z_2|, |z_3|\big) \leq H$$ is at most 
\begin{equation}\label{desired}
|{\cal V}_{\mathbf{r}; (\bm \alpha, \bm \delta, \bm \zeta)}(\mathbf{X}, \mathbf{Y}, H)| \ll  \tau^2\Bigl(\prod_{j=1}^3r_j\alpha_j\delta_j\Bigr) \frac{(X_1X_2X_3)^{2/3}(Y_1Y_2Y_3)^{1/3}}{(\alpha_1\alpha_2\alpha_3\delta_1\delta_2\delta_3)^{2/3} \zeta_1\zeta_2\zeta_3}(\log Z)^2 \log H.
\end{equation}
 \end{lem}

%The dependence on $\mathbf{r}$ is not optimized. We just need polynomial dependence. \\
 
\begin{proof} We use some ideas from  \cite[Section 7]{BB1}. Changing variables
$$r_k \mapsto r_k \alpha_k \delta_k, \quad X_k \mapsto \frac{X_k}{\alpha_k\zeta_k}, \quad Y_k \mapsto \frac{Y_k}{\delta_i\delta_j\zeta_k}$$
with $\{i, j, k\} = \{1, 2, 3\}$, the general version of \eqref{desired} is reduced to the case where $\bm\alpha = \bm \delta = \bm \zeta = (1, 1, 1)$, so that we may concentrate on the latter from now on. Accordingly, we drop $\bm \alpha, \bm \delta, \bm \zeta$ from the notation as these are now fixed to $(1,1,1)$. 
%We need to count the number of $\mathbf{a},  \mathbf{d}, \mathbf{z}$ satisfying
%\begin{displaymath}
%\begin{split}
%&  |a_j| z_j \leq X, \quad 1 \leq j \leq 3,\\
%&  d_id_jz_k \leq Y, \quad \{i, j, k\} = \{1, 2, 3\},\\
%&  a_1d_1 + a_2d_2 + a_3d_3=0,\\
% & \min(|a_1|, |a_2|, |a_3|, d_1, d_2, d_3, |z_1|, |z_2|, |z_3|) \leq H.
 % \end{split}
  %\end{displaymath}
 Without loss of generality we may also assume that $(r_1; r_2; r_3) = 1$. 
 
 We first consider the restricted set $\tilde{ {\cal V}}_{\mathbf{r}}(\mathbf{X}, \mathbf{Y}, H)$ of  $(\mathbf{a}, \mathbf{d}, \mathbf{z}) \in {\cal V}_{\mathbf{r}}(\mathbf{X}, \mathbf{Y}, H)$ satisfying the additional condition
 \begin{equation}\label{add}
 (r_1d_1; r_2d_2; r_3d_3) = (r_1a_1; r_2a_2; r_3a_3) = 1 .
 \end{equation}
We cut the $a_j$ and $d_j$ in dyadic ranges $A_j < a_j \leq 2A_j$ and $D_j < d_j \leq 2D_j$. %In each such range we can consider the lattice equation \eqref{21} as an equation with variables $d_j$ or $a_j$. Correspondingly, 
Lemma \ref{lattice} shows that the number of   $(\mathbf{a}, \mathbf{d}) \in \Bbb{N}^6$ satisfying \eqref{add} and \eqref{21} in a given dyadic range   is at most 
\begin{equation}\label{prelim}
\ll \min(D_1D_2D_3, A_1A_2A_3) + \frac{\prod_j (A_jD_j)}{\max_j(A_jD_j)} \ll \prod_{j=1}^3 (A_jD_j)^{2/3}.
\end{equation}
 Summing this over $\mathbf{z} = (z_1, z_2, z_3)$ with $|z_j| \leq Z_j$ for $1 \leq j \leq 3$, we obtain that for each 6-tuple of dyadic ranges $A_j,D_j$ the contribution is  
\be{asbelow} \ll  \min\left(\frac{X_1}{A_1}, \frac{Y_1}{D_2D_3}, Z_1\right)\min\left(\frac{X_2}{A_2}, \frac{Y_2}{D_1D_3}, Z_2\right)\min\left(\frac{X_3}{A_3}, \frac{Y_3}{D_1D_2}, Z_3\right)\prod_{j=1}^3 (A_jD_j)^{2/3}.\ee
 If we define $E_j = D_1D_2D_3/D_j$, the above simplifies to 
\begin{equation}\label{bound}
 \ll \prod_{j=1}^3 A_j^{2/3}E_j^{1/3} \min\left(\frac{X_j}{A_j}, \frac{Y_j}{E_j}, Z_j\right).
\end{equation} 
Notice now that as $(\nu_1, \nu_2, \nu_3)$ runs over $\Bbb{N}^3$, the triples $(\nu_2 + \nu_3, \nu_1 + \nu_3, \nu_1 + \nu_2)$ take each value in $\Bbb{N}^3$ at most once. Hence we can replace a summation in which the $D_j=2^{\nu_j}$ run over powers of 2 by a sum in which the $E_j$ run over powers of 2.   It remains to sum \eqref{bound} over $A_j$ and $E_j$ which run over powers of 2. For any $X, Y, H \geq 1$ we have  \begin{equation}\label{sum1}
  \sum_{A = 2^{\nu}} A^{2/3}E^{1/3} \min\left(\frac{X}{A}, \frac{Y}{E}, H\right) \ll X^{2/3} \min(Y, HE)^{1/3} 
\end{equation}
uniformly in $1 \leq E \leq Y$, and 
\begin{equation}\label{sum2}
  \sum_{E = 2^{\nu}} A^{2/3}E^{1/3} \min\left(\frac{X}{A}, \frac{Y}{E}\right) \ll X^{2/3} Y^{1/3}
\end{equation}
uniformly in $1 \leq A \leq X$.

If $|a_j| \leq H$ for some $1 \leq j \leq 3$, then summing \eqref{bound} first over $E_1, E_2, E_3$ using \eqref{sum2} and then trivially over $A_1, A_2, A_3$, we arrive at  
 $$|\tilde{{\cal V}}_{\mathbf{r}}(\mathbf{X}, \mathbf{Y}, H)| \ll (X_1X_2X_3)^{2/3}(Y_1Y_2Y_3)^{1/3} (\log Z)^2 \log H.$$

If $|z_j| \leq H$ for some $1 \leq j \leq 3$, then summing \eqref{bound} first over $A_1, A_2, A_3$ using \eqref{sum1} and then trivially over $E_1, E_2, E_3$, we arrive again at  
 $$|\tilde{{\cal V}}_{\mathbf{r}}(\mathbf{X}, \mathbf{Y}, H)| \ll (X_1X_2X_3)^{2/3}(Y_1Y_2Y_3)^{1/3} (\log Z)^2 \log H.$$
 
Finally, if $|d_j| \leq H$ for some $1\leq j \leq 3$, then again we sum \eqref{bound} first over $A_1, A_2, A_3$ using \eqref{sum1}. Noticing that $$\frac{E_iE_k}{H^2} \ll E_j \ll E_iE_k, \quad \{i, j, k\} = \{1, 2, 3\},$$
there are at most $(\log Z)^2 \log H$ terms in the sum over $E_1, E_2, E_3$, and again we obtain
\be{asbelow2}   |\tilde{{\cal V}}_{\mathbf{r}}(\mathbf{X}, \mathbf{Y}, H)| \ll (X_1X_2X_3)^{2/3}(Y_1Y_2Y_3)^{1/3} (\log Z)^2 \log H.\ee
%and again we arrive at the desired bound. %a total contribution of $X^2Y (\log Y)^2 \log H$. 
% then \eqref{bound} becomes
%$$ \ll A_j^{2/3}E_j^{1/3} \min\left(\frac{X_j}{A_j}, \frac{Y_j}{E_j}, H\right) \prod_{i \not= j} A_i^{2/3}E_i^{1/3} \min\left(\frac{X_i}{A_i}, \frac{Y_i}{E_i}\right). $$
%Summing this over $A_i$ (as powers of 2) gives
%$(X_1X_2X_3)^{2/3} (Y_2Y_3)^{1/3}\min(Y_1, HE_1)^{1/3}$, and summing this over $E_1 \leq Y_1$ gives the desired bound. %$X^2Y (\log Y)^2 \log H$, as claimed.\\

With the above bound for  $ |\tilde{{\cal V}}_{\mathbf{r}}(\mathbf{X}, \mathbf{Y}, H)|$ we can easily finish the proof.   If $(r_1a_1; r_2a_2; r_3a_3) = a$ and $(r_1d_1; r_2d_2; r_3d_3) = d$, we now apply our bounds with $X_j(a; r_1r_2r_3)/a$ in place of $X_j$ and $Y_k(d; r_1r_2r_3)^2/d^2$ in place of $Y_k$. Summing over $a$ and $d$ yields \eqref{desired} in all cases.
\end{proof}

For   $B, H \geq 1$ and $\mathbf{r}, \bm \alpha, \bm \delta, \bm \zeta \in \Bbb{N}^3$  let ${\cal V}_{\mathbf{r};  (\bm \alpha, \bm \delta, \bm \zeta)}(B, H)$   be the set of $9$-tuples $(\mathbf{a}, \mathbf{d}, \mathbf{z}) \in \Bbb{Z}_0^9$ satisfying
\begin{equation}\label{new-equation}
 \max_{1 \leq j \leq 3} (|a_j z_j|)^2 \max_{\{i, j, k\} = \{1, 2, 3\}}(|d_id_jz_k|)  \leq B
 \end{equation}
as well as \eqref{43}, \eqref{44} and  \eqref{45}. 

%\begin{displaymath}
%\begin{split}
%& \max_{1 \leq j \leq 3} (|a_j z_j|)^2 \max_{\{i, j, k\} = \{1, 2, 3\}}(|d_id_jz_k|)  \leq B,\\
%& \min(|a_1|, |a_2|, |a_3|, |d_1|, |d_2|, |d_3|, |z_1|, |z_2|, |z_3|) \leq H,\\
%&  r_1a_1d_1 + r_2a_2d_2 + r_3a_3d_3=0,\\
%& \alpha_j \mid a_j, \quad \delta_j \mid d_j, \quad \zeta_j \mid z_j \quad (1\leq j \leq 3).
 %&(a_1, a_2, a_3) = (d_1, d_2, d_3) = 1.
%  \end{split}
 % \end{displaymath}

 Summing  \eqref{desired} over $O(\log B)$ tuples $(\mathbf{X}, \mathbf{Y}) = (4^j, 4^j, 4^j, 4^{2-2j}B, 4^{2-2j}B, 4^{2-2j}B)$, we may now conclude as follows.
 
 \begin{lem}\label{cor6} For   $B, H \geq 1$ and $\mathbf{r}, \bm \alpha, \bm \delta, \bm \zeta \in \Bbb{N}^3$   we have
 $$|{\cal V}_{\mathbf{r};  (\bm \alpha, \bm \delta, \bm \zeta)}(B, H)| \ll \tau^2\Bigl(\prod_{j=1}^3r_j\alpha_j\delta_j\Bigr) \frac{B}{(\alpha_1\alpha_2\alpha_3\delta_1\delta_2\delta_3)^{2/3} \zeta_1\zeta_2\zeta_3}(\log B)^3 \log H.$$
 \end{lem}

A continuous version is given by the following lemma. 
 
 \begin{lem}\label{integral-bound}  Let    $B, H \geq 1$ and let $\mathcal{S} = \mathcal{S}(B, H)$ denote the set of points $(\mathbf{a}, \mathbf{d}, \mathbf{z}) \in [1, \infty)^9$ satisfying \eqref{43} and \eqref{new-equation}. 
% \begin{displaymath}
%\begin{split}
%& \max_{1 \leq j \leq 3} (|a_j z_j|)^2 \max_{\{i, j, k\} = \{1, 2, 3\}}(|d_id_jz_k|)  \leq B,\\
%& \min(|a_1|, |a_2|, |a_3|, |d_1|, |d_2|, |d_3|, |z_1|, |z_2|, |z_3|) \leq H.\\
 %  \end{split}
  %\end{displaymath}
Then
$$\int_{\mathcal{S}} \frac{1}{(a_1a_2a_3d_1d_2d_3)^{1/3}} {\rm d}(\mathbf{a}, \mathbf{d}, \mathbf{z}) \ll B (\log B)^3 (\log H).$$
 \end{lem}
 
 \begin{proof} This is a simpler version of the proof of Lemmas \ref{lem3} and   \ref{cor6}, so we can be brief. 
We cut the variables into   ranges $A_j \leq a_j \leq 2 A_j$, $D_j \leq d_j \leq 2D_j$ and $z_j \leq Z_j$. Fix  $1 \leq X, Y, \leq B$ and consider first the contribution of points where $a_jz_j \leq X$ for $1 \leq j \leq 3$ and  $d_id_jz_k \leq Y$ for $\{i, j, k\} = \{1, 2, 3\}$. Then the integral restricted to this set is
$$ \ll  \min\left(\frac{X}{A_1}, \frac{Y}{D_2D_3}, Z_1\right)\min\left(\frac{X}{A_2}, \frac{Y}{D_1D_3}, Z_2\right)\min\left(\frac{X}{A_3}, \frac{Y}{D_1D_2}, Z_3\right)\prod_{j=1}^3 (A_jD_j)^{2/3}$$
as in \eqref{asbelow}. Arguing as in the proof of Lemma \ref{lem3} with $X_1 = X_2 = X_3 = X$, $Y_1=Y_2= Y_3 = Y$, we see that total contribution of all choices $A_j, D_j, Z_j$ is $\ll X^2Y (\log B)^2 \log H$, as in \eqref{asbelow2}. Finally summing over $O(\log B)$ tuples $(\textbf{X}, \textbf{Y})$, we complete the proof. \end{proof}

\section{The analytic part of the argument}

\subsection{Preliminary transformations}

We begin with some notation. In an effort to establish a sufficiently compact presentation we write a typical vector $\mathbf{x} \in \Bbb{C}^9$ as
\be{coord9}\mathbf{x} = (\mathbf{x}_1, \mathbf{x}_2, \mathbf{x}_3) = (x_{11}, x_{12}, x_{13}; x_{21}, x_{22}, x_{23}; x_{31}, x_{32}, x_{33}).\ee
For a typical index we write $\ell = (i, j) \in \{1, 2, 3\} \times \{1, 2, 3\}$. 
In the notation of the previous sections we write $ (\mathbf{x}_1, \mathbf{x}_2, \mathbf{x}_3)  = (\mathbf{a}, \mathbf{d}, \mathbf{z})$. For vectors $\mathbf{x} = (x_1, \ldots, x_n)$, $\mathbf{y} = (y_1, \ldots, y_n)$ we write $$\mathbf{x}\cdot \mathbf{y}  = (x_1y_1, \ldots, x_ny_n).$$
 % Next suppose that $(A_{ij})$ is a real symmetric $3\times3$ matrix, and that ${\mathbf C}\in \RR^3$.
%We shall frequently encounter 3-tuples  
%$(A_{2 3}C_1, A_{1 3} C_2, A_{12} C_3)$, and we abbreviate this to $(A_{ij}C_k)_k \in \Bbb{R}^3$. For instance,   
%when ${\mathbf b, \mathbf d}\in \NN^3$, then 
%$$ \Bigl(\frac{b_ib_jd_k}{(b_i; b_j)}\Bigr)_k = 
%\left(\frac{b_2b_3d_1}{(b_2; b_3)}, \frac{b_1b_3d_2}{(b_1; b_3)}, \frac{b_1b_2d_3}{(b_1; b_2)} \right).$$

With coordinates on $\Bbb{Z}_0^{9}$ given by \eqref{coord9},   let $\chi  :  \Bbb{Z}_0^{9}     \rightarrow [0, 1]$  be the characteristic function on the set defined by $ x_{11}x_{21} +  x_{12}x_{22} +  x_{13}x_{23} = 0$,
and let $\psi :  \Bbb{Z}_0^{9}     \rightarrow [0, 1]$ be the characteristic function on the set of $9$-tuples satisfying the coprimality conditions corresponding to \eqref{24}, that is, 
\begin{equation*}%\label{24}
\begin{array}{ll}
(x_{2i}; x_{2j}) = (x_{3i}; x_{3j}) = (x_{2k}; x_{3k}) =   1 & (1\le i < j \le 3, \, 
1\le k\le 3), \\
(x_{11}; x_{22}; x_{33}) = (x_{1i}; x_{1j}; x_{3k}) =1 & (\{i, j, k\} = \{1, 2, 3\}).
\end{array}
\end{equation*}
For  $0 \leq \Delta < 1$, we then put
\begin{equation}\label{capF}
F_{\Delta, B}(\mathbf{x}) = \prod_{l=1}^3 \prod_{\substack{1\le i<j\le 3\\\{i, j, k\} = \{1, 2, 3\}}} f_{\Delta}\left(\frac{|(x_{1l}x_{3l})^2 x_{2i}x_{2j}x_{3k}|}{B}\right)
\end{equation}
%and 
%\begin{equation}\label{minmax}
 % \min(|a_1|, |a_2|, |a_3|, |d_1|, |d_2|, |d_3|) \geq \max(|a_1|, |a_2|, |a_3|, |d_1|, |d_2|, |d_3|, z_1, z_2, z_3)^{\delta}.
% \end{equation} 
where $f_{\Delta}$ was defined in \eqref{deff}.  Finally, we introduce the sum
\begin{displaymath}
  N_{\Delta } (B) = \frac{1}{4} \sum_{ \mathbf{x}_1 \in \Bbb{Z}_0^3} \sum_{\mathbf{x}_2 \in \Bbb{N}^3} \sum_{  \mathbf{x}_3 \in \Bbb{Z}^3_0}  \chi (\mathbf{x} ) \psi(\mathbf{x}) F_{\Delta, B}(\mathbf{x})  .  
 \end{displaymath}
We  extend the summation over $\mathbf{x}_2$ to $\Bbb{Z}_0^3$ and include an additional factor $1/8$.  This does not change the value of $N_{\Delta}(B)$, but it is notationally slightly more convenient. 
  Recalling the height condition \eqref{height}, it follows from   Lemma \ref{torsor} and \eqref{23} that
$N_{0}(B) = N(B)$, but it is analytically easier to treat the   smooth version $N_{\Delta } (B) $ for $\Delta > 0$. 
But an asymptotic formula of $N_{\Delta } (B) $ with $\Delta>0$ is all what we require because from  \eqref{suppfsimple} we readily see that the chain of inequalities  
  \begin{equation}\label{envelope}
     N_{\Delta}(B(1-\Delta)) \leq N(B) \leq N_{\Delta}(B) 
  \end{equation}
holds.
  %for any $\Delta \geq 0$. 
   %is non-decreasing in $B$ and $\Delta$. Let us momentarily write 
%\begin{equation}\label{defF}
%F(\mathbf{a}, \mathbf{d}, \mathbf{z}) = \chi (\mathbf{a}, \mathbf{d}, \mathbf{z}) \prod_{l=1}^3 \prod_{\substack{\{i, j, k\} = \{1, 2, 3\}\\i < j}} f_{\Delta}\left(\frac{|a_l|z_l |d_id_j|z_k}{B}\right). 
%\end{equation}
%It is   convenient to introduce the notation
% $$\mathbf{a} \cdot \mathbf{b} = \mathbf{c} \in \Bbb{Z}^9, \quad c_{ij} = a_{ij} b_{ij} \quad (1 \leq i, j \leq 3)$$ 
 %for vectors $\mathbf{a}, \mathbf{b}  \in \Bbb{Z}^9$ and similarly
 % $$\bm \alpha  \cdot \bm \beta =  \bm \gamma \in \Bbb{Z}^3, \quad \gamma_{j} = \alpha_{ij} b_{ij} \quad (1 \leq i, j \leq 3)$$ 
 We remove the function $\psi$, which captures the coprimality conditions, by Lemma \ref{kor2} and estimate the error term by Lemma \ref{cor6} with $H = B$ and $\mathbf{r} = (1, 1, 1)$.  For $T\geq 1$ this gives
 \begin{equation}\label{truncation}
  N_{\Delta} (B) = N_{\Delta, T}(B) + O\left(B (\log B)^4   T^{\varepsilon-1/3}\right), 
\end{equation}
where
\begin{equation}\label{ndt}
\begin{split}
 & N_{\Delta, T} (B) =  \frac{1}{32} \sum_{|\mathbf{b}|, |\mathbf{c}|, |\mathbf{f}|, |\mathbf{g}|, h\leq T} \mu((\mathbf{b}, \mathbf{c}, \mathbf{f}, \mathbf{g}, h))
  \sum_{\mathbf{x}_1 \in \Bbb{Z}_0^3} \sum_{\mathbf{x}_2 \in  \Bbb{Z}_0^3}  \sum_{\mathbf{x}_3 \in \Bbb{Z}_0^3} \chi ({\bm \alpha} \cdot \mathbf{x} )F_{\Delta, B}({\bm \alpha} \cdot \mathbf{x}),
 \end{split}
\end{equation} 
and ${\bm \alpha}$ is as in \eqref{defalpha}. % = ({\bm \alpha}_1, {\bm \alpha}_2, {\bm \alpha}_3) \in \Bbb{N}^9$ with $\bm \alpha_j = (\alpha_{j1}, \alpha_{j2}, \alpha_{j3})$ and  
%\begin{equation}\label{defalpha}
  % \alpha_{1k} = \frac{g_ig_jh }{(h(g_i; g_j); g_ig_j)},   \quad  \alpha_{2k} =   \frac{b_ib_jf_k}{(f_k(b_i; b_j); b_ib_j)}, \quad    \alpha_{3k} =  \frac{c_ic_jf_kg_k}{(f_kg_k(c_i; c_j); c_ic_jg_k; c_ic_jf_k)}
%\end{equation}
%for $k = 1, 2, 3$, $\{i, j, k\} = \{1, 2, 3\}$. 
  The factor $T^{\varepsilon-1/3}$ in the error term of \eqref{truncation} comes from observing that for every subset ${\cal S}$ in the error term of Lemma \ref{kor2},   the corresponding variables $x \in {\cal S}$ occur by Lemma \ref{cor6} at least with an exponent $4/3 - \varepsilon$ in the denominator.

From now on, the analysis will frequently feature multiple Mellin-Barnes integrals over specific vertical lines, and we write $\int^{(n)}$ for an $n$-fold iterated such integral; the lines of integration will be clear from the context or otherwise specified in the text. If all $n$ integrations are over the same line $(\beta)$, then we write  this as $\int^{(n)}_{(\beta)}$.

We continue to manipulate $N_{\Delta, T}(B)$. Let $\Delta > 0$, and  recall the definition \eqref{capF}.
We then  use Mellin inversion and the notation \eqref{power} to recast $N_{\Delta, T} (B)$ as
\begin{equation}\label{recast}
\begin{split}
   \frac{1}{32}& \sum_{|\mathbf{b}|, |\mathbf{c}|, |\mathbf{f}|, |\mathbf{g}|, h\leq T}  \mu((\mathbf{b}, \mathbf{c}, \mathbf{f}, \mathbf{g}, h))% \\
 %& \times
\int_{(1)}^{(9)}   \sum_{\mathbf{x}_1, \mathbf{x}_2, \mathbf{x}_3 \in \Bbb{Z}_0^3} %\sum_{\mathbf{x}_2 \in  \Bbb{Z}_0^3}  \sum_{\mathbf{x}_3 \in \Bbb{Z}_0^3} 
\frac{\chi ({\bm \alpha} \cdot \mathbf{x} )}{   {\bm \alpha}^{\mathbf{v}} \mathbf{x}^{\mathbf{v}}} \prod_{\ell} \left(\widehat{f}_{\Delta}(s_{\ell})B^{s_{\ell}}\right) \frac{\dd\mathbf{s}}{(2\pi \ii)^9},
 \end{split}
\end{equation}
where %we used the convenient vector notation $$\mathbf{c}^{\mathbf{d}} = \prod_{i, j} |c_{\ell}|^{d_{\ell}}   $$ 
 $\mathbf{v} = \mathbf{v}(\mathbf{s})  = (\mathbf{v}_1, \mathbf{v}_2, \mathbf{v}_3) \in \Bbb{C}^3 \times \Bbb{C}^3 \times \Bbb{C}^3$ %(v_{11}, v_{12}, v_{13};  v_{21}, v_{22}, v_{23}; v_{31}, v_{32}, v_{33})$
  is defined by  
 \begin{equation}\label{change1}
 \begin{split}
  &  v_{11} = 2(s_{11}+s_{12}+s_{13}), \quad v_{12} = 2(s_{21}+s_{22}+s_{23}), \quad v_{13} = 2(s_{31}+s_{32}+s_{33}),\\
    &v_{21} =  s_{11}+s_{12}+s_{21}+s_{22}+s_{31}+s_{32},  \\
   & v_{22} = s_{11}+s_{13}+s_{21}+s_{23}+s_{31}+s_{33}, \\
    & v_{23} = s_{12}+s_{13}+s_{22}+s_{23}+s_{32}+s_{33},\\
     & v_{31} = 2(s_{11}+s_{12}+s_{13}) + s_{13}+s_{23}+s_{33} ,\\
   & v_{32} = 2(s_{21}+s_{22}+s_{23}) + s_{12}+s_{22}+s_{32}, \\
    & v_{33} = 2(s_{31}+s_{32}+s_{33}) + s_{11}+s_{21}+s_{31},
 \end{split}
 \end{equation}
and $\ell$ runs over $\{1,2,3\}^2$.
In view of \eqref{mellinrho} and \eqref{difff}, the $\mathbf{s}$-integral in \eqref{recast} is   absolutely convergent.

 At this point it would be possible to evaluate the $\mathbf{x}_3$-sum directly in terms of Riemann's zeta function. This is because $\chi(\bm \alpha \cdot \mathbf{x})$ is independent of $\mathbf{x}_3$. However, it is easier to treat $\mathbf{x}_1$, $\mathbf{x}_2$, $\mathbf{x}_3$ on equal footing. 
 By partial summation and then unfolding the integral, we have
 \begin{displaymath}
 \begin{split}
    \sum_{\mathbf{x}_1, \mathbf{x}_2, \mathbf{x}_3 \in \Bbb{Z}_0^3}\frac{\chi ({\bm \alpha} \cdot \mathbf{x} )}{   {\bm \alpha}^{\mathbf{v}} \mathbf{x}^{\mathbf{v}}} & = \frac{1}{{\bm\alpha}^{\mathbf{v}}} \Bigl(\prod_{\ell}  v_{\ell} \Bigr) \int_{[1, \infty)^9} \sum_{\substack{0 < |x_{\ell}| \leq X_{\ell} }} \chi({\bm \alpha} \cdot \mathbf{x}) \mathbf{X}^{-\mathbf{v} - 1}\dd\mathbf{X}\\
    & =  \frac{1}{{\bm\alpha}^{\mathbf{v}}} \Bigl(\prod_{\ell}  \frac{ v_{\ell}}{1-2^{-v_{\ell}} } \Bigr)
      \int_{[1, \infty)^9} \sum_{\substack{\frac{1}{2}X_{\ell} < |x_{\ell}| \leq X_{\ell} }} \chi({\bm \alpha} \cdot \mathbf{x}) \mathbf{X}^{-\mathbf{v} - 1}\dd\mathbf{X}. 
      \end{split}
      \end{displaymath}     
  In the notation of Proposition \ref{asymp} this equals
  \begin{displaymath}
  \begin{split}      
    %  & \frac{1}{{\bm\alpha}^{\mathbf{v}}} \prod_{\ell}  \frac{ v_{\ell}}{1-2^{-v_{\ell}} } 
%      \int_{[0, \infty)^9} N_{{\bm \alpha}_1 \cdot {\bm \alpha}_2}(\mathbf{X}_1, \mathbf{X}_2) \Bigl(\sum_{\frac{1}{2} %X_{3j} < |x_{3j}| \leq X_{3j}} 1\Bigr) \mathbf{X}^{-\mathbf{v} - 1}\dd\mathbf{X}\\
 %= &   
  \frac{1}{{\bm\alpha}^{\mathbf{v}}} \Bigl(\prod_{\ell}  \frac{ v_{\ell}}{1-2^{-v_{\ell}} }  \Bigr)
      \int_{[1, \infty)^9} N_{{\bm \alpha}_1 \cdot {\bm \alpha}_2}(\mathbf{X}_1, \mathbf{X}_2) \cdot 8 \prod_{j=1}^3\left([X_{3j} ] - \Bigl[\frac{X_{3j}}{2}\Bigr]\right)   \mathbf{X}^{-\mathbf{v} - 1}\dd\mathbf{X}.\\
    \end{split}
 \end{displaymath}
 %Notice that trivially $N_{{\bm \alpha}_1 \cdot {\bm \alpha}_2}(\mathbf{X}_1, \mathbf{X}_2)  \prod_{j=1}^3 ([X_{3j} ] -  [X_{3j}/2]) = 0$ if $\min(X_{ij}) < 1$, so that the previous integral is absolutely convergent. 
 
 We would like to evaluate this integral with the aid  of Proposition \ref{asymp}, and this is successful if we replace the region $[1, \infty)^9$ with  
 \begin{equation}\label{rdelta}
 {\cal R}_{\delta} := \{\textbf{x} \in [1, \infty) : \min(x_1, \ldots, x_n) \geq \max(x_1, \ldots, x_n)^{\delta}\}
 \end{equation}
  for  $0<\delta<1/10$, say. With this in mind,   
 for such $ \delta$, we define
 %$${\cal R}_{\delta} = \{\mathbf{X} \in [1, \infty)^9 \mid \min(X_{\ell}) \geq \max(X_{\ell})^{\delta}\}$$
 %and 
  \begin{equation}\label{ndtd}
\begin{split}
  N_{\Delta, T, \delta} (B) & =  \frac{1}{4} \sum_{|\mathbf{b}|, |\mathbf{c}|, |\mathbf{f}|, |\mathbf{g}|, h\leq T} \mu((\mathbf{b}, \mathbf{c}, \mathbf{f}, \mathbf{g}, h))
\int_{(1)}^{(9)}    \frac{1}{{\bm\alpha}^{\mathbf{v}}} \prod_{\ell}  \frac{ v_{\ell}}{1-2^{-v_{\ell}} } \\
&  \times  \int_{{\cal R}_{\delta}} N_{{\bm \alpha}_1 \cdot {\bm \alpha}_2}(\mathbf{X}_1, \mathbf{X}_2)  \prod_{j=1}^3\left([X_{3j} ] - \Bigl[\frac{X_{3j}}{2}\Bigr]\right)   \mathbf{X}^{-\mathbf{v} - 1}\dd\mathbf{X} \, 
       \prod_{\ell} \widehat{f}_{\Delta}(s_{\ell})B^{s_{\ell}} \frac{{\mathrm d}\mathbf s}{(2\pi \ii)^9}. 
 \end{split}
\end{equation}
%with ${\cal R}_{\delta}$ as in \eqref{rdelta} for $n=9$. %, which by the same argument as before is absolutely convergent. More explicitly, we have
%\begin{equation}\label{L1-function}
 % \prod_{\ell}\frac{v_{\ell} \widehat{f}_{\Delta}(s_{\ell}) }{1 - 2^{-v_{\ell}}}
%\end{equation}
The next lemma estimates
 the error   that we infer by throwing away the information in the cusps.
\begin{lem}\label{delta}
Uniformly for $B\ge 1$, $T\ge 1$, $0<\Delta<1$, $0<\delta<1/10$, one has
$$N_{\Delta, T, \delta}(B) = N_{\Delta, T}(B) + O\left(T^{13} \delta B (\log B)^4\right).$$
\end{lem}
We postpone the proof to the end of this section. We will eventually choose  $T$ to be a small power of $\log B$ and $\delta, \Delta$ small powers of $(\log B)^{-1}$, see \eqref{choice}. 

%We observe that $$\lim_{\delta \rightarrow 0}   N_{\Delta, T, \delta} (B) = N_{\Delta, T}(B).$$ The multiple sum and integral is absolutely convergent. 

\subsection{The error term}

We are now ready to insert the asymptotic formula from Proposition \ref{asymp}, and we also insert the obvious asymptotic formula
$$[X ] - \Bigl[\frac{X}{2}\Bigr] = \frac{X}{2} + O(1) $$
along with the trivial bound $$ N_{{\bm \alpha}_1 \cdot {\bm \alpha}_2}(\mathbf{X}_1, \mathbf{X}_2)  \ll \frac{(X_{11}X_{12}X_{13} X_{21}X_{22}X_{23})^{1+\varepsilon}}{\max(X_{11}X_{21}, X_{12}X_{22},X_{13}X_{23})},$$ which follows from a simple divisor argument. 
This gives
\begin{equation}\label{main-error}
  N_{{\bm \alpha}_1 \cdot {\bm \alpha}_2}(\mathbf{X}_1, \mathbf{X}_2)  \prod_{j=1}^3\left([X_{3j} ] - \Bigl[\frac{X_{3j}}{2}\Bigr]\right) = {\cal E}_{{\bm \alpha}_1 \cdot {\bm \alpha}_2}  {\cal I}_{{\bm \alpha}_1 \cdot {\bm \alpha}_2} (\mathbf{X}_1, \mathbf{X}_2) \frac{X_{31} X_{32}X_{33}}{8}  + \Psi_{{\bm \alpha}_1 \cdot {\bm \alpha}_2}(\mathbf{X}), 
  \end{equation}
  where in the case when  $\mathbf{X} \in {\cal R}_{\delta}$, one has the estimate
  \begin{equation}\label{boundPsi}
  \begin{split}
    \Psi_{{\bm \alpha}_1 \cdot {\bm \alpha}_2}(\mathbf{X}) & \ll  \frac{X_{31}X_{32}X_{33}\prod_{i=1}^2\prod_{j=1}^3 (\alpha_{ij} X_{ij})^{1+\varepsilon}}{\max(X_{11}X_{21}, X_{12}X_{22}, X_{13}X_{23}) \min_{\ell}(X_{\ell}^{1/6})}\\
    & \leq \Bigl(\prod_{i=1}^2\prod_{j=1}^3 \alpha_{ij}^{1+\varepsilon} \Bigr)\Bigl( \prod_{i=1}^2\prod_{j=1}^3 X_{ij}^{2/3 + \varepsilon - \frac{1}{54}\delta} \Bigr)\Bigl(\prod_{j=1}^3 X_{3j}^{1  - \frac{1}{54}\delta}\Bigr).
    \end{split}
  \end{equation}
At this point we see why it is convenient to restrict to the set ${\cal R}_{\delta}$: the asymptotic formula of Proposition \ref{asymp} provides a power saving with respect to the \emph{largest} variable because of the inequality 
$$\min_{\ell} X_{\ell} \geq \prod_{\ell} X_{\ell}^{\delta/9}.$$ 

Inserting the right-hand side of \eqref{main-error} into \eqref{ndtd} yields a corresponding decomposition
\begin{equation}\label{firstdecomp}
 N_{\Delta, T, \delta} (B) =  N^{(1)}_{\Delta, T, \delta} (B) + E_{\Delta, T, \delta} (B).
 \end{equation}
In this section  we estimate the error term. The bound \eqref{boundPsi} implies the bound 
$$\int_{{\cal R}_{\delta}}  \Psi_{{\bm \alpha}_1 \cdot {\bm \alpha}_2}(\mathbf{X}) \mathbf{X}^{\mathbf{v} - 1} \dd\mathbf{X} \ll  \delta^{-9} \prod_{i=1}^2\prod_{j=1}^3 \alpha_{ij}^{1+\varepsilon}  $$
that is valid subject to
$$\Re(v_{ij}) \geq \frac{2}{3} - \frac{\delta}{60} \quad (1 \leq i \leq 2, 1 \leq j \leq 3), \quad \Re v_{3j} \geq 1 - \frac{\delta}{60} \quad (1 \leq j \leq 3).$$
Let $  \sigma = \frac{1}{9} - \frac{\delta}{540}$.  Shifting all  contours to $\Re s_{\ell} = \sigma$, % and using the bounds  \eqref{mellinrho} and \eqref{difff}, 
we obtain
\begin{displaymath}
%\begin{split}
 E_{\Delta, T, \delta} (B) \ll  \frac{B^{1 - \frac{\delta}{60}}}{\delta^9}\sum_{|\mathbf{b}|, |\mathbf{c}|, |\mathbf{f}|, |\mathbf{g}|, h\leq T}  \Big(\prod_{i=1}^2 \prod_{j=1}^3 \alpha_{ij}^{\frac{1}{3} + \varepsilon + \frac{\delta}{90}}\Big) %\prod_{j=1}^3 \alpha_{3j}^{-1+ \frac{\delta}{60}} 
  \int_{(\sigma)}^{(9)}   \prod_\ell \frac{|v_{\ell} \widehat{f}_{\Delta}(s_{\ell})| }{|1 - 2^{-v_{\ell}}|} \,  |\mathrm d\mathbf{s}| .  
  \end{displaymath}
Also with later applications in mind, we observe that for $\Re s_\ell \geq 1/100$ the bounds \eqref{mellinrho} and \eqref{difff} imply that
\begin{equation}\label{L1-function}
  {\cal D}\left(\frac{v_{\ell} \widehat{f}_{\Delta}(s_{\ell}) }{1 - 2^{-v_{\ell}}}\right) \ll_{\mathcal{D}} \frac{\Delta^{18}}{|s_{11}s_{12} \cdots  s_{33}|^2},
\end{equation}
holds for any differential operator ${\cal D}$  in the variables $s_{11},\ldots,s_{33}$. For now we use this with $\mathcal{D} = {\rm id}$, getting
$$  E_{\Delta, T, \delta} (B)  \ll B^{1 - \frac{\delta}{60}} \delta^{-9} \Delta^{-18} T^{13 + 6(\frac{1}{3} + \varepsilon + \frac{\delta}{90})  }.$$
In particular, we then have
  \begin{equation}\label{errorA}
  E_{\Delta, T, \delta} (B)  \ll B^{1 - \frac{\delta}{60}} \delta^{-9} \Delta^{-18} T^{16}, 
% \end{split}
 \end{equation}
 %\begin{equation}
  %\left(\frac{T}{\Delta \delta}\right)^{20} B^{1- \frac{\delta}{60}}
%\end{equation}
uniformly for $B,T,\delta, \Delta$ as in Lemma \ref{delta}. 
 
\subsection{The main term}

We insert the main term in  \eqref{main-error} into \eqref{ndtd} getting
\begin{displaymath}
\begin{split}
  N^{(1)}_{\Delta, T, \delta} (B)  = &  \frac{1}{32} \sum_{|\textbf{b}|, |\textbf{c}|, |\textbf{f}|, |\textbf{g}|, h\leq T} \mu((\mathbf{b}, \mathbf{c}, \mathbf{f}, \mathbf{g}, h))
\int_{(1)}^{(9)} \frac{1}{{\bm\alpha}^{\textbf{v}}} \prod_{\ell}  \frac{ v_{\ell}  \widehat{f}_{\Delta}(s_{\ell})B^{s_{\ell}}}{1-2^{-v_{\ell}} } \\
&  \times \int_{\mathcal{R}_{\delta}}   \mathcal{E}_{{\bm \alpha}_1 \cdot {\bm \alpha}_2}\cdot  \mathcal{I}_{{\bm \alpha}_1 \cdot {\bm \alpha}_2} (\textbf{X}_1, \textbf{X}_2)  X_{31} X_{32}X_{33}  \,
  \textbf{X}^{-\textbf{v} - 1}{\rm d}\textbf{X} \, 
      \frac{{\rm d} \textbf{s}}{(2\pi \ii)^9}. 
 \end{split}
\end{displaymath}
As a first step we would like to make this independent of $\delta$ by replacing $\mathcal{R}_{\delta}$ (defined in \eqref{rdelta}) with the full range $[1, \infty)^9$. We write
$$\mathcal{R}_{\delta} = [1, \infty)^9 \setminus \mathcal{S}_{\delta}$$
and obtain a corresponding decomposition
\begin{equation}\label{decomp}
 N^{(1)}_{\Delta, T, \delta} (B)  =  N^{(2)}_{\Delta, T} (B)  -  N^{(2)}_{\Delta, T, \delta} (B) .
 \end{equation}
We anticipate that $N^{(2)}_{\Delta, T, \delta} (B)$ is small and quantify this in the following lemma.
\begin{lem}\label{new-lemma} We have
$$N^{(2)}_{\Delta, T, \delta} (B) \ll T^{14} \Delta^{-9} \delta B (\log B)^4.$$
\end{lem}

\begin{proof} We first consider the $\textbf{s}$-integral
$$\int_{(1)}^{(9)} \frac{1}{{\bm\alpha}^{\textbf{v}}} \prod_{\ell}  \frac{ v_{\ell}  \widehat{f}_{\Delta}(s_{\ell})B^{s_{\ell}}}{1-2^{-v_{\ell}} }  \textbf{X}^{-\textbf{v}} \frac{{\rm d} \textbf{s}}{(2\pi \ii)^9} = \sum_{\textbf{n} \in \Bbb{N}_0^9} \int_{(1)}^{(9)}  (2^{\textbf{n}} \cdot {\bm \alpha} \cdot \textbf{X})^{-\textbf{v}}  \prod_{\ell}   (v_{\ell}  \widehat{f}_{\Delta}(s_{\ell}) B^{s_{\ell}}) \frac{{\rm d} \textbf{s}}{(2\pi \ii)^9}, $$
where of course $2^{\textbf{n}}$ is the vector $(2^{n_{\ell}})_{\ell \in \{1, 2, 3\}^2}$. 
By \eqref{mellin-related} and  \eqref{change1}, the 9-fold inverse   Mellin transform of $\textbf{s} \mapsto \prod_{\ell} v_{\ell} \widehat{f}_{\ell}(s)$ is a linear combination of functions of the type
$$\textbf{x} \mapsto \prod_{\ell} \phi_{\ell}(x_{\ell}), \quad \phi_{\ell} = {\tt D}^{\nu_{\ell}} f_{\Delta}$$
for ${\bm \nu} \in \Bbb{N}_0^9$ with  $| \bm \nu |_1 = 9$. Hence by Mellin inversion, the above 9-fold integral is a linear combination of expressions of the type
$$\sum_{\textbf{n} \in \Bbb{N}_0^9}  \tilde{F}_{\Delta, B} (2^{\textbf{n}} \cdot {\bm \alpha} \cdot \textbf{X}),$$
where $
\tilde{F}$ is defined as in \eqref{capF} but with some of the functions $f_{\Delta}$ replaced with ${\tt D}^{\nu} f_{\Delta}$. Invoking also the bounds of Lemma \ref{scaling}a and  \eqref{bound-I}   along with the trivial bound $(r_1; r_2; r_3) \leq (r_1r_2r_3)^{1/3}$, it suffices to bound
\begin{displaymath}
\begin{split}
& T^{\varepsilon} \sum_{\textbf{n} \in \Bbb{N}_0^9}  \sum_{|\textbf{b}|, |\textbf{c}|, |\textbf{f}|, |\textbf{g}|, h\leq T}  \int_{\mathcal{S}_{\delta}}  \frac{| \tilde{F}_{\Delta, B} (2^{\textbf{n}} \cdot {\bm \alpha} \cdot \textbf{X})|}{(X_{11} X_{12} X_{13}X_{21} X_{22} X_{23})^{1/3}}  \, {\rm d} \textbf{X} \\
\ll &  T^{13+\varepsilon} \Delta^{-9}\sum_{\textbf{n} \in \Bbb{N}_0^9}  \int_{\mathcal{S}_{\delta}}  \frac{ F_{0, \tilde{B}(\textbf{n})} (  \textbf{X})}{(X_{11} X_{12} X_{13}X_{21} X_{22} X_{23})^{1/3}} \, {\rm d} \textbf{X}
\end{split}
\end{displaymath}
with $\tilde{B}(\textbf{n}) = B(1+\Delta) 2^{-| \textbf{n}|_1}$. 
 Here we just used the simple observation that each $\tilde{F}$ is of size $O(\Delta^{-9})$ by \eqref{suppfsimple} and the above remarks, and $f_{\Delta}$ has support $[0, 1 +\Delta]$. We can further relax the integral by integrating over the slightly larger set $\{\textbf{X} \in [1, \infty)^9 \mid \min(X_{\ell}) \leq (2B)^{\delta}\},$ so that the desired bound follows from Lemma \ref{integral-bound} with $H = (2B)^{\delta}$ and $B = \tilde{B}(\textbf{n})$. 
 \end{proof}

We now focus on the main term $N^{(1)}_{\Delta, T} (B)$ and introduce some notation. Let $z_1, z_2 \in \Bbb{C}$, and let $\mathbf{v}=(\mathbf{v}_1, \mathbf{v}_2, \mathbf{v}_3)$ as in \eqref{change1}. 
Now define 
\begin{equation}\label{w} \mathbf{w}_1 = \mathbf{v}_1 +(z_1, z_2, 1-z_1-z_2),  
\quad \mathbf{w}_2 = \mathbf{v}_2 + (z_1, z_2, 1-z_1-z_2), \quad \mathbf{w}_3 = \mathbf{v}_3,
\end{equation}
 and put
$  \mathbf{w} = (\mathbf{w}_1, \mathbf{w}_2, \mathbf{w}_3) \in \Bbb{C}^9$  so that $\mathbf{w}$ is a linear function in $\mathbf{s}$ and $\mathbf z =(z_1, z_2)$.  We  use \eqref{defer},  \eqref{defI} and Lemma \ref{mellin} to write 
  \begin{displaymath}
\begin{split}
  N^{(2)}_{\Delta, T} (B) =   \int_{(\frac{1}{3})} \int_{(\frac{1}{3})}  \int_{(1)}^{(9)}  {\cal G}_T(\mathbf{s}, \mathbf{z}) \Xi_{\Delta}(\mathbf{s}, \mathbf{z})    B^{s} \frac{{\rm d}\mathbf{s}}{(2\pi \ii)^9} \frac{ \mathrm d z_1 {\rm d} z_2}{(2\pi \ii)^2},  
\end{split}
\end{displaymath}
where $$s = \sum_\ell s_\ell, $$ and where 
\begin{eqnarray}
  \label{g} {\cal G}_T(\mathbf{s}, \mathbf{z})& =& \sum_{ |\mathbf{b}|, |\mathbf{c}|, |\mathbf{f}|,  |\mathbf{g}|, h \leq T} \frac{\mu((\mathbf{b}, \mathbf{c}, \mathbf{f}, \mathbf{g}, h)) }{ \bm \alpha_1^{\mathbf{w}_1 }\bm \alpha_2^{\mathbf{w}_2 } \bm \alpha_3^{\mathbf{w}_3}} \sum_{ q} \frac{\phi(q) 
  }{q^3} \prod_{k=1}^3  (q;  \alpha_{1k}\alpha_{2k} ) ,\\
\label{xi}
  \Xi_{\Delta}(\mathbf{s}, \mathbf{z}) &=&\frac{2}{\pi} K(z_1)K(z_2)K(1-z_1-z_2) \prod_{\ell} \frac{v_{\ell} \widehat{f}_{\Delta}(s_{\ell})  }{1-2^{-v_{\ell}}},
\end{eqnarray}
with $K$ as in \eqref{defK}. Here we have quite a bit of flexibility for the $\textbf{s}$-contours, we only need to make sure that we stay 
\begin{equation}\label{poles}
\text{to the right of poles of $\widehat{f}_{\Delta}(s_{\ell})(1 - 2^{-v_{\ell}})^{-1} (w_{\ell} - 1)^{-1}$.}
\end{equation}

 This is the case, for instance, if 
 %In this section we evaluate asymptotically $N^{(2)}_{\Delta, T } (B) $. We observe that 
% The integrands of $N^{(2)}_{\Delta, T } (B) $ and $  N^{(2)}_{\Delta, T, \delta}(B; \ell, {\bm \sigma}) $ are holomorphic in the region 
 $\Re s_{\ell} > 1/9$ holds for all $\ell$.  
We make the following affine-linear change of variables in the $\mathbf{s}$-integral:
 \begin{equation}\label{change2}
 \begin{split}
 &  y_1 = v_{11} - (1 - z_1), \quad y_2 = v_{21} - (1-z_1),\\
   &  y_3 = v_{12} - (1-z_2), \quad y_4 = v_{22} - (1 - z_2), \quad y_5 = -1+s = -1+\sum_{i, j = 1}^3 s_{ij},
   \end{split}
 \end{equation}
 and $y_6, \ldots, y_9$ are chosen to make the transformation unimodular, e.g.
 \begin{equation}\label{change2a}
   y_6 = s_{11}, \quad y_7 = s_{12}, \quad y_8 = \frac{1}{2}s_{21}, \quad y_9 = \frac{1}{2} s_{22}. 
 \end{equation}
 We write $A(\mathbf{y}) = \mathbf{s}$ for the corresponding inverse transformation $A$, whose Jacobian is 1. This gives 
\begin{equation}\label{defN2}
   N^{(2)}_{\Delta, T } (B)  =   \int^{(11)}  \frac{{\cal H}_{T, \Delta}(A(\mathbf{y}), \mathbf{z}) B^{1+y_5}}{{\cal L}(\mathbf{y})} \frac{{\mathrm d}(\mathbf{y},{\mathbf z})}{(2\pi \ii)^{11}} 
\end{equation}
with the lines of integration defined by
$$ \Re z_j = 1/3, \quad    \Re y_1=  \ldots =  \Re y_4 =   \eta, \quad  \Re y_5 =  5 \eta, \quad \Re y_6= \ldots  = \Re y_9 = 1/15, $$
with
%\begin{equation}\label{h}
 $$ {\cal H}_{T, \Delta}(A(\mathbf{y}), \mathbf{z}) = {\cal G}_{T}(A(\mathbf{y}), \mathbf{z})\Xi_{\Delta}(A(\mathbf{y}), \mathbf{z})$$ 
and 
  %\end{equation} 
\begin{equation*}%\label{defL}
\begin{split}
  {\cal L}(\mathbf{y}) &= \prod_{\ell} (w_{\ell} - 1) \\
  &= y_1y_2y_3y_4 (2y_5-y_3-y_1 )  (2y_5 - y_4 - y_2 )%\\
  %  &\times 
  (y_5-y_2+y_1)(y_5-y_4+y_3 )(y_5 + y_4 - y_3 + y_2 - y_1 ),
  \end{split}
\end{equation*} 
and $\eta > 0$ is chosen so small (say $\eta = 10^{-6}$) that we stay to the right of the poles of $(w_{\ell} - 1)^{-1}$. 
The lines of integration for $y_6, \ldots, y_9$ are to some extent   arbitrary, for instance every line to the right of  $1/18$ and to the left  of $1/12$ satisfies $\Re s_{\ell} > 0$ (as one can check by expressing $s_{\ell}$ in terms of $y_j$) and hence is in agreement with  the condition \eqref{poles}. The fact that the integrand in \eqref{defN2}  has 9 polar lines with 5 variables $y_1, \ldots, y_5$ shows that we can obtain at most  $9-5=4$ $\log$-powers in the final asymptotic formula. By successive contour shifts we show the following asymptotic evaluation.

\begin{lem}\label{contour}
Let  $B\ge 1$, $T\ge 1$, $0<\Delta<1$, $0<\delta<1/10$ and define
\begin{equation}\label{ctd}
  c_{T, \Delta} = \frac{1}{24}\int^{(6)}{\cal H}_{T, \Delta}\Big(A(\mathbf{y})|_{y_1 = \ldots = y_5 = 0}, \mathbf{z}\Big)  \frac{{\mathrm d}(y_6,y_7,y_8,y_9,z_1,z_2)}{(2\pi \ii)^6}, 
\end{equation}
with $\Re z_1 = \Re z _2 = 1/3$, $ \Re y_6 =  \ldots =  \Re y_9 = 1/15$ as lines of integration.  Then
\begin{equation}\label{lemma1}
   N^{(2)}_{\Delta, T } (B) = \frac{1}{24}c_{T, \Delta} B (\log B)^4 + O(T^{14}\Delta^{-18} B (\log B)^3).
\end{equation}

\end{lem}

The proof is a straightforward, but tedious computation that we postpone to the next  section. 
Combining Lemma \ref{contour} with Lemma \ref{new-lemma}, \eqref{decomp}, \eqref{firstdecomp} and \eqref{errorA}, we obtain
\begin{equation}\label{asymp1}
  N_{\Delta, T, \delta}(B) = \frac{1}{24}c_{T, \Delta} B (\log B)^4 + O\left( \Bigl(\frac{T}{\Delta \delta}\Bigr)^{18} B^{1-\frac{\delta}{60}} + \Delta^{-18}T^{14}  B (\log B)^3(1 + \delta  \log B)\right). 
\end{equation}

\subsection{Computation of the leading constant} In this section we compute the constant $c_{T, \Delta}$ defined in \eqref{ctd}. First we observe that $y_1 = \ldots = y_5 = 0$ in combination with \eqref{change2} and \eqref{change1} implies 
\begin{equation}\label{vchange}
  \mathbf{v}_1 =  \mathbf{v}_2 = (1-z_1, 1-z_2, z_1+z_2), \quad \mathbf{v}_3 = (1, 1, 1),
  \end{equation}
hence $\mathbf{w}_1 = \mathbf{w}_2 = \mathbf{w}_3 = (1, 1, 1)$ 
by  \eqref{w}. Inserting this into \eqref{g}, we conclude from Lemma \ref{supereuler} with $\alpha = 3/4$, say, that 
\begin{displaymath}
%\begin{split}
  {\cal G}_T\Big(A(\mathbf{y})|_{y_1 = \ldots = y_5 = 0}, \mathbf{z}\Big)  ={\cal G}_T %=  \sum_{|\mathbf{b}|, |\mathbf{c}|, |\mathbf{f}|, |\mathbf{g}|, h \leq T} \sum_{ q}     \frac{\phi(q)  \mu(\mathbf{b})\mu(\mathbf{c})\mu(\mathbf{f})\mu(\mathbf{g})\mu(h)}{q^3h^3\prod_{j=1}^3 b_j^2c_j^2g_j^3 f_j^2  } \\
%  &\times  \prod_{k=1}^3 \Bigl(q(h(g_i; g_j); g_ig_j)(f_k(b_i; b_j); b_ib_j);  g_ig_jh b_ib_jf_k\Bigr)   \big(f_kg_k(c_i; c_j); c_ic_jg_k; c_ic_jf_k\big)
 %\end{split}
 %\end{displaymath}
 %with our usual convention $\{i, j, k\} = \{1, 2, 3\}$.  From Lemma \ref{supereuler} with $\alpha = 3/4$, say,  % and $\eta = 0$, say,  
 %we obtain
% \begin{displaymath}
  %  {\cal G}_T 
%  = \prod_p \left(1-\frac{1}{p^5}\right) \left(1 + \frac{5}{p} + \frac{5}{p^2} + \frac{1}{p^3}\right) + O(T^{-\frac{3}{4}}) 
= C + O(T^{-\frac{3}{4}}), 
 \end{displaymath}
 where $C$ is as in  \eqref{EPC}. 
% The bound for the error term follows  from Rankin's trick of bounding the characteristic function on $x \geq T$ by $(x/T)^{3/4}$. 
Combining \eqref{defK}, \eqref{xi}, \eqref{vchange}, and writing $s_{ij}$ in terms of $y_j$ by  \eqref{change1}, \eqref{change2} and \eqref{change2a}, we find after a short calculation that 
\begin{displaymath}
%\begin{split}
  \Xi_{\Delta}\Big(A(\mathbf{y})|_{y_1 = \ldots = y_5 = 0}, \mathbf{z}\Big)  = \frac{2}{\pi} {\cal K}(z_1, z_2)  {\cal F}_{\Delta}(y_6, \ldots, y_9, z_1, z_2) \Bigl(\frac{1}{1-\frac{1}{2}}\Bigr)^3,
 \end{displaymath}
 where
 \begin{displaymath}
 \begin{split} 
   {\cal F}_{\Delta}(\mathbf{y}, \mathbf{z}) =   & \widehat{f}_{\Delta}(y_6)\widehat{f}_{\Delta}(y_7)\widehat{f}_{\Delta}\left(\frac{1-z_1}{2} - y_6-y_7\right)\widehat{f}_{\Delta}(2y_8)\widehat{f}_{\Delta}(2y_9)\widehat{f}_{\Delta}\left(\frac{1-z_2}{2} - 2y_8-2y_9\right)\\
   &\times  \widehat{f}_{\Delta}(1-z_1-z_2 - y_6 - 2y_8)  \widehat{f}_{\Delta}(z_2 - y_7 - 2y_9)\widehat{f}_{\Delta}\left(\frac{3z_1 + z_2 -2}{2} + y_6+y_7 + 2y_8 + 2y_9\right)
\end{split}
\end{displaymath} 
and 
\begin{displaymath}
  {\cal K}(z_1, z_2) = \Gamma(z_1) \cos\left(\frac{\pi z_1}{2}\right)\Gamma(z_2) \cos\left(\frac{\pi z_2}{2}\right)\Gamma(1-z_1-z_2) \cos\left(\frac{\pi (1-z_1-z_2)}{2}\right).
  \end{displaymath}
  %K^{\ast}(z) = \Gamma(z) \cos(\pi z/2)$. 
  We would like to replace $ {\cal F}_{\Delta}(\mathbf{y}, \mathbf{z}) $ with  ${\cal F}_{0}(\mathbf{y}, \mathbf{z})$ and estimate the corresponding error.  For $\Re z_j = 1/3$, $\Re y_j = 1/15$, Stirling's formula yields the crude bound ${\cal K}(z_1, z_2) \ll |z_1z_2|^{-1/6}$, and \eqref{approx} -- \eqref{approx2} deliver the bound
  \begin{displaymath}
\begin{split}
 &    {\cal F}_{\Delta}(\mathbf{y}, \mathbf{z}) -  {\cal F}_{0}(\mathbf{y}, \mathbf{z}) 
    \ll \Delta^{\frac{1}{20}} \left( \left|\frac{1-z_2}{2} - 2y_8-2y_9\right| \left|\frac{1-z_1}{2} - y_6-y_7\right| |y_6y_7y_8y_9|\right)^{-\frac{19}{20}} 
  .     \end{split}
\end{displaymath}
%\begin{displaymath}
%\begin{split}
% &  {\cal K}(z_1, z_2) \big({\cal F}_{\Delta}(y_6, \ldots, y_9, z_1, z_2) -  {\cal F}_{0}(y_6, \ldots, y_9, z_1, z_2)\big)\\
%   & \ll \Delta^{\frac{1}{60}} \left( \left|\frac{1-z_2}{2} - 2y_8-2y_9\right| \left|\frac{1-z_1}{2} - y_6-y_7\right|\right)^{-\frac{59}{60}}\prod_{j=6}^9(1+|y_j|)^{-\frac{59}{60}}     
%\prod_{j=1}^2(1+|z_j|)^{-\frac{1}{6}} .     \end{split}
%\end{displaymath}
We now observe that for $\Re z_j = 1/3$, $\Re y_j = 1/15$ the integral 
\begin{displaymath}
\begin{split}
&\int^{(4)}\!\!\int^{(2)} |z_1z_2|^{-1/6} \left( \left|\frac{1-z_2}{2} - 2y_8-2y_9\right| \left|\frac{1-z_1}{2} - y_6-y_7\right| |y_6y_7y_8y_9|\right)^{-\frac{19}{20}}  |\mathrm d\mathbf{z}| \, |\mathrm d\mathbf{y}|\\
 \ll & \int^{(4)} \left( |y_8+y_9| |y_6+y_7| \right)^{\frac{1}{20} - \frac{1}{6}} |y_6y_7y_8y_9|^{-\frac{19}{20}}    |\mathrm d\mathbf{y}| = \Bigl(\int_{(\frac{1}{15})} \int_{(\frac{1}{15})}     |y_6+y_7|^{-\frac{7}{60}} |y_6y_7|^{-\frac{19}{20}}    |\dd y_6\dd y_7|\Bigr)^2
\end{split}
\end{displaymath}
is absolutely convergent, and we conclude that
\begin{displaymath}
\begin{split}
\int^{(6)}& \Xi_{\Delta}\Big(A(\mathbf{y})|_{y_1 = \ldots = y_5 = 0}, \mathbf{z}\Big) 
\frac{{\mathrm d}(y_6,y_7,y_8,y_9,z_1,z_2)}{(2\pi \ii)^6}
 \\
 & = \frac{16}{\pi} \int^{(6)}  {\cal K}(z_1, z_2)  {\cal F}_0(y_6, \ldots, y_9, z_1, z_2) 
\frac{{\mathrm d}(y_6,y_7,y_8,y_9,z_1,z_2)}{(2\pi \ii)^6}
 + O(\Delta^{1/20})
\end{split}
\end{displaymath}
where we integrate over $\Re z_j = 1/3$, $ \Re y_j = 1/15$.   
%\begin{displaymath}
 %  c_{T, \Delta} = \frac{{\cal G}_T}{12\pi} \int_{\Re z_j = 1/3} K^{\ast}(z_1) K^{\ast}(z_2) K^{\ast}(1-z_1-z_2)  \int_{  \Re y_j = 1/20 }F_{\Delta}(\mathbf{y}, \mathbf{z}) \frac{\dd\mathbf{y}}{(2\pi i)^4} \frac{dz_1\, dz_2}{(2\pi i)^2}.
%\end{displaymath}
  % Moreover, since $F_{\Delta} \rightarrow F_0$ for $\Delta \rightarrow 0$, and $|F_0|$ is integrable, we conclude by standard Lebesque theory that  $\lim_{\Delta \rightarrow 0} c_{T, \Delta}  = c_{T, 0}$ for all $T \geq 1$. 
 We evaluate the $\mathbf{y}$-integral by    Lemma \ref{supermellin1} and the $\mathbf{z}$-integral by Lemma \ref{supermellin2}. %, the $\mathbf{y}$-integral  equals
%\begin{displaymath}
 % \int_{\Re y_j = 1/15} {\cal F}_{0}(\mathbf{y}, \mathbf{z})  \frac{\dd\mathbf{y}}{(2\pi i)^4} = \frac{2}{(1-z_1)(1-z_2)z_1z_2(z_1+z_2)(1-z_1-z_2)}
%\end{displaymath}
 % if $\Re z_1 = \Re z_2 = 1/3$, and by Lemma \ref{supermellin2} we have
 % \begin{displaymath}
 %  2 \int_{\Re z_j = 1/3} \frac{{\cal K}(z_1, z_2) }{(1-z_1)(1-z_2)z_1z_2(z_1+z_2)(1-z_1-z_2)} \frac{dz_1\, dz_2}{(2\pi i)^2} =  \frac{\pi}{4}(\pi^2 - 3 + 24\log 2). 
%  \end{displaymath}
% Recalling \eqref{h} and \eqref{ctd},  we can 
The above discussion now delivers 
 \begin{equation}\label{finalconst}
 \begin{split}
  c_{T, \Delta} & = \frac{1}{24} \cdot \frac{16}{\pi} \left(C + O(T^{-3/4})\right) \left( \frac{\pi}{4}(\pi^2 - 3 + 24\log 2) + O(\Delta^{1/20})\right)\\
  & =   \frac{\pi^2 - 3 + 24\log 2}{ 6}C + O\left(\Delta^{1/20}  + T^{-3/4}\right).
  \end{split}
  \end{equation}
 
 \subsection{The endgame} We are ready to complete the proof of Theorem \ref{thm1}. Collecting \eqref{truncation}, Lemma \ref{delta}, \eqref{asymp1} and \eqref{finalconst}, we find that
 \begin{displaymath}
 \begin{split}
   N_{\Delta}(B) & = N_{\Delta, T, \delta}(B) + O\left( B(\log B)^4 \Bigl(\frac{1}{T^{1/4}} + T^{13}\delta\Bigr)\right) \\
   & =  \frac{\pi^2 - 3 + 24\log 2}{6 \cdot 24}C  B(\log B)^4\\
   &  \quad\quad 
   + O\left( B(\log B)^4 \Bigl(\frac{1}{T^{1/4}} +  \Delta^{1/20} + T^{14} \Delta^{-18} \Bigl(\delta + \frac{1}{\log B} \Bigr)  \Bigr)+  \Bigl(\frac{T}{\Delta \delta}\Bigr)^{18} B^{1-\frac{\delta}{60}}  \right). 
 \end{split}  
 \end{displaymath}
 By \eqref{envelope}, the passage from $N_{\Delta}(B)$ to $N(B)$ introduces an error $\Delta B (\log B)^4$ that is already present in the above asymptotic formula, hence the same formula holds for $N(B)$. We now choose  
 \begin{equation}\label{choice}
    \delta = \frac{1}{(\log B)^{99/100}}, \quad \Delta = \frac{1}{(\log B)^{1/24}}, \quad T = (\log B)^{1/120}.
    \end{equation}
 to  complete the proof of Theorem \ref{thm1}, but it remains to provide proofs for Lemmas \ref{delta} and \ref{contour}. 
  
 \subsection{Proof of Lemma \ref{delta}}\label{deltaproof}

In order to estimate the difference between $N_{\Delta, T, \delta}(B)$ and $N_{\Delta, T}(B)$, we reverse the steps between \eqref{ndt} and \eqref{ndtd} and eventually use Lemma \ref{cor6}. Starting from \eqref{ndtd}, we have
\begin{displaymath}
\begin{split}
N_{\Delta, T, \delta}(B) =   \frac{1}{32} \sum_{|\mathbf{b}|, |\mathbf{c}|, |\mathbf{f}|, |\mathbf{g}|, h\leq T} &\mu((\mathbf{b}, \mathbf{c}, \mathbf{f}, \mathbf{g}, h))
 \int^{(9)}_{(1)}    \frac{1}{{\bm\alpha}^{\mathbf{v}}} \prod_{\ell} \Bigl( \frac{ v_{\ell}}{1-2^{-v_{\ell}} }\widehat{f}_{\Delta}(s_{\ell})B^{s_{\ell}}  \Bigr)\\
 & \times
    \Bigl(  \int_{{\cal R}_{\delta}} \sum_{\substack{\frac{1}{2}X_{\ell} < |x_{\ell}| \leq X_{\ell} \\ \ell \in \{1, 2, 3\}^2 }} \chi({\bm \alpha} \cdot \mathbf{x}) \mathbf{X}^{-\mathbf{v} - 1}\dd\mathbf{X}\Bigr) \frac{\dd\mathbf{s}}{(2\pi \ii)^9}
      \end{split}
\end{displaymath}
with $\bm \alpha$ as in \eqref{defalpha}. 
For a vector ${\bm \sigma} \in \{0, 1\}^9$ we define $\mathbf{X}_{\bm \sigma} = (2^{\sigma_{\ell}} X_{\ell})_{\ell \in \{1, 2, 3\}^2 }$ and  
$${\cal R}_{\delta}(\bm \sigma) = \{\mathbf{X} \in [1, \infty)^9 \mid \min(2^{\sigma_{\ell}} X_{\ell}) \geq \max(2^{\sigma_{\ell}} X_{\ell})^{\delta} \}. $$
By a change of variables, the integral over ${\cal R}_{\delta}$ equals
\begin{displaymath}
  \sum_{{\bm \sigma} \in \{0, 1\}^9} (-2)^{|{\bm \sigma}|_1}\int_{{\cal R}_{\delta}(\bm \sigma) } \sum_{\substack{0 < |x_{\ell}| \leq X_{\ell} }} \chi({\bm \alpha} \cdot \mathbf{x}) \mathbf{X}_{\bm \sigma}^{-\mathbf{v} - 1}\dd\mathbf{X}, 
\end{displaymath}
and by partial summation this equals
\begin{displaymath}
\Bigl( \prod_{\ell} \frac{1}{v_{\ell}} \Bigr)\sum_{{\bm \sigma} \in \{0, 1\}^9} (-1)^{|{\bm \sigma}|_1} 2^{-\sum_{\ell} \sigma_{\ell} v_{\ell}} \underset{\hspace{-0.6cm}\mathbf{x} \in \Bbb{Z}_0^9}{\left.\sum \right._+^{(\bm\sigma, \delta)}}\, 
\frac{\chi({\bm \alpha} \cdot \mathbf{x})}{\mathbf{x}^{\mathbf{v}}},
\end{displaymath}
where $\sum_+^{(\bm\sigma, \delta)}$ denotes a summation over $\mathbf{x}$ satisfying
$$
   \min_{\ell}(2^{\sigma_{\ell}} |x_{\ell}|) \geq \max_{\ell}(2^{\sigma_{\ell}} |x_{\ell}|)^{\delta} , $$
and correspondingly we write $\sum_-^{(\bm\sigma, \delta)}$ for a summation with the opposite condition 
\begin{equation}\label{opposite}
   \min_{\ell}(2^{\sigma_{\ell}} |x_{\ell}|) < \max_{\ell}(2^{\sigma_{\ell}} |x_{\ell}|)^{\delta}.
   \end{equation}
In this this notation, one has   $N_{\Delta, T}(B)  - N_{\Delta, T, \delta}(B)\ll M_{\Delta, T, \delta}(B) $ where 
\begin{displaymath}
\begin{split}
 %N_{\Delta, T}(B)  - N_{\Delta, T, \delta}(B) =& \frac{1}{32}
M_{\Delta, T, \delta}(B) =   \sum_{|\mathbf{b}|, |\mathbf{c}|, |\mathbf{f}|, |\mathbf{g}|, h\leq T}   \sum_{{\bm \sigma} \in \{0, 1\}^9} %\mu(\mathbf{b})\mu(\mathbf{c})\mu(\mathbf{f})\mu(\mathbf{g})  \mu(h) \\
%& \times 
\Bigl| \int^{(9)}_{(1)}    2^{-\sum_{\ell} \sigma_{\ell} v_{\ell}} \underset{\hspace{-0.6cm}\mathbf{x} \in \Bbb{Z}_0^9}{\left.\sum \right._-^{(\bm\sigma, \delta)}}\,  \frac{\chi({\bm \alpha} \cdot \mathbf{x})}{{\bm \alpha}^{\mathbf{v}} \mathbf{x}^{\mathbf{v}}} \prod_{\ell} \frac{\widehat{f}_{\Delta}(s_{\ell})B^{s_{\ell}} }{1-2^{-v_{\ell}}} \frac{\dd\mathbf{s}}{(2\pi \ii)^9}\Bigr|.
\end{split}
\end{displaymath}
 %Notice that if the summation condition in the inner most sum was independent of ${\bm \sigma}$, the sum over ${\bm \sigma}$ would cancel the denominator $1-2^{-v_{\ell}}$ in the product over $\ell$. In the present situation 
We write the factor $(1-2^{-v_{\ell}})^{-1}$ as a geometric series and apply Mellin inversion to recast the integral  as
\begin{displaymath}
  \sum_{\mathbf{k} \in \Bbb{N}_0^9}  
 \underset{\hspace{-0.6cm}\mathbf{x} \in \Bbb{Z}_0^9}{\left.\sum \right._-^{(\bm\sigma, \delta)}}\, 
   \chi ({\bm \alpha} \cdot \mathbf{x} )F_{\Delta, B}\left({\bm \alpha} \cdot (2^{k_{\ell} + \sigma_{\ell}}x_{\ell})_{\ell}\right), % = \sum_{\mathbf{x}_1, \mathbf{x}_2, \mathbf{x}_3\in \Bbb{Z}_0^3}  \chi ({\bm \alpha} \cdot \mathbf{x} ) G_{\bm \alpha}(\mathbf{x}),
\end{displaymath}
so that
$$M_{\Delta, T, \delta}(B)  =  \sum_{|\mathbf{b}|, |\mathbf{c}|, |\mathbf{f}|, |\mathbf{g}|, h\leq T}  \sum_{\mathbf{x} \in \Bbb{Z}_0^3}  \chi ({\bm \alpha} \cdot \mathbf{x} ) G_{\bm \alpha}(\mathbf{x})$$
with
$$G_{{\bm \alpha}}(\mathbf{x}) =    \sum_{{\bm \sigma} \in \{0, 1\}^9}   \sum_{\mathbf{k} \in \Bbb{N}_0^9}  F_{\Delta, B}({\bm \alpha} \cdot (2^{k_{\ell} + \sigma_{\ell}}x_{\ell})_{\ell}) \Psi_{\bm\sigma, \delta}(\mathbf{x}),% \mathbbm{1}_{\min(2^{\sigma_{\ell}} |x_{\ell}|) < \max(2^{\sigma_{\ell}} |x_{\ell}|)^{\delta} }(\mathbf{x}). 
$$ 
in which  $\Psi_{\bm \sigma, \delta}$ is the characteristic function of the set defined by \eqref{opposite}. 
In particular, $$\text{supp}\, G_{\bm \alpha} \subseteq  \{\mathbf{x} \in \Bbb{R}^9 \mid \min(|x_{\ell}|) \leq (2B(1+\Delta))^{\delta}\}, \quad G_{{\bm \alpha}}(\mathbf{x}) \ll \sum_{\mathbf{k} \in \Bbb{N}_0^9} F_{0, B(1+\Delta)}\left(\bm \alpha \cdot (2^{k_{\ell} }x_{\ell})_{\ell}\right).$$  
By Lemma \ref{cor6} with $(2B(1+\Delta))^{\delta}$ in place of $H$ and $B(1+\Delta)$ in place of $B$ we conclude that
\begin{displaymath}
 M_{\Delta, T, \delta}(B)\ll   \sum_{|\mathbf{b}|, |\mathbf{c}|, |\mathbf{f}|, |\mathbf{g}|, h\leq T}  \sum_{\mathbf{k} \in \Bbb{N}_0^9}   2^{-(\frac{2}{3}-\varepsilon)\sum_{\ell} k_{\ell}} B (\log B)^3  \log (B^{\delta}) \ll \delta T^{13} B (\log B)^4,
   \end{displaymath}
    as desired. 
 
\section{Proof of Lemma \ref{contour}}

We start with the evaluation of $N^{(2)}_{\Delta, T}(B)$, defined in \eqref{defN2}, and prove \eqref{lemma1}. We perform various contour shifts with the variables $y_1, \ldots, y_5$.  The variables $y_6, \ldots, y_9$, $z_1, z_2$ will be kept fixed. We will always stay in the region  $|\Re y_1|, \ldots, |\Re y_5|  \leq 10 \eta$ with $\eta=10^{-6}$ as before, and we remember our choice $\Re z_1 = \Re z_2 = 1/3$. In this region we have $\Re w_{\ell} \geq 1 - 40 \eta $, as one can check from  \eqref{change1}, \eqref{w}, \eqref{change2} and \eqref{change2a}, and we derive now rather crude, but convenient bounds for the  function ${\cal H}_{T, \Delta}(\mathbf{s}, \mathbf{z})$ and its derivatives appearing in the integrand of \eqref{defN2} (the derivatives are needed for residue computations). Let ${\cal D}_j$ denote a differential operator of degree $j$ in $s_{11}, \ldots, s_{33}$. Then for $\Re w_{\ell} \geq 1 - 40\eta $ we obtain by the most trivial estimates
\begin{equation}\label{boundg}
  {\cal D}_j {\cal G}_T(\mathbf{s}, \mathbf{z}) \ll_j \sum_{ |\mathbf{b}|, |\mathbf{c}|, |\mathbf{f}|,  |\mathbf{g}|, h \leq T}    \prod_{k=1}^3  \alpha_{1k}^{1-\Re w_{1k} + \varepsilon} \alpha_{2k}^{1-\Re w_{2k} +\varepsilon} \ll T^{13 + 720\eta + \varepsilon } \ll T^{14}.
\end{equation}
Similarly, for  $\Re s_\ell > 1/100$ we conclude from \eqref{L1-function} that
\begin{equation}\label{boundxi}
  {\cal D}_j  \Xi_{\Delta}(\mathbf{s}, \mathbf{z}) \ll_j \Delta^{-18} |s_{11}s_{12}  \cdots  s_{33} z_1z_2|^{-2}. 
  \end{equation}

We now   shift successively the $y_1, \ldots, y_4$ contours to $\Re y_j = - j\eta$, $1 \leq j \leq 4$, thereby picking up a simple  pole at $0$ and a remaining integral.  This leaves us altogether   with 16 terms, some of which are identical by symmetry. We denote by $V \subseteq \{y_1, y_2, y_3, y_4\}$ the set of variables that have not been integrated out and  distinguish several cases. For notational simplicity we write
$\tilde{{\cal H}}(y_1, \ldots, y_5) = {\cal H}_{T, \Delta}(A(\mathbf{y}), \mathbf{z}).$ 

\subsection{Case I: $V = \emptyset$} The term consisting only of residues equals
\begin{displaymath}
    \int_{  (5 \eta) }\frac{\tilde{{\cal H}}(0, 0, 0, 0, y_5)  B^{y_5}}{4y_5^5} \frac{\mathrm d y_5 }{2\pi \ii}  =  \frac{1}{4} \tilde{H}(\mathbf{0}) \frac{  (\log B)^4}{4!} + O( (\log B)^3  T^{14} \Delta^{-18})
\end{displaymath}
%with
%\begin{displaymath}
 % c_{T, \Delta} = \frac{1}{24}\int_{\Re z_j = 1/3} \int_{  \Re y_6, \ldots \Re y_9 = 1/15 } {\cal H}_{T, \Delta}\Big(A\mathbf{y}|_{\substack{y_1=y_2 = 1-z_1\\ y_3=y_4 = 1-z_2\\ y_5 = 1}}, \mathbf{z}\Big)  \frac{dy_6 \cdots dy_9}{(2\pi i)^4} \frac{dz_1 dz_2}{(2\pi i)^2}. 
%\end{displaymath} \\
by shifting the contour to $\Re y_5 = -\eta$, say, and spelling out the leading term of the residue, while estimating the lower order terms and the remaining integral trivially (and crudely) by \eqref{boundg} and \eqref{boundxi}. 

\subsection{Case II: $V = \{y_1\}$} 
Here we have
\begin{displaymath}
\begin{split}
\int_{(-\eta)} \int_{(5\eta)}
\frac{\tilde{{\cal H}} (y_1, 0, 0, 0, y_5) B^{y_5}}{2 y_1 y_5^2 (y_5-y_1 ) ( 2 y_5   - y_1 ) (y_5+y_1)} \frac{\mathrm d y_5 \dd y_1  }{(2\pi \ii)^2}  . 
   \end{split}
\end{displaymath}
Shifting the line  $\Re y_5= 5\eta$ to $\Re y_5= -\eta/3$, we pick up a pole at $y_5 = -y_1$ and $y_5 = 0$, the latter of which as well as the remaining integral we estimate trivially. Hence the previous expression equals
 \begin{displaymath}
 \begin{split}
  &   \int_{  ( - \eta)}\frac{\tilde{{\cal H}}(y_1, 0, 0, 0, 0) B^{-y_1}}{12    y_1^5 } \frac{\mathrm d y_1 }{2\pi \ii}   + O ( T^{14}\Delta^{-18}  \log B)   = - \frac{1}{12}\tilde{H}(\mathbf{0}) \frac{  (\log B)^4}{4!} + O(T^{14}\Delta^{-18}   (\log B)^3)
\end{split}   
\end{displaymath}
which we realize after shifting the line of integration to $\Re y_1= \eta$ and spelling out only the leading term of the residue at $y_1 = 0$.  The same evaluation holds for  $V = \{y_2\}$,  $V = \{y_3\}$, $V = \{y_4\}$. 

\subsection{Case IIIa: $V = \{y_1, y_2\}$} In 
\begin{displaymath}
\begin{split}
    & \int_{(-\eta)} \int_{(-2\eta)} \int_{(5\eta)}
\frac{\tilde{{\cal H}}(y_1, y_2, 0, 0, y_5)   B^{y_5}}{ y_1y_2 y_5 (y_5 + y_1 - y_2  ) (y_5 - y_1 + y_2  ) (2y_5 - y_1  ) (2y_5  - y_2 ) }
  \frac{\mathrm dy_5 \dd y_2 \dd y_1  }{(2\pi \ii)^3}  
   \end{split}
\end{displaymath}
we shift the line  $\Re y_5=5\eta$ to $\Re y_5=  - \eta/3$ and argue similarly.  Up to an error of $O( T^{14}\Delta^{-18} )$ coming from the simple pole at $y_5 = 0$, we pick up the residue at $y_5 =   y_1 - y_2$, which equals
\begin{displaymath}
\int_{(-2\eta)}\int_{(-\eta)} \frac{\tilde{{\cal H}} (y_1, y_2, 0, 0, y_1-y_2) B^{ y_1 - y_2}}{2 y_1y_2 (y_1 - y_2)^2 (  2 y_1 - 3 y_2 ) (  y_1 - 2 y_2 )   } 
  \frac{\mathrm dy_1 \dd y_2  }{(2\pi \ii)^2}  = O( T^{14}\Delta^{-18}  \log B), 
\end{displaymath}
as we see from shifting $\Re y_1$ to $  - 5/2\eta$ and estimating trivially the contribution of the double pole at $y_2 = y_1$.  The same bound holds by symmetry for $V = \{y_3, y_4\}$.

%  The terms where $y_1$ and $y_2$ or where $y_3$ and $y_4$ have been integrated out is easily seen to be $O_{T, \Delta}(B \log B)$ (shift $y_5$ to the left and then $y_1$ further to the left). 
 
 \subsection{Case IIIb: $V = \{y_1, y_3\}$} 
 In 
\begin{displaymath}
     \int_{(-\eta)} \int_{(-3\eta)} \int_{(5\eta)}
\frac{\tilde{{\cal H}}(y_1, 0, y_3, 0, y_5)   B^{y_5}}{ 2 y_1y_3y_5   (y_5 +  y_1  ) (y_5 - y_1 - y_3  )  (2y_5-y_1 - y_3   )  (y_5 + y_3   )}
  \frac{\mathrm dy_5 \dd y_3 \dd y_1  }{(2\pi \ii)^3} 
   \end{displaymath}
we shift the line $\Re y_5=5\eta$ to $\Re y_5=  - \eta$. Up to an error of $O( T^{14}\Delta^{-18} )$ for the simple pole at $y_5 = 0$ we get contributions from two poles at $y_5 =   - y_1  $ and $y_5 =   - y_3 $. The former yields
\begin{displaymath}
\int_{(-3\eta)}\int_{(-\eta)}
\frac{\tilde{{\cal H}}(y_1, 0, y_3, 0, -y_1)   B^{-y_1 }}{2 y_1^2 y_3 (y_1 - y_3 ) (2 y_1 + 
   y_3  )  (3 y_1 + y_3  )  }
  \frac{\mathrm d y_1 \dd y_3  }{(2\pi \ii)^2}   = O( T^{14}\Delta^{-18}  \log B), 
\end{displaymath}
as one finds after shifting the line $\Re y_1=-\eta $ to $ \Re y_1= \eta/2$ and estimating trivially the double pole at $y_1 = 0$.  The latter yields 
\begin{displaymath}
\int_{(-\eta)}\int_{(-3\eta)}
\frac{\tilde{{\cal H}}(y_1, 0, y_3, 0, -y_3)   B^{-y_3 }}{2 y_1 y_3^2 (y_3 - y_1 ) ( y_1 + 
   2y_3  )  ( y_1 + 3y_3  )  }
  \frac{\mathrm d y_3 \dd y_1  }{(2\pi \ii)^2} .  
\end{displaymath}
We shift the line $\Re y_3=-3\eta$ to $\Re y_3=\eta/4$. Up to an error of $O( T^{14}\Delta^{-18} \log B)$, we get a contribution of the  pole at $y_3 = y_1$, and its residue is
 \begin{displaymath}
-\int_{ ( - \eta) }\frac{\tilde{{\cal H}}(y_1, 0, y_1, 0, -y_1)   B^{-y_1 }}{24 y_1^5  }
  \frac{\mathrm d y_1    }{2\pi \ii}  = \frac{1}{24} \tilde{{\cal H}}(\mathbf{0}) \frac{(\log B)^4}{4!} + O( T^{14}\Delta^{-18} (\log B)^3).  
\end{displaymath}

\subsection{Case IIIc:   $V = \{y_2, y_3\}$} In
 \begin{displaymath}\int_{(-2\eta)} \int_{(-3\eta)} \int_{(5\eta)}
\frac{\tilde{{\cal H}}  (0, y_2, y_3, 0, y_5) B^{y_5}}{y_2y_3 (2y_5 - y_2 ) (y_5 - y_2  ) (2y_5  - y_3    )  (y_5 + y_2 - y_3 )   (y_5+ y_3  )  }
  \frac{\mathrm dy_5 \dd y_3 \dd y_2  }{(2\pi \ii)^3}  
\end{displaymath}
 we shift the line  $\Re y_5=5\eta $ to $\Re y_5= - \eta/2$. We pick up one pole at $y_5 =   - y_3  $ that contributes
 \begin{displaymath}
\int_{(-2\eta)}\int_{(-3\eta)}
     \frac{\tilde{{\cal H}}(0, y_2, y_3, 0, -y_3) B^{ - y_3  }}{3 y_2 y_3^2 (2 y_3 -y_2 )   ( 
   y_2 + y_3  ) (y_2 + 2 y_3  )  }
  \frac{\mathrm d y_3 \dd y_2 }{(2\pi \ii)^2}  . 
   \end{displaymath}
We shift the line  $\Re y_3=-3\eta$ to $ \Re y_3= \eta/2$. The pole at $y_3  = 0$ contributes $O( T^{14}\Delta^{-18}\log B)$, and the residue at $y_3 =  y_2/2$ equals
\begin{displaymath}
  - \frac{2}{9}  \int_{  (- 2\eta)}\frac{\tilde{{\cal H}} (0, y_2, y_2/2, 0, -y_2/2) B^{ - y_2  /2  }}{ y_2^5  }
  \frac{\mathrm d y_2 }{2\pi \ii}  = \frac{1}{72} \tilde{H}(0) \frac{(\log B)^4}{4!} + O( T^{14}\Delta^{-18}  (\log B)^3)
  , 
\end{displaymath}
as is readily confirmed by shifting the line of integration to the far right. The same evaluation holds for $V = \{y_1, y_4\}$.

  \subsection{Case IIId: $V = \{y_2, y_4\}$}
 We consider
 \begin{displaymath}\int_{(-2\eta)} \int_{(-4\eta)} \int_{(5\eta)}
   \frac{\tilde{{\cal H}}(0, y_2, 0, y_4, y_5)   B^{y_5}}{ 2 y_2y_4  y_5 (y_5-y_2  )  (y_5-y_4   )  (2y_5-y_2 - 
   y_4  ) (y_5+ y_2 + y_4   )}
    \frac{\mathrm d y_5\dd y_4 \dd y_2 }{(2\pi \ii)^3}  . 
\end{displaymath}
We begin by moving $\Re y_5=5\eta$ to $\Re y_5= - \eta$. Observing the pole at $y_5=-y_2-y_4$, we then see that this integral equals 
\begin{displaymath}
 \int_{(-4\eta)} \int_{(-2\eta)} \frac{\tilde{{\cal H}}(0, y_2, 0, y_4, -y_2-y_4 )B^{ - y_2-y_4 }}{6  y_2   y_4   (y_2 + y_4 )^2     ( 
   2 y_2 + y_4  ) (y_2 + 2 y_4  ) }
    \frac{\mathrm d y_2\dd y_4  }{(2\pi \ii)^2}  + O( T^{14}\Delta^{-18}). 
\end{displaymath}
Next, by shifting $\Re y_2=-2\eta$ to $\Re y_2= 5\eta$, we pick up three residues and a remaining integral of size  $O( T^{14}\Delta^{-18}B^{-\eta})$. The pole at $y_2 =  -y_4$ contributes $O( T^{14}\Delta^{-18} \log B)$, the pole at $y_2 = 0$ gives
\begin{displaymath}
 -   \int_{ (- 4\eta) }\frac{\tilde{{\cal H}}(0, 0, 0, y_4, -y_4)  B^{ -y_4  }}{12   y_4^5  }
    \frac{\mathrm d y_4  }{2\pi \ii }   = \frac{1}{12} \tilde{{\cal H}}(\mathbf{0}) \frac{(\log B)^4}{4!} + O( T^{14}\Delta^{-18} (\log B)^3), 
\end{displaymath}
and the pole at $y_2 = -y_4/2 $ contributes
\begin{displaymath}
  \frac{4}{9}  \int_{   ( - 4\eta ) }\frac{\tilde{{\cal H}}(0, -y_4/2, 0, y_4 -y_4/2) B^{ -y_4 /2}}{   y_4  ^5  }
    \frac{ \mathrm d y_4 }{2\pi \ii}  = -\frac{1}{36} \tilde{{\cal H}}(\mathbf{0}) \frac{(\log B)^4}{4!} + O( T^{14}\Delta^{-18} (\log B)^3). 
\end{displaymath}
%Hence the total contribution of the case $V = \{y_2, y_4\}$ is
%\begin{displaymath}
% - \frac{13}{36}c_{T, \Delta} B (\log B)^4 + O_{T, \Delta}(B (\log B)^3).
%\end{displaymath}

\subsection{Case IVa: $V = \{y_2, y_3, y_4\}$} We wish to evaluate 
\begin{displaymath}
  \int^{(4)}\frac{\tilde{{\cal H}}(0, y_2, y_3, y_4, y_5)   B^{y_5}}{y_2y_3y_4(y_5 + y_3 - y_4  ) (y_5 + y_4 - y_3 + y_2  )  (y_5- y_2 )  ( y_3 - 2 y_5  ) (y_2 + y_4 - 2 y_5  )}    \frac{{\mathrm d}(y_2,y_3,y_4,y_5)}{(2\pi \ii)^4} 
,
\end{displaymath}
with integrations over $ \Re y_2 =  - 2\eta$, $  \Re y_3 =  - 3\eta$,  $\Re y_4=  -4\eta $, $\Re y_5 =   5 \eta$. First, 
we shift  $\Re y_5=5\eta$ to $\Re y_5= - \eta/2$. The remaining integral is $O(T^{14}\Delta^{-18}B^{-\eta/2})$ and we pick up a pole at $y_5 =  - y_2 + y_3 - y_4 $ with residue
\begin{displaymath}
    \int^{(3)}  \frac{\tilde{{\cal H}}(0, y_2, y_3, y_4, -y_2+y_3-y_4)  B^{ - y_2 + y_3 - y_4 }}{y_2y_3y_4 (y_2 - 2 y_3 + 2 y_4 ) (2 y_2 - y_3 + y_4  )  (2 y_2 - y_3 + 2 y_4  ) (3 y_2 - 2 y_3 + 3 y_4  )}    \frac{{\mathrm d}(y_2,y_3,y_4) }{(2\pi \ii)^3}  ;
\end{displaymath}
here the integrations are over the same lines as before. 
Next, we shift $\Re y_3=-3\eta$ to $\Re y_3=  - 7 \eta$. The remaining integral then is $O( T^{14}\Delta^{-18}B^{-\eta})$, and we pick up a pole at $y_3 = (2y_4 + y_2  )/2$ with residue 
\begin{displaymath}
  \frac{-4}{3}
\int_{(-4\eta)} \int_{(-2\eta)}
 \frac{\tilde{{\cal H}}(0, y_2, y_4 + y_2/2, y_4, -y_2/2) B^{ -y_2/2}}{ y_2^2 y_4   (2 y_2 + y_4  )  ( 
   y_2 + 2 y_4 ) ( 3 y_2 + 2 y_4  ) }    \frac{\mathrm d y_2 \dd y_4 }{(2\pi \ii)^2}   = O( T^{14}\Delta^{-18} \log B),
\end{displaymath}
as is readily seen after 
shifting $\Re y_2=-2\eta$ to $\Re y_2= \eta$. The same bound holds, by symmetry, for $V = \{y_1, y_2, y_4\}$. 
 
 \subsection{Case IVb: $V = \{y_1, y_3, y_4\}$} In this case we consider 
\begin{displaymath}
   \int^{(4)} \frac{\tilde{{\cal H}}(y_1, 0, y_3, y_4, y_5)   B^{y_5}}{y_1y_3y_4  ( y_3 - y_4 + y_5)   (y_5 - y_1 - y_3 + y_4  ) ( 
   y_1 + y_5 )   (y_4 - 2 y_5  ) (y_1 + y_3 - 2 y_5 )}    \frac{{\mathrm d}(y_1, y_3 , y_4 , y_5)}{(2\pi \ii)^4}  ,
\end{displaymath}
with integrations over the lines$ \Re y_1 =  - \eta$, $ \Re y_3 =   - 3\eta$, $ \Re y_4 =  - 4\eta$, $ \Re y_5 =   5 \eta $. As in the previous case, 
we shift  $\Re y_5=5\eta$ to $\Re y_5=- \eta/2$. The remaining integral is $O( T^{14}\Delta^{-18}B^{-\eta/2})$, the pole at $y_5 = y_1+y_3 - y_4  $ contributes $O(T^{14}\Delta^{-18})$, and we are left with the pole at $y_5 = -y_1  $. The latter has  residue
\begin{displaymath}
\begin{split}
  &  \int^{(3)}\frac{\tilde{{\cal H}}(y_1, 0, y_3, y_4, -y_1)   B^{- y_1 }}{ y_1y_3y_4 (y_1 - y_3 + y_4  ) (2 y_1 + y_3 - y_4 )   (2 y_1 + y_4 ) (3 y_1 + y_3  )}    \frac{{\mathrm d}(y_1, y_3 , y_4) }{(2\pi \ii)^3} ,   \end{split}
\end{displaymath}
with lines of integrations as before.
We shift $\Re y_1=-\eta$ to $\Re y_1=   \eta/2$. The remaining integral  is $O(T^{14}\Delta^{-18}B^{-\eta/2})$, the pole at $y_1 = 0$ contributes $O(T^{14}\Delta^{-18})$, and we only need to consider the pole at $y_1 = (y_4 - y_3 )/2$ whose residue contributes
\begin{displaymath}
\begin{split}
   -\frac{4}{3} \int_{(-4\eta)} \int_{(-3\eta)}
 \frac{\tilde{{\cal H}}((y_4-y_3)/2, 0, y_3, y_4, (y_3-y_4)/2)  B^{  ( y_3 - y_4)/2}}{(y_3 - y_4)^2 ( y_3 - 3 y_4  ) ( y_3 - 2 y_4  )    }  \frac{\mathrm  d y_3 \dd y_4 }{(2\pi \ii)^2}.
   \end{split}
\end{displaymath}
 Now we shift $\Re y_3=-3\eta$ to $\Re y_3= - 5\eta$. The remaining integral  is $O(T^{14}\Delta^{-18}B^{-\eta})$, and the pole at $y_3 = y_4$ contributes $O(T^{14}\Delta^{-18} \log B)$. %, and we need to consider the pole at $y_3 = 1-z_2$ with residue
%\begin{displaymath}
%\begin{split}
 %  -\frac{2}{9} \int_{\Re z_j = 1/3}& \int_{\substack{    \Re y_4 = 2/3 - 4\eta \\   \Re y_6, \ldots \Re y_9 = 1/15 }}  \frac{{\cal H}_{T, \Delta}\Big(A\mathbf{y}|_{\substack{y_1 = (y_4 +z_2- 2 z_1 + 1)/2\\ y_2 = 1-z_1\\ y_3 = 1-z_2\\ y_5 = (3-y_4-z_2)/2 }}, \mathbf{z}\Big) B^{(3  - y_4-z_2)/2}}{(y_4 + z_2 - 1)^5}    \frac{    dy_4\, dy_6 \cdots dy_9}{(2\pi i)^5} \frac{dz_1 dz_2}{(2\pi i)^2}\\
% &  = \frac{1}{72} c_{T, \Delta} (B \log B)^4 + O(\Delta^{-1}B (\log B)^3). 
 %  \end{split}
%\end{displaymath}
 The same bound  holds in the case $V = \{y_1, y_2, y_3\}$. 
 
\subsection{Case V:  $V = \{y_1, y_2, y_3, y_4\}$} Finally, in the case where none of the variables has been integrated out, we shift $\Re y_5=5\eta$ to $ - \eta/2$; the remaining integral is $O(T^{14}\Delta^{-18}B^{-\eta/2})$, and we pick up a pole with residue
\begin{displaymath}
\begin{split}
   \int^{(4)}& \frac{\tilde{{\cal H}}(y_1, y_2, y_3, y_4, y_1-y_2+y_3-y_4  ) B^{ y_1 - y_2 + y_3 - y_4}}{ y_1y_2y_3y_4 (y_1 - y_2 + 2 y_3 - 2 y_4) (2 y_1 - 2 y_2 + y_3 - y_4) } \\
   &\hspace{4em}\times \frac{1}{  (  2 y_1 - 3 y_2 + 2 y_3 - 3 y_4 ) (  y_1 - 2 y_2 +
    y_3 - 2 y_4)}
     \frac{{\mathrm d}(y_1,y_2,y_3,y_4)  }{(2\pi \ii)^4}. 
   \end{split}
\end{displaymath}
Here, all lines of integration are given by $\Re y_j =  - j\eta$. 
We shift $\Re y_1$ to $ - 7/2\eta$. The remaining integral is $O(T^{14}\Delta^{-18}B^{ -\eta/2})$, and we pick up a simple pole at $y_1 = y_2 - (y_3 - y_4)/2$ with residue
\begin{displaymath}
 -\frac{4}{3}   \int^{(3)}\frac{\tilde{{\cal H}}( y_2 - (y_3-y_4)/2, y_2, y_3, y_4, (y_3-y_4)/2 ) B^{(y_3-y_4)/2}}{  y_2y_3y_4(y_3 - y_4) (  - 2 y_2 + y_3 - y_4  )  ( - y_2 + y_3 - 
   2 y_4  )   (  - 2 y_2 + y_3 - 3 y_4  )}
     \frac{{\mathrm d}(y_2,y_3,y_4) }{(2\pi \ii)^3}, 
\end{displaymath}
with lines of integration still given by $\Re y_j = -j\eta$.
Next we shift the line for $y_3$ to $\Re y_3  - 5\eta$. The remaining integral is $O(T^{14}\Delta^{-18}B^{-\eta})$, and the pole at $y_3 = y_4$ contributes $O(T^{14}\Delta^{-18})$.\\ %, and the pole at $y_3 = 1-z_2$ contributes
%\begin{displaymath}
%\begin{split}
%  & -\frac{8}{3} \int_{\Re z_j = 1/3} \int_{\substack{ \Re y_2 = 2/3 - 2\eta\\ \Re y_4 = 2/3 - 4\eta \\  \Re y_6, \ldots \Re y_9 = 1/15 }}\frac{{\cal H}_{T, \Delta}\Big(A\mathbf{y}|_{\substack{y_1 = y_2 + y_4/2 - (1-z_2)/2 \\ y_3 = 1 - z_2\\ y_5 =(3 - y_4 - z_2)/2}}, \mathbf{z}\Big) B^{(3-y_4-z_2)/2}}{ ( y_2 + z_1-1) (y_4 + z_2-1)^2 (2 y_2 + y_4 + 2 z_1 + z_2-3)}\\ & \frac{1}{ ( 
 %  y_2 + 2 y_4 + z_1 + 2 z_2-3) ( 2 y_2 + 3 y_4 + 2 z_1 + 3 z_2-5)  }
 %    \frac{dy_2\, dy_4 \, dy_6 \cdots dy_9}{(2\pi i)^6} \frac{dz_1 dz_2}{(2\pi i)^2} = O(\Delta^{-20}B\log B) 
 %  \end{split}
%\end{displaymath}
%upon shifting $\Re y_4$ to $2/3 + \eta/2$. \\

Summarizing all previous calculations, we obtain \eqref{lemma1} from Cases I, II (with multiplicity 4), IIIb, IIIc (with multiplicity 2) and IIId, since  $$\frac{1}{4} - \frac{4}{12} + \frac{1}{24} + \frac{2}{72}    + \frac{1}{12} - \frac{1}{36}   = \frac{1}{24}.$$

\section{The geometry of the crepant resolution}\label{geometry}

Let $X\subset \Bbb{P}^{2}\times \Bbb{P}^{2}\times \Bbb{P}^{2}$ be
the smooth triprojective variety described in \eqref{tri} with  trihomogeneous coordinates  
$(\mathbf{x},\mathbf{y},\mathbf{z}%
)=(x_{1},x_{2},x_{3};y_{1},y_{2},y_{3};z_{1},z_{2},z_{3})$. The aim of this chapter is
to compute Peyre's alpha invariant of $X$. We will not specify the base
field as the results in this chapter are purely algebraic and independent of
the base field.

Along with $X$ we consider the non-singular biprojective surface $Y\subset \Bbb{P}^{2}\times \Bbb{P}^{2}$ 
defined in
bihomogeneous coordinates $(\mathbf{y},\mathbf{z})$  by  
$y_{1}z_{1}=y_{2}z_{2}=y_{3}z_{3}$, and the subvariety $Z\subset  \Bbb{P}^{2}\times \Bbb{P}^{2}$
defined   in
bihomogeneous coordinates $(\mathbf{x},\mathbf{z})$    by $x_{1}z_{1}+x_{2}z_{2}+x_{3}z_{3}=0$.   We also recall that  $V\subset \Bbb{P}^{2}\times \Bbb{P}^{2}$ is the
singular biprojective cubic threefold with bihomogeneous coordinates $(%
\mathbf{x},\mathbf{y})=(x_{1},x_{2},x_{3};y_{1},y_{2},y_{3})$ as in \eqref{1}. There are natural projections
$$p : X \rightarrow Y, \quad g : X \rightarrow Z, \quad f : X \rightarrow V$$
defined by $p : (\mathbf{x}, \mathbf{y}, \mathbf{z}) \mapsto (\mathbf{y}, \mathbf{z})$, $g : (\mathbf{x}, \mathbf{y}, \mathbf{z}) \mapsto (\mathbf{x}, \mathbf{z})$, $f :  (\mathbf{x}, \mathbf{y}, \mathbf{z}) \mapsto (\mathbf{x}, \mathbf{y})$. We will frequently use these maps and its corresponding induced functorial maps. We will also use   the $\Bbb{G}_{m}^{3}$-action on $%
\Bbb{P}^{2}\times \Bbb{P}^{2}\times \Bbb{P}^{2}$ defined by
\begin{equation*}\label{71}
 {\bm \gamma }({\mathbf x},{\mathbf y},{\mathbf z})
= (\gamma _{1}x_{1},\gamma _{2}x_{2},\gamma _{3}x_{3};\gamma _{1}y_{1},\gamma
_{2}y_{2},\gamma _{3}y_{3};\gamma _{1}^{-1}z_{1},\gamma
_{2}^{-1}z_{2},\gamma _{3}^{-1}z_{3})
\end{equation*}
for ${\bm \gamma =}(\gamma _{1},\gamma _{2},\gamma _{3})\in \Bbb{G}%
_{m}^{3}$,  and its restriction to   $\Bbb{G}_{m}^{3}$-actions on $X$, $\Bbb{P}^2 \times \Bbb{P}^2$ and   $Z$, the latter two given by    \begin{equation}\label{73}
  {\bm \gamma }({\mathbf x},{\mathbf z})=(\gamma
_{1}x_{1},\gamma _{2}x_{2},\gamma _{3}x_{3};\gamma _{1}^{-1}z_{1},\gamma
_{2}^{-1}z_{2},\gamma _{3}^{-1}z_{3}).
\end{equation}
The morphism $g$ is then $\Bbb{G}_{m}^{3}$%
-equivariant and the base extension of the morphism $h:Y\rightarrow \Bbb{P%
}^{2}$, $(\mathbf{y},\mathbf{z})\mapsto \mathbf{z}$ along the second
projection $\text{pr}_{2}:Z\rightarrow \Bbb{P}^{2}$. As $h$ is the blow-up of $\Bbb{P}^{2}$ at the
three points where two of the $\mathbf{z}$-coordinates vanish, we thus obtain that $g$ is
the blow-up at the union of the three disjoint lines $l_{i}$ on $Z$ defined
by 
\begin{equation}\label{lines}
  l_i : \quad x_{i}=z_{j}=z_{k}=0
\end{equation}  
for $\{i, j, k\} = \{1, 2, 3\}$.

\subsection{The pseudoeffective cone}

Nine integral subsurfaces
of $X$ will be important for the computation of $\alpha (X)$: if $1\leq i\leq 3$ and $\left\{ j,k\right\} =\left\{
1,2,3\right\} \backslash \left\{ i\right\},$ then
\begin{itemize}
\item[] 
$D_{i} \subset \Bbb{P}^{2}\times \Bbb{P}^{2}\times \Bbb{P}^{2}$  is defined by the equations $x_{i}=y_{i}=z_{j}=z_{k}=0$,
\item[] $E_{i}\subset \Bbb{P}^{2}\times \Bbb{P}^{2}\times \Bbb{P}^{2}$  
is defined by the equations $x_jz_j +x_kz_k$ and  $ y_{j}=y_{k}=z_{i}=0$,
\item[] $F_{i} \subset \Bbb{P}^{2}\times \Bbb{P}^{2}\times \Bbb{P}^{2}$  
is defined by the equations $x_{i}=x_{j}y_{k}+x_{k}y_{j}=x_{j}z_{j}+x_{k}z_{k}=0$ and $%
y_{j}z_{j}-y_{i}z_{i}=y_{k}z_{k}-y_{i}z_{i}=0$.
 \end{itemize}
 Here $D_{i}$ is isomorphic to $\Bbb{P}^{1}\times \Bbb{P}^{1}$ while $E_{i}$
is a $\Bbb{P}^{1}$-bundle over the line in $\Bbb{P}^{2}\times \Bbb{P%
}^{2}$ with coordinates $(\mathbf{y},\mathbf{z})$ defined by $%
y_{j}=y_{k}=z_{i}=0$. For $(\mathbf{x},\mathbf{y},\mathbf{z})\in F_{i}$, we
note that one of the two equalities $(x_{j},x_{k})=(y_{j},-y_{k})$ or $(x_{j},x_{k}) = (z_{k},-z_{j})$ holds in $\Bbb{P}%
^{1}$. Hence $p:X\rightarrow Y$ restricts to an isomorphism from $F_{i}$ to $%
Y$.\\

The $\alpha $-invariant is defined by means of Cartier divisors. As $X$
is smooth, we may also view such divisors as Weil divisors \cite[p.\ 141]{Ha} and
regard them as members of the free abelian group $\text{Div} \, X$ generated by the
prime divisors. We may then extend this group to the group $\text{Div}_{\Bbb{R}%
}X = \text{Div} \, X\otimes _{\Bbb{Z}}\Bbb{R}$ of $\Bbb{R}$-divisors and
consider the submonoid of effective $\Bbb{R}$-divisors (see \cite[p.\ 48]{La}).
The pseudoeffective cone $C_{\text{eff}}(X) \subset  \text{Pic} \, X\otimes  
\Bbb{R}$ is  the closure of the convex cone spanned by
the classes of all effective $\Bbb{R}$-divisors on $X$ (see \cite[p.\ 47]{La}). The main result of this subsection is Proposition \ref{pseudo} below, asserting that $C_{\text{eff}}(X)$ is spanned by the nine classes $[D_i], [E_i], [F_i]$, $1 \leq i \leq 3$.  We start with the following lemma.

\begin{lem} The group $\Bbb{G}_{m}^{3}$ acts trivially on  ${\rm Pic} \, 
 X$.\end{lem}
 
\begin{proof}   The surface $Y$ is a del Pezzo surface of degree six
and $\text{Pic} \, Y$ is spanned by the classes of its six lines. The image $p^{\ast }(\text{Pic} \, 
Y)$ of the functorial map $p^{\ast } : \text{Pic} \, Y\rightarrow \text{Pic} \, X$ is
therefore spanned by the classes of all $D_{i}$ and $E_{i}$. As $D_{i}$ and $%
E_{i}$ are $\Bbb{G}_{m}^{3}$-invariant, $\Bbb{G}_{m}^{3}$ acts
trivially on $p^{\ast }(\text{Pic} \,Y)$.
  
Next, let $L=\text{pr}_{1}^{\ast }(O_{\Bbb{P}^{2}}(1))$ be the sheaf
associated to the first projection $\text{pr}_{1}$ of $X\subset \Bbb{P}%
^{2}\times \Bbb{P}^{2}\times \Bbb{P}^{2}$. Then, $[L]+p^{\ast }(\text{Pic}\, 
Y)$ generates $\text{Pic}\, X/p^{\ast }(\text{Pic}\, Y) \cong\Bbb{Z}$ as $X$ is a $\Bbb{P}%
^{1}$-bundle over $Y$. As $\Bbb{G}_{m}^{3}$ acts trivially on $[L]$ and $p^{\ast }(\text{Pic}\, Y)$, it acts trivially on  $\text{Pic} \,  X$ .\end{proof}

The following lemma will make it easier to determine $C_{\text{eff}}(X)$.

\begin{lem}\label{lem21}  An effective divisor  on $X$ is linearly equivalent to a $\Bbb{G}_m^3$-invariant effective divisor on $X$.
\end{lem}
 
\begin{proof}  An effective divisor $D$
on $X$ is given by the vanishing of a global section of the invertible $%
O_{X} $-module $L=O_{X}(D)$. As $\Bbb{G}_{m}^{3}$ stabilizes the class of 
$L$ in $\text{Pic}\, X$, it therefore follows from the proof of \cite[Prop.\ 1.5, p.\ 34]{Mu} that we may endow $L$ with a $\Bbb{G}_{m}^{3}$-linearization (see
also \cite[Prop.\ 2.3]{Ha}). This is equivalent to a lifting of the $\Bbb{G}%
_{m}^{3}$-action on $X$ to a $\Bbb{G}_{m}^{3}$-action on the line bundle $%
\mathbf{L\rightarrow }X$ defined by $L$ (see \cite[p.\ 31]{Mu}). There is thus an
induced rational representation of $\Bbb{G}_{m}^{3}$ on $H^{0}(X, L)$.
Since  $\Bbb{G}_{m}^{3}$
is diagonalizable (\cite[p.\ 21]{Sp}), this induced rational representation 
must be a direct sum of one-dimensional ones.  Hence there is a $\Bbb{G}_{m}^{3}$-invariant
one-dimensional subspace $S$ of $H^{0}(X, L)$. The divisor of zeros of  $S$ (\cite[p.\ 157]{Ha}) is then a $\Bbb{G}_{m}^{3}$-invariant effective
divisor on $X$ linearly equivalent to $D$, as desired.\end{proof}
  
 %We recall the $\Bbb{G}_{m}^{3}$-action on $\Bbb{P}%
%^{2}\times \Bbb{P}^{2}$ defined by \eqref{73}, and we 
In the following we will use  $\Bbb{G}_{m}^{3}$%
-linearizations on invertible sheaves $L$ on $\Bbb{P}^{2}\times \Bbb{P}%
^{2}$ and $Z$ compatible with   \eqref{73}. To
construct such a linearization on $L=O_{\Bbb{P}^{2}\times \Bbb{P}%
^{2}}(m,n),$ let 
$\langle m, n\rangle = \binom{m+3}{3}\binom{n+3}{3}-1$ and  $$h_{m,n}:\Bbb{P}^{2}\times \Bbb{P}^{2}\rightarrow 
\Bbb{P}^{\langle m, n\rangle}$$ be the morphism defined by all
monomials of bidegree $(m,n)$ in $(x_{1},x_{2},x_{3};z_{1},z_{2},z_{3})$.
Then we have $O_{\Bbb{P}^{2}\times \Bbb{P}^{2}}(m,n)=h_{m,n}^{\ast }O_{\Bbb{%
P}^{\langle m, n\rangle}}(1)$,  and there is a
natural $\Bbb{G}_{m}^{3}$-action on $H^{0}(\Bbb{P}^{\langle m, n\rangle},  O_{\Bbb{P}^{\langle m, n\rangle}%
}(1))$ given by
\begin{equation}\label{74}  
  ({\bm \gamma} G)\mathbf{(\mathbf{x},\mathbf{z})=}G(\gamma
_{1}x_{1},\gamma _{2}x_{2},\gamma _{3}x_{3};\gamma _{1}^{-1}z_{1},\gamma
_{2}^{-1}z_{2},\gamma _{3}^{-1}z_{3})
\end{equation}
for homogeneous polynomials   $G(\mathbf{x},\mathbf{z})$ of
bidegree $(m,n)$. This $\Bbb{G}_{m}^{3}$-action gives rise to a $\Bbb{G%
}_{m}^{3}$-linearization on $O_{\Bbb{P}}(1)$, which may be pulled back to
a $\Bbb{G}_{m}^{3}$-linearization on $O_{\Bbb{P}^{2}\times \Bbb{P}%
^{2}}(m,n)=$ $h_{m,n}^{\ast }O_{\Bbb{P}^{\langle m, n\rangle}}(1)$ (see   \cite[Prop.\ 1.7, p.\ 34]{Mu}). Similarly, by considering the restriction of $h_{m,n}$ to $Z$, we
obtain a $\Bbb{G}_{m}^{3}$-linearization on $O_{Z}(m,n)$ such that the
induced restriction from $H^{0}(\Bbb{P}^{2}\times \Bbb{P}^{2},O_{%
\Bbb{P}^{2}\times \Bbb{P}^{2}}(m,n))$ to $H^{0}(Z,O_{Z}(m,n))$ is $%
\Bbb{G}_{m}^{3}$-equivariant.

\begin{lem}\label{lem22}  Let $\Delta $  be a $
\Bbb{G}_{m}^{3}$-invariant effective divisor 
 on $Z$.  Then there exists a one-dimensional $\Bbb{G}%
_{m}^{3}$-invariant subspace $S$ of $H^{0}(%
\Bbb{P}^{2}\times \Bbb{P}^{2},O_{\Bbb{P}^{2}\times \Bbb{P}%
^{2}}(m,n))$ such that $\Delta $  is  the  
 divisor of  the section $s_{Z}\in H^{0}(Z,O_{Z}(m,n))$  for any $s\in
S\setminus \{0\}$.
\end{lem}

\begin{proof}  Every effective divisor on $Z$ is the divisor $\text{div}(\sigma )
$ of some global section $\sigma $ of an invertible sheaf $L$ on $Z$. It is
well known (cf.\ e.g.\ \cite[Th.\ 2.4]{Sc}) that any invertible sheaf on $Z$ is
isomorphic to some $O_{Z}(m,n)$ where $m,n\geq 0$ whenever $%
H^{0}(Z,O_{Z}(m,n))\neq 0$.  We may and shall thus assume that $%
L=O_{Z}(m,n)$ for $m,n\geq 0$. An effective divisor $\Delta $ on $Z$ will
then correspond to a one-dimensional subspace $\Sigma $ of $%
H^{0}(Z,O_{Z}(m,n))$ for some $m,n\geq 0$ (cf.\ \cite[p.\ 157]{Ha}),  and $\Delta $
will be $\Bbb{G}_{m}^{3}$-invariant if and only if $\Sigma $ is $%
\Bbb{G}_{m}^{3}$-invariant.

We now apply the K\"{u}nneth formula in \cite{SW} to $\text{pr}_{1}^{\ast }(O_{\Bbb{P}%
^{2}}(k))\otimes \text{pr}_{2}^{\ast }(O_{\Bbb{P}^{2}}(l)).$ We then obtain
that $H^{1}(\Bbb{P}^{2}\times \Bbb{P}^{2},O_{\Bbb{P}^{2}\times 
\Bbb{P}^{2}}(k,l))= 0$ as $H^{1}(\Bbb{P}^{2},O_{\Bbb{P}^{2}}(k)) 
=0$ for all $k$. In particular, we conclude that   $H^{1}(\Bbb{P}^{2}\times \Bbb{P}^{2},O_{%
\Bbb{P}^{2}\times \Bbb{P}^{2}}(m-1,n-1))=0$, and hence the restriction map from $%
H^{0}(\Bbb{P}^{2}\times \Bbb{P}^{2},O_{\Bbb{P}^{2}\times \Bbb{P}%
^{2}}(m,n))$ to $H^{0}(Z,O_{Z}(m,n))$ must be surjective.

As this restriction map is $\Bbb{G}_{m}^{3}$-equivariant and the $\Bbb{%
G}_{m}^{3}$-representation on $H^{0}(\Bbb{P}^{2}\times \Bbb{P}^{2},O_{%
\Bbb{P}^{2}\times \Bbb{P}^{2}}(m,n))$ is a direct sum of
one-dimensional ones, there is  some one-dimensional $\Bbb{G%
}_{m}^{3}$-invariant subspace of $H^{0}(\Bbb{P}^{2}\times 
\Bbb{P}^{2},O_{\Bbb{P}^{2}\times \Bbb{P}^{2}}(m,n))$, which
restricts to $\Sigma $ on $Z$.
\end{proof}

We prepare for the statement of the next lemma with a definition.
  A bihomogeneous polynomial $G(\mathbf{x},%
\mathbf{z})$ of bidegree $(m,n)$ is said to be an \emph{eigenpolynomial}
under $\Bbb{G}_{m}^{3}$ if it is contained in a one-dimensional $\Bbb{G%
}_{m}^{3}$-invariant subspace of $H^{0}(\Bbb{P}^{2}\times \Bbb{P}%
^{2},O_{\Bbb{P}^{2}\times \Bbb{P}^{2}}(m,n))$. In other words, $G$ is   an
eigenpolynomial if and only if for each ${\bm \gamma} \in \Bbb{G}_{m}^{3}$, we
can find a constant $c$ such that ${\bm \gamma} G=cG$ under the action in \eqref{74}. 
In the following lemma, monomials in three variables occur. We write these in the compact form 
$\mathbf{x}^{\mathbf{a}} = x_1^{a_1} x_2^{a_2} x_3^{a_3}$ for $\mathbf{a} \in \Bbb{N}_0^3$. Confusion with
the notation \eqref{power} that was used in the analytic part should not arise.

\begin{lem}\label{lem23}  Let $G(\mathbf{x},\mathbf{z})$  
be    a bihomogeneous eigenpolynomial   under  $\Bbb{G}_{m}^{3}$. 
 Then there exists a monomial $%
M_{0}=\mathbf{x}^{\mathbf{e}}\mathbf{z}^{\mathbf{f}}$
  and  a ternary homogeneous 
polynomial $H$ such that $%
G=M_{0}H(x_{1}z_{1},x_{2}z_{2},x_{3}z_{3})$. 
\end{lem}

\begin{proof}  Let $I$ be the set of all sixtuples $%
(a_{1},a_{2},a_{3},b_{1},b_{2},b_{3})$ such that $%
\mathbf{x}^{\mathbf{a}} \mathbf{z}^{\mathbf{b}}
$ is a monomial in $G$ with non-zero coefficient. As ${\bm \gamma} M=\bm \gamma^{\mathbf{a} - \mathbf{b}}M$ for $%
M=\mathbf{x}^{\mathbf{a}} \mathbf{z}^{\mathbf{b}}
$,   the characters  sending  ${\bm\gamma} \in \Bbb{G}_m^3$ to $\bm \gamma^{\mathbf{a} - \mathbf{b}} \in \Bbb{G}_m$ 
coincide for all sixtuples in $I$. The triples $%
(a_{1}-b_{1},a_{2}-b_{2},a_{3}-b_{3})$ and the sixtuples $%
(e_{1},e_{2},e_{3},f_{1},f_{2},f_{3})$ with $$e_{i}=\max (a_{i}-b_{i},0), \quad 
 f_{i}=\max (b_{i}-a_{i},0)$$ will thus be the same for all sixtuples in $I$.
Hence, defining $%
M_{0}=\mathbf{x}^{\mathbf{e}}\mathbf{z}^{\mathbf{f}}
$, we get that $M=M_{0}\prod_i(x_{i}z_{i})^{\min (a_{i},b_{i})}$
for any monomial $%
M=\mathbf{x}^{\mathbf{a}} \mathbf{z}^{\mathbf{b}}
$ in $G$ with non-zero coefficient. Thus there   exists  a homogeneous polynomial $%
H $ of degree $$ \sum_{i=1}^3 \min(a_i, b_i) = \frac{1}{2}\sum_{i=1}^3(a_{i}+b_{i}-e_{i}-f_{i}) \geq 0$$ with $%
G=M_{0}H(x_{1}z_{1},x_{2}z_{2},x_{3}z_{3})$.\end{proof}

We now consider the images $g^{\ast }(\Delta )\in \text{Div} \, X$ of effective
divisors $\Delta \in \text{Div} \, Z$ under the functorial map $g^{\ast }: \text{Div} \, 
Z\rightarrow \text{Div} \, X$.

\begin{lem}\label{lem24} Let $H(t_{1},t_{2},t_{3})$ be a ternary   homogeneous polynomial of
degree $n$  not  divisible by $t_{1}+t_{2}+t_{3}$, and let  
 $\Delta \in \text{\rm Div}\, Z$  be the effective divisor defined
by $H(x_{1}z_{1},x_{2}z_{2},x_{3}z_{3})$. Then the multiplicity of %
$D_{i}$  in $g^{\ast }(\Delta )\in \text{\rm Div} \, X$  is equal to $n$ %
  for $i=1,2$  and $3$.
\end{lem}

\begin{proof} Let $Z_0\subset Z$ be the subscheme associated to $\Delta $
(cf.\ \cite[p.\ 145]{Ha}) and let $l_{1},l_{2},l_{3}$ be the three lines on $Z$ described
in \eqref{lines}. Then, $D_{1}+D_{2}+D_{3}$ is the exceptional divisor (cf.\ \cite[App.\ B6]{Fu})   of the blow-up $g:X\rightarrow Z$ at $l_{1}\cup l_{2}\cup l_{3}$.
Therefore, the multiplicity of $D_{i}$ in $g^{\ast
}(\Delta )$ must be equal to the multiplicity $m_{i}$ of $Z_0$ along $l_{i}$
(see \cite[p.\ 79]{Fu}  for the definition of $m_{i}$ and \cite[Ch.\ 5, Prop.\ 5.3]{Ha} for
the proof of a similar statement).

It suffices to prove the assertion for $D_{3}$ and we may also use the
equation $x_{1}z_{1}+x_{2}z_{2}+x_{3}z_{3}=0$ for $Z$ to eliminate $%
t_{3}=-t_{1}-t_{2}$. This  replaces $H(t_{1},t_{2},t_{3})$ by a non-zero binary form 
 $G(t_{1},t_{2})$. Then, $Z_0$ is the subscheme of $\Bbb{P}%
^{2}\times \Bbb{P}^{2}$  defined by $%
(x_{1}z_{1}+x_{2}z_{2}+x_{3}z_{3},G(x_{1}z_{1},x_{2}z_{2}))$ and $l_{3}$ the
subscheme defined by $(z_{1},z_{2},x_{3})$. It is now clear from the
definition of $m_{3}$ that $m_{3}=n$, as $G(x_{1}z_{1},x_{2}z_{2})$ is of
degree $n$ with respect to $(z_{1},z_{2})$.\end{proof}

We are now in a position to determine $C_{\text{eff}}(X)$. Recall that $\text{Pic}  \, X%
 $\ is a free abelian group of rank five (\cite[Theorem 4]{BBS}). 

\begin{prop}\label{pseudo}  The pseudoeffective cone $C_{%
\text{{\rm eff}}}(X)$ is spanned  by the 
 nine classes $[D_{i}],[E_{i}],[F_{i}]$, $1\leq i\leq 3$.
\end{prop}

\begin{proof}  By Lemma \ref{lem21}, it is enough to show that the class $[D]\in \text{Pic}\, X$ of any $\Bbb{G}_{m}^{3}$-invariant effective divisor $D$ on $X$
is in the cone spanned by the nine classes above. To do this, it suffices to
treat the case where none of $D_{1},D_{2},D_{3}$ occur in the prime
decomposition  of $D$ as $%
D_{1},D_{2} $ and $D_{3}$ are $\Bbb{G}_{m}^{3}$-invariant.

Now let $D_{U}$ be the restriction of $D$ to $U=X\setminus \cup _{i=1}^{3}D_{i}$ and let $%
g_{U}:U\rightarrow Z$ be the restriction of $g$ to $U$. Then $g_{U}$ is an
open immersion with $Z \setminus g_{U}(U)$ of codimension two in $Z$. The functorial
map $g_{U}^{\ast }:  \text{Div} \, Z\rightarrow \text{Div}\, U$ is thus an isomorphism,
which restricts to an isomorphism between the submonoids of $\Bbb{G}%
_{m}^{3}$-invariant effective divisors on $Z$ and $U$. There are therefore a
unique $\Bbb{G}_{m}^{3}$-invariant effective divisor $\Delta $ on $Z$
with $g_{U}^{\ast }(\Delta )=D_{U}$ and unique non-negative integers $n_{i}$
with $g^{\ast }(\Delta )=D+\sum_{i=1}^{3}n_{i}D_{i}$ .

By Lemma \ref{lem22} and Lemma \ref{lem23} there is a decomposition $\Delta =\Delta'+\Delta''$ into two $\Bbb{G}_{m}^{3}$-invariant
effective divisors on $Z$ where $\Delta'$ is defined by a
monomial $%
\mathbf{x}^{\mathbf{e}} \mathbf{y}^{\mathbf{f}}$ and $\Delta''$ by $%
H(x_{1}z_{1},x_{2}z_{2},x_{3}z_{3})$ for a ternary form $H(t_{1},t_{2},t_{3})
$. As the divisors of $x_{i}$ (resp.\ $z_{i}$) are given by $D_{i}+F_{i}$
(resp.\ $D_{j}+D_{k}+E_{i}$), we infer that $$g^{\ast }(\Delta')=\sum_{i=1}^{3}e_{i}(D_{i}+F_{i})+\sum_{i=1}^{3}f_{i}(D_{j}+D_{k}+E_{i}).$$ 
By Lemma \ref{lem24} we have also a decomposition $g^{\ast }(\Delta'')=n(D_{1}+D_{2}+D_{3})+D^{\ast }$ where $n=\deg $ $H$ and $D^{\ast
}$ is an effective divisor where $D_{1},D_{2}$ and $D_{3}$ do not occur. By
adding these two decompositions and comparing the result with $g^{\ast
}(\Delta )=D+\sum_{i=1}^{3}n_{i}D_{i}$, we obtain that $$D=D^{\ast
}+\sum_{i=1}^{3}e_{i}F_{i}+\sum_{i=1}^{3}f_{i}E_{i}.$$ Moreover, as $g^{\ast
}(\Delta'')$ is linearly equivalent to the divisor $%
n(D_{i}+F_{i})+n(D_{j}+D_{k}+E_{i})$ of $x_{i}^{n}z_{i}^{n}$, we get that $%
\left[ D^{\ast }\right] =n\left[ E_{i}+F_{i}\right] $ for any $i\in \{1,2,3\}
$ and that $\left[ D\right] $ belongs to the cone spanned by $\left[ E_{1}%
\right] ,\left[ E_{2}\right] ,\left[ E_{3}\right] ,\left[ F_{1}\right] ,%
\left[ F_{2}\right] $ and $\left[ F_{3}\right] $.
\end{proof}

\subsection{Computation of $\alpha (X)$} In this
section we compute Peyre's $\alpha $-invariant (see \cite[Def.\ 2.4]{Pe}) for $%
X$. To do this, we let $D_{0}$ (resp.\ $D_{4}$) be the effective divisors
given by $L(\mathbf{z})=0$ (resp.\ $M(\mathbf{x})=0$) for two fixed ternary
linear forms $L$ and $M$. We then have the following linear equivalences
\begin{equation}\label{75}
 E_{i} \sim  D_{0}-D_{j}-D_{k}
\end{equation} 
as $\text{div}(z_{i}=0)\sim D_{0}$, and
\begin{equation}\label{76}
 F_{i} \sim  D_{4}-D_{i}
\end{equation}
as $\text{div}(x_{i}=0)\sim D_{4}$. 

\begin{lem}\label{lem25} The divisor 
 $2D_{0}-D_{1}-D_{2}-D_{3}+2D_{4}$  is an anticanonical divisor on $X$.\end{lem}

\begin{proof} The canonical sheaf $\omega _{V}$ is isomorphic to $O_{V}(-2,-1)$ as $V$ is
of bidegree $(1,2)$. Further, by \cite[Theorem 4]{BBS} we have that $\omega
_{X}=f^{\ast }\omega $ for the morphism $f:X\rightarrow V$. Hence the
divisor $2D_{4}+(D_{i}+E_{j}+E_{k})$ of $M(\mathbf{x})^{2}y_{i}$
is anticanonical. Moreover, $2D_{4}+D_{i}+E_{j}+E_{k}\sim
2D_{0}-D_{1}-D_{2}-D_{3}+2D_{4}$ by \eqref{75}, thereby completing the proof.\end{proof}

Now let $C_{\text{eff}}(X)^{\vee }\subset \text{Hom}(\text{Pic} \, X\otimes  \Bbb{R}$, 
$\Bbb{R}$) be the dual cone of all linear maps $\Lambda : \text{Pic}\, X\otimes  \Bbb{R}\rightarrow  \Bbb{R}$ such that $\Lambda (\left[ D\right]
)\geq 0$ for every effective divisor $D$ on $X$,  and let $l : \text{Hom}(\text{Pic}\, X\otimes 
\Bbb{R}, \Bbb{R}$)$\rightarrow \Bbb{R}$ be the linear map which
sends $\Lambda $ to $\Lambda (\left[ -K_{X}\right] )$. We then endow $\text{Hom}(\text{Pic}\,  X 
\otimes \Bbb{R}, \Bbb{R}$) with the Lebesque measure ${\rm d}s$
normalized such that the lattice $\text{Hom}(\text{Pic} \, X, \Bbb{Z}$) has covolume 1,  and we endow $%
H_{X}=l^{-1}(1)$ with the measure ${\rm d}s/{\rm d}l$. Explicitly, if $w_{0},\ldots , w_{4}$ are
coordinates for $\text{Hom}(\text{Pic} \, X\otimes \Bbb{R}, \Bbb{R}) =\Bbb{R}^{5}$
with respect to a $\Bbb{Z}$-basis of $L$ and $%
l(w_{0},\ldots,w_{4})=\alpha _{0}w_{0}+\ldots+\alpha _{4}w_{4}$, then $$%
\frac{{\rm d}s}{{\rm d}l}={\rm d}w_{1}\ldots \widehat{{\rm d}w_{i}}\ldots {\rm d}w_{5}/\left\vert \alpha
_{i}\right\vert $$ whenever $\alpha _{i}\neq 0$.\\

After these preparations, we can now define $\alpha (X)$ as
$$\alpha (X)=\int\limits_{C_{\text{eff}}(X)^{\vee
}\cap H_{X} } \frac{{\rm d}s}{{\rm d}l}.$$
If we let $e_{0},\ldots,e_{4}$ be the $\Bbb{Z}$-basis of $L$ with 
$e_{i}([D_{j}]) = \delta _{ij}$, then we have the following

\begin{lem}\label{prop3} $  $\\
{\rm (a)} The hyperplane $H_{X} \subset \Bbb{R}%
^{5}$ is  defined by the equation $2w_{0}-w_{1}-w_{2}-w_{3}+2w_{4}=1$.\\
%{\rm (b)} One has
%$$\frac{ds}{d(l-1)}= \frac{1}{2}dw_{1}dw_{2}dw_{3}dw_{4}.$$
{\rm (b)}  The dual cone   $C_{\text{{\rm eff}}}(X)^{\vee }$ is defined by the inequalities
\begin{displaymath}
\begin{split}
  & w_{4}\geq w_{i}\geq 0, \quad 1\leq i\leq 3;\\
   & w_{0}-w_{i}-w_{j} \geq 0, \quad 1\leq i<j\leq 3.
\end{split}
\end{displaymath}
 \end{lem}

\begin{proof} (a) One has $
l(w_{0},w_{1},w_{2},w_{3},w_{4})=\sum_{i=0}^{4}e_{i}([-K_{X}])w_{j}$ by the
definition of $l$. Hence $l=2w_{0}-w_{1}-w_{2}-w_{3}+2w_{4}$ by Lemma \ref{lem25}. \\%Part (b) is a direct consequence. \\
(b) One has $\sum_{i=0}^{4}w_{i}e_{i}\in C_{\text{eff}}(X)^{\vee }$ if and only if 
$\sum_{i=0}^{4}w_{i}e_{i}([D])\geq 0$ for all $[D]\in C_{\text{eff}}(X)$.
Hence, by Proposition \ref{pseudo} we have that $(w_{0},w_{1},w_{2},w_{3},w_{4})\in C_{%
\text{eff}}(X)^{\vee }$ if and only if $\sum_{i=0}^{4}w_{i}e_{i}([D])\geq 0$
for any $D\in \{D_{1},D_{2},D_{3},E_{1},E_{2},E_{3},F_{1},F_{2},F_{3}\}$.
Now by using \eqref{75} and \eqref{76} and $e_{i}([D_{j}])= \delta _{ij}$, we
conclude that these nine inequalities are the same as the inequalities in the statement of the lemma.
\end{proof}

We are now prepared to compute the $\alpha$-invariant of $X$:

\begin{prop}\label{prop4} One has
$$\alpha(X) = \frac{1}{2^63^2} = \frac{1}{576}.$$
\end{prop}

\begin{proof} Eliminating $w_{0}=(1+w_{1}+w_{2}+w_{3}-2w_{4})/2$ and then using symmetry between $w_1, w_2, w_3$,  we obtain from Lemma \ref{prop3} that
$$\alpha(X) = \frac{1}{2}\cdot 6 \cdot  \text{vol}(Q),$$ where $Q \subset \Bbb{R}^{4}$ 
is defined by the inequalities   $$w_{4}\geq
w_{1}\geq w_{2}\geq w_{3}\geq 0 \quad \text{and} \quad w_{1}+w_{2}+2w_{4}\leq 1+w_{3}.$$ 
Changing variables by the unimodular linear transformation
$$  v_{1}=w_{1}-w_{2},\quad v_{2}=w_{2}-w_{3}, \quad
v_{3}=w_{3},\quad v_{4}=w_{4}-w_{1}, $$ 
we find that $\alpha(X) = 3  \text{vol}(P)$, %Therefore, $\det \left( 
%\frac{\delta (w_{1},w_{2},w_{3},w_{4})}{\delta (v_{1},v_{2},v_{3},v_{4})}%
%\right) =1$ and
%$\bigskip $
 where  $P\subset [0,\infty)^{4}$ is defined by 
$3v_{1}+4v_{2}+3v_{3}+2v_{4}\leq 1$. Hence, 
$$\alpha (X)=\frac{1}{4!}\frac{3}{%
3\cdot 4\cdot 3\cdot 2}=\frac{1}{2^{6}3^{2}}.$$
\end{proof}

\section{The adelic volume of $X$}\label{peyre-constant}

We keep the notation of the previous chapter. %Again let $X$ be the resolution of $V$ described in \eqref{tri}. 
The aim of this chapter
is to give an explicit description of Peyre's Tamagawa measure $%
%TCIMACRO{\U{3bc} }%
%BeginExpansion
\mu
%EndExpansion
_{H}$ on $X(\mathbf{A})=X(\Bbb{R})\times \prod\nolimits_{p}X(\Bbb{Q}%
_{p})$, and to compute the volume 
$\mu_{H}(X(\mathbf{A}))$. The interest in this comes from Peyre's prediction \cite{Pe2} that the constant $c$ in the expected asymptotic formula $$%
N(B)=cB(\log B)^{\text{rk Pic}X-1}(1+o(1))$$ should be given by $c=\alpha (X)\mu_{H}(X(\mathbf{A}))$.

\subsection{Heights and adelic metrics}\label{sec81} %Let $V\subset \Bbb{P}%
%^{2}\times \Bbb{P}^{2}$ and $X\subset \Bbb{P}^{2}\times \Bbb{P}%
%^{2}\times \Bbb{P}^{2}$ be the varieties with coordinates $(\mathbf{x},%
%\mathbf{y})$ resp.\ $(\mathbf{x},\mathbf{y},\mathbf{z})$ considered in
%previous chapters. 
The morphism $f:X\rightarrow V$ %defined by $f(\mathbf{x},%
%\mathbf{y},\mathbf{z})=(\mathbf{x},\mathbf{y})$ 
restricts to an isomorphism
between the open subsets $X^{\circ }\subset X$ and $V^{\circ }\subset  V$
defined by $x_{1}x_{2}x_{3}y_{1}y_{2}y_{3}\neq 0$. We conclude that
$$N(B)=|\{w\in V^{\circ }(\Bbb{Q}):H(w)\leq B\}|=|\{x\in X^{\circ }(%
\Bbb{Q}):H(f(x))\leq B\}|$$
where the height $H:V(\Bbb{Q})\rightarrow \Bbb{N}$ was defined in \eqref{height} for a certain choice of representatives and can also be written as 
\begin{equation*}\label{81}
   H(\mathbf{x},\mathbf{y})=\prod_{v}\max_{1\leq i,j\leq
3}\left\vert x_{i}^{2}y_{j}\right\vert _{v}.
\end{equation*}

The aim of this section is to reinterpret this height and $H\circ f:X(%
\Bbb{Q})\rightarrow \Bbb{N}$ in terms of adelic metrics on $\omega
_{V}^{-1}$ and $\omega _{X}^{-1}$ as in \cite{Pe}. These metrics
will be constructed by means of global sections on $\omega _{V}^{-1}$ and $%
\omega _{X}^{-1}=f^{\ast }(\omega _{V}^{-1})$, which we obtain by glueing
local sections on the open subsets $V_{i,j}\subset V$ and $X_{i,j}$ $\subset
X$ where $x_{i}y_{j}\neq 0$.

We write
$(\Bbb{P}^{2}\times \Bbb{P}^{2})_{i,j}$ for the open subset of $\Bbb{P}^{2}\times \Bbb{P}^{2}$ where $%
x_{i}y_{j}\neq 0$. On this set, we 
shall use affine coordinates. For  $k \not= i$ and $l \not= j$ % $k\equiv i+1$ or  $i+2\bmod 3$ and $l\equiv j+1$ or 
%$j+2\bmod 3$, 
these are given by
 $$x_{k}^{(i)}=\frac{x_{k}}{x_{i}} \quad \text{ and } \quad  %
y_{l}^{(j)}=\frac{y_{l}}{y_{j}}.$$ 
 Then $V_{i,j}\subset (\Bbb{P}^{2}\times \Bbb{P}%
^{2})_{i,j}$ is the affine hypersurface in $\Bbb{A}^{4}$ defined by $F_{ij}=0$, where
\begin{equation*}
  F_{i,j}(x_{i+1}^{(i)},x_{i+2}^{(i)},y_{j+1}^{(j)},y_{j+2}^{(j)})=x_{1}^{(i)}y_{2}^{(j)}y_{3}^{(j)}+x_{2}^{(i)}y_{1}^{(j)}y_{3}^{(j)}+x_{3}^{(i)}y_{1}^{(j)}y_{2}^{(j)};
\end{equation*}
here and in the following we put  $x_{i}^{(i)}=y_{j}^{(j)}=1$ and we interpret indices $i,j,k$ in $\Bbb{Z}/3\Bbb{Z}$. 
\\

%From now on, we interpret indices $i,j,k$ in $\Bbb{Z}/3\Bbb{Z}$. 
There is a unique global section $s$ of $\omega _{\Bbb{P}^{2}\times 
\Bbb{P}^{2}}(D)$ which for any choice of $i, j$ restricts to
\begin{equation*}s_{(\Bbb{P}^{2}\times \Bbb{P}^{2})_{i,j}}=\frac{%
{\rm d}x_{i+1}^{(i)}}{x_{i+1}^{(i)}}\wedge \frac{{\rm d}x_{i+2}^{(i)}}{x_{i+2}^{(i)}}%
\wedge \frac{{\rm d}y_{j+1}^{(j)}}{y_{j+1}^{(j)}}\wedge \frac{{\rm d}y_{j+2}^{(j)}}{%
y_{j+2}^{(j)}}\in \Gamma \left((\Bbb{P}^{2}\times \Bbb{P}^{2})_{i,j},\omega
_{\Bbb{P}^{2}\times \Bbb{P}^{2}}(D)\right).
\end{equation*}
 This can be seen directly because one has $$%
\frac{{\rm d}x_{i+1}^{(i)}}{x_{i+1}^{(i)}}\wedge \frac{{\rm d}x_{i+2}^{(i)}}{%
x_{i+2}^{(i)}}=\frac{{\rm d}x_{k+1}^{(k)}}{x_{k+1}^{(k)}}\wedge \frac{%
{\rm d}x_{k+2}^{(k)}}{x_{k+2}^{(k)}}$$ on the open subset of $\Bbb{P}^{2}$ where 
$x_{i}x_{k}\neq 0$. Alternatively, this claim is a special case of a general result for toric
varieties (see \cite[Lemma 12]{BBS}). The latter result also shows that 
$s$ is a global generator of the $O_{\Bbb{P}^{2}\times \Bbb{P}^{2}}$%
-module $\omega _{\Bbb{P}^{2}\times \Bbb{P}^{2}}(D)$.

Now put  $F=x_{1}y_{2}y_{3}+x_{2}y_{1}y_{3}+x_{3}y_{1}y_{2}$, and then define 
\begin{equation}\label{84}
\omega _{i,j}=\frac{x_{1}x_{2}x_{3}y_{1}y_{2}y_{3}}{%
x_{i}^{2}y_{j}F}s\in \Gamma \left((\Bbb{P}^{2}\times \Bbb{P}^{2}),\omega _{%
\Bbb{P}^{2}\times \Bbb{P}^{2}}(V+2H_{x_{i}}+H_{y_{i}})\right), 
\end{equation}
where $H_{x_{i}}$ (resp.\ $H_{y_{i}}$) are the prime divisors on $\Bbb{P}%
^{2}\times \Bbb{P}^{2}$ defined by the vanishing of $x_{i}$ (resp.\ $y_{j}$%
). 
 Then, $\omega
_{i,j}$ is a global generator of $\omega _{\Bbb{P}^{2}\times \Bbb{P}%
^{2}}(V+2H_{x_{i}}+H_{y_{i}})$ with
\begin{equation*}\label{85}
   \omega _{i,j}= \frac{1}{F_{i,j}}%
{\rm d}x_{i+1}^{(i)}\wedge {\rm d}x_{i+2}^{(i)}\wedge {\rm d}y_{j+1}^{(j)}\wedge
{\rm d}y_{j+2}^{(j)}
\end{equation*}
 on $(\Bbb{P}^{2}\times \Bbb{P}^{2})_{i,j}$.

We now consider the Poincar\'{e} residue map Res: $\omega _{\mathbb{P}
^{2}\times \mathbb{P}^{2}}(V)\rightarrow \iota _{\ast }\omega_{V}$ for the
inclusion map $\iota :V\rightarrow \mathbb{P}^{2}\times \mathbb{P}^{2}$. The
Poincar\'{e} residue map is usually given as a homomorphism $\Omega
_{W}^{n}(V)\rightarrow \iota _{\ast }\Omega _{V}^{n-1}$ for the inclusion
map $\iota :V\rightarrow W$ of a non-singular hypersurface $V\subset W$ in
an $n$-dimensional non-singular variety (cf. [Re3, p.\ 89], for example).
More generally, one can also use Poincar\'{e} residues to define local
sections on the canonical sheaf $\omega _{V}$ of an arbitrary normal
hypersurface $V\subset W$ (cf.\ [We]) as one still gets regular $(n-1)$-forms
on the non-singular locus $V_{ns}$ of $V$ and since $\omega _{V}= 
j_{\ast }\Omega _{V_{ns}}^{n-1}$ for the inclusion map $j:V_{ns}\rightarrow V
$. After these general remarks we return to our specific situation.

 By regarding $\omega _{i,j}$ as a local section of $\omega _{\Bbb{%
P}^{2}\times \Bbb{P}^{2}}(V)$ on $(\Bbb{P}^{2}\times \Bbb{P}%
^{2})_{i,j}$ we obtain a local section $\text{Res}(\omega _{i,j})\in \Gamma
(V_{i,j},\omega _{V})$, where
\begin{equation*}\label{86}
   \text{Res}(\omega _{i,j})=(-1)^{\text{pos}(z)+1}\frac{1}{\partial
F_{i,j}/\partial z}{\rm d}x_{i+1}^{(i)}\wedge \ldots \widehat{{\rm d}z}\ldots \wedge
{\rm d}y_{j+2}^{(j)},
\end{equation*}
on the open subset of $V_{i,j}$ where $\partial F_{i,j}/\partial z\neq 0$ and pos%
$(z)\in \{1, 2, 3, 4\}$ is the position of $z \in
\{x_{i+1}^{(i)},x_{i+2}^{(i)},y_{j+1}^{(j)},y_{j+2}^{(j)}\}$. This defines 
$\text{Res}(\omega _{i,j})$ on the non-singular locus $U_{i,j}$ of $V_{i,j}$ with 
$\text{Res}(\omega _{i,j})\neq 0$ everywhere on $U_{i,j}$. As $V_{i,j}$ is normal,
we may then extend $\text{Res}(\omega _{i,j})$ to a volume form on $V_{i,j}$ by a
standard argument (see \cite[p.\ 181]{Ha}).

Hence there is   an inverse nowhere vanishing local section $\tau _{i,j}= \text{Res}%
(\omega _{i,j})^{-1} \in \Gamma (V_{i,j},\omega _{V}^{-1})$ with
\begin{equation*}\label{87}
  \tau _{i,j}=(-1)^{\text{pos}(z)+1}\partial F_{i,j}/\partial z%
\frac{\partial }{\partial x_{i+1}^{(i)}}\wedge \ldots \widehat{\frac{\partial }{%
\partial z}}\ldots \wedge \frac{\partial }{\partial y_{j+2}^{(j)}}
\end{equation*}
 on the non-singular locus of $V_{i,j}$.

We shall also write $\sigma _{i,j}\in  \Gamma (X_{i,j},\omega _{X}^{-1})$
for the local section corresponding to $$f^{\ast }\tau _{i,j}:=\tau
_{i,j}\otimes 1\in  \Gamma \left(f^{-1}(V_{i,j}),f^{-1}\omega _{V}^{-1}\otimes
_{f^{-1}O_{V}}O_{X}\right)= \Gamma \left(X_{i,j},f^{\ast }\omega _{V}^{-1}\right).$$ As $%
\tau _{i,j}\in  \Gamma (V_{i,j},\omega _{X}^{-1})$ is inverse to the
volume form $\text{Res}(\omega _{i,j})$ on $V_{i,j}$, we conclude that $\sigma
_{i,j}$ is inverse to the volume form $\sigma _{i,j}^{-1}$ on $X_{i,j}$
corresponding to $$f^{\ast }(\text{Res}(\omega _{i,j}))=\text{Res}(\omega
_{i,j})\otimes 1\in \Gamma \left(f^{-1}(V_{i,j}),f^{-1}\omega _{V}\otimes
_{f^{-1}O_{V}}O_{X}\right)= \Gamma \left(X_{i,j},f^{\ast }\omega _{V}\right).$$

\begin{lem}\label{lem28} Let $i, j, k, l \in \{1, 2, 3\}$.\\ {\rm (a)} We have  $\tau _{i,j}=\big( x_{i}^{(k)}\big)
^{2}y_{j}^{(l)}\tau _{k,l}$ on $V_{i,j}\cap V_{k,l}$.\\
{\rm (b)} We have  $\sigma _{i,j}=\big( x_{i}^{(k)}\big) ^{2}y_{j}^{(l)}\sigma _{k,l}$ 
on $X_{i,j}\cap X_{k,l}$. 
\end{lem}

\begin{proof} (a) By \eqref{84} we have $\omega _{i,j}=\big(
x_{k}^{(i)}\big) ^{2}y_{l}^{(j)}\omega _{k,l}$ on $V_{i,j}\cap V_{k,l}$.
Hence $\text{Res}(\omega _{i,j})=\big( x_{k}^{(i)}\big) ^{2}y_{l}^{(j)}\text{Res}%
(\omega _{k,l})$ and $\tau _{i,j}=\big( x_{i}^{(k)}\big)
^{2}y_{j}^{(l)}\tau _{k,l}$ on $V_{i,j}\cap V_{k,l}$.\\
(b) Let $\overline{a}\in \Gamma (X_{i,j}\cap X_{k,l},O_{X})$ be the image
of $a\in \Gamma (X_{i,j}\cap X_{k,l},O_{X})=\Gamma (X_{i,j}\cap
X_{k,l},f^{-1}O_{V})$ under the natural map from $f^{-1}O_{V}$ to $O_{X}$%
. Then, $$f^{\ast }(a\tau _{k,l})=a\tau _{k,l}\otimes 1=\overline{a}(\tau
_{k,l}\otimes 1)=\overline{a}f^{\ast }(\tau _{k,l})$$ on $X_{i,j}\cap X_{k,l}.$ %
Hence $f^{\ast }(\tau _{i,j})=\big( x_{i}^{(k)}\big)
^{2}y_{j}^{(l)}f^{\ast }(\tau _{k,l})$ on $X_{i,j}\cap X_{k,l}$ by (a),
thereby proving the assertion in (b).\end{proof}

The lemma implies that $\tau _{i,j}\in \Gamma (V_{i,j},\omega _{V}^{-1})$
 extends to a global anticanonical section that we still 
denote by  $\tau _{i,j}\in \Gamma (V,\omega _{V}^{-1})$.   
Similarly, we let $\sigma _{i,j}$ be the global anticanonical section on $X$
defined by $\sigma _{i,j}=f^{\ast }\tau _{i,j}$.

For each place $%
v$ of $\Bbb{Q}$,
the global sections $\tau _{i,j}$, $1\leq i,j\leq 3$, define  a $v$-adic norm on $\omega _{V}^{-1}$ with
\begin{equation}\label{88}
  \left\Vert \tau (w_{v})\right\Vert _{v}=\min_{i,j}\left\vert 
\frac{\tau }{\tau _{i,j}}(w_{v})\right\vert _{v}=\min_{i,j}\left\vert \tau 
\text{Res}(\omega _{i,j})\right\vert _{v}
\end{equation}
 for a local section $\tau $ of $\omega _{V}^{-1}$ defined at $w_{v}\in V(%
\Bbb{Q}_{v})$. Here the minimum is taken over all $i,j \in
\{1,2,3\} $ such that $\tau _{i,j}(w_{v})\neq 0$. This definition is similar
to the definition in \cite[pp.\ 107-108]{Pe}, although it is called a $v$%
-adic metric there.

In the same way we may define a $v$-adic norm on $\omega _{X}^{-1}$ 
by letting
\begin{equation}\label{89}
   \left\Vert \sigma (x_{v})\right\Vert
_{v}=\min_{i,j}\left\vert \frac{\sigma }{\sigma _{i,j}}(x_{v})\right\vert
_{v}.
\end{equation}
for a local section $\sigma$ of $\omega _{X}^{-1}$ defined at $x_{v}\in X(%
\Bbb{Q}_{v})$. Here now the minimum is taken over all $i,j\in \{1,2,3\}$
such that $\sigma _{i,j}(x_{v})\neq 0$. We then have, just as in \cite[Lemma 15]{BBS},  the following result.

\begin{lem}\label{lem29} $  $\\
{\rm (a)} Let $w\in V(\Bbb{Q})$ and  $\tau$ be a local section of  $\omega _{V}^{-1}$ with $\tau (w)\neq 0$. Then  $%
H(w)=\prod_{v}\left\Vert \tau (w)\right\Vert _{v}^{-1}$.\\
{\rm (b)}  Let 
 $x\in X(\Bbb{Q})$ and  $\sigma $  be a local section of 
  $\omega
_{X}^{-1}$ with  $\sigma (x)\neq 0$. Then $%
H(f(x))=\prod_{v}\left\Vert \sigma (x)\right\Vert _{v}^{-1}$.
\end{lem}

\begin{proof} On applying the product formula $\prod_{v} 
\left\vert \alpha \right\vert _{v}=1$ for $\alpha \in \Bbb{Q}^{\ast}$, it suffices in both cases to prove the formula for one local
section. To prove (a), suppose that $w\in V_{k,l}$ and let $\tau =\tau
_{k,l}$. Then $\tau (w)\neq 0$, and  by \eqref{88} and Lemma \ref{lem28}(a) we see that
$$\left\Vert \tau (w)\right\Vert _{v}^{-1}=\max_{i,j}\left\vert \frac{\tau
_{i,j}}{\tau _{k,l}}(w_{v})\right\vert _{v}=\frac{1}{\left\vert
x_{k}^{2}y_{l}\right\vert _{v}}\max_{i,j}\left\vert
x_{i}^{2}y_{j}\right\vert _{v}$$
holds for each place $v$. Hence the desired identity $\prod_{v}\left\Vert \tau (w)\right\Vert
_{v}^{-1}=H(w)$ follows from the product formula. 

To prove (b), we may assume that $x\in X_{k,l}$ and choose $\sigma =\sigma
_{k,l}$. The proof is then the same as for (a), but based on using \eqref{89} and
Lemma \ref{lem28}(b). 
\end{proof}

\subsection{The volume of the adelic space $X(\mathbf{A})$}

We now describe Peyre's Tamagawa measure $%
%TCIMACRO{\U{3bc} }%
%BeginExpansion
\mu
%EndExpansion
_{H}$ on $X(\mathbf{A})=X(\Bbb{R})\times
\prod_{p}X(\Bbb{Q}_{p})$ defined by the adelic metric on $%
\omega _{X}^{-1}$ of all $v$-adic norms in \eqref{89}, and compute the volume of
the adelic space $X(\mathbf{A})$ with respect to this measure.

To obtain this measure, we recall the definition in \cite[(2.2.1)]{Pe} of a measure $%
%TCIMACRO{\U{3bc} }%
%BeginExpansion
\mu
%EndExpansion
_{v}$ on $X(\Bbb{Q}_{v})$ associated to a $v$-adic norm on $\omega
_{X}^{-1}$. Let $|\sigma _{i,j}^{-1}|_{v}$ be the $v$-adic density on $%
X_{i,j}(\Bbb{Q}_{v})$ of the volume form $\sigma _{i,j}^{-1}$ on $X_{i,j}$. 
Then, for a Borel subset $N_{v}$ of $X_{i,j}(\Bbb{Q}_{v})$, and  with 
the $v$-adic norm on $\omega _{X}^{-1}$ defined in \eqref{89}, we put
\begin{equation*}\label{810}
\mu_{v}(N_{v})=\int_{N_{v}}\frac{|\sigma _{i,j}^{-1}|_{v}}{\max_{k,l}\left\vert
\sigma _{k,l}\sigma _{i,j}^{-1}\right\vert _{v}}.\end{equation*}
This defines a measure $\mu_v$ on  $X(\Bbb{Q}_{v})$.
 On applying Lemma
\ref{lem28}(b), we may rewrite this as
\begin{equation*}\label{811}
\mu_{v}(N_{v}) =\int_{N_{v}}\frac{|\sigma _{i,j}^{-1}|_{v}}{%
\max_{k,l}\big\vert \big( x_{k}^{(i)}\big) ^{2}y_{l}^{(j)}\big\vert
_{v}}.
\end{equation*}
As usual, we write $%
%TCIMACRO{\U{3bc} }%
%BeginExpansion
\mu
%EndExpansion
_{\infty}=%
%TCIMACRO{\U{3bc} }%
%BeginExpansion
\mu
%EndExpansion
_{v}$ when $\Bbb{Q}_{v}=\Bbb{R}$ and $%
%TCIMACRO{\U{3bc} }%
%BeginExpansion
\mu
%EndExpansion
_{p}=%
%TCIMACRO{\U{3bc} }%
%BeginExpansion
\mu
%EndExpansion
_{v}$ when $\Bbb{Q}_{v}=\Bbb{Q}_{p}$.

\begin{lem}\label{lem30} 
Let $$D=\left\{w\in V^{\circ }(\Bbb{R}):| x_{1}| \leq|
x_{3}| ,\, |x_{2}| \leq |
x_{3}|, \,| y_{1}| \geq |
y_{2}|, \,x_{1}^{(3)}>0 , \,y_{1}^{(3)}>0\right\}.$$
Then 
$$\mu_{\infty }(X(\Bbb{R}))=24\int_{D}\frac{\big\vert \text{{\rm Res}}\big(
{\rm d}x_{1}^{(3)}\wedge {\rm d}x_{2}^{(3)}\wedge {\rm d}y_{1}^{(3)}\wedge {\rm d}y_{2}^{(3)}\big)
\big\vert }{\max \big( y_{1}^{(3)}, 1\big) }.$$ 
\end{lem}

\begin{proof}   As the inverse image $%
f^{\ast }\left( \tau _{i,j}^{-1}\right) $ of the volume form $\tau
_{i,j}^{-1}= \text{Res}(\omega _{i,j})$ on $V_{i,j}$ is sent to the volume form $%
\sigma _{i,j}^{-1}$ on $X_{i,j}$ under the canonical map from $f^{\ast
}\omega _{V}\rightarrow \omega _{X}$, we have thus for Borel subsets $%
N\subset X^{\circ }(\Bbb{R})$ that
\begin{equation}\label{812}
%TCIMACRO{\U{3bc} }%
%BeginExpansion
\mu
%EndExpansion
_{\infty }(N)=\int_{f(N)}\frac{|\text{Res}(\omega _{i,j})|}{%
\max_{k,l}\left\vert \sigma _{k,l}\text{Res}(\omega _{i,j})\right\vert }%
=\int_{f(N)}\frac{|\text{Res}(\omega _{i,j})|}{\max_{k,l}\big\vert \big(
x_{k}^{(i)}\big) ^{2}y_{l}^{(j)}\big\vert }.
\end{equation}
for any fixed $i,j\in \{1,2,3\}$. 

The hyperoctahedral group $\Bbb{Z}_{2}\wr S_{3}$ of order $2^{3}\times 3!$
acts on the affine hypersurface in $\Bbb{A}^{6}$ defined by $F=0$. This
group consists of signed symmetries over $\rho \in S_{3}$ sending $%
(x_{i},y_{i})$ to one of $(x_{\rho (i)},y_{\rho (i)})$ or $-(x_{\rho
(i)},y_{\rho (i)})$ for $i\in \{1,2,3\}$ and we obtain in this way an action
of $\Bbb{Z}_{2}\wr S_{3}$ on $V$ as well. %As the symmetry on $\Bbb{A}%
%^{6}$ sending symmetries on  and $W^{\circ }$. 
As the symmetry sending all $(%
\mathbf{x},\mathbf{y})$ to $-(\mathbf{x},\mathbf{y})$ is trivial on $V$, we
get in fact a faithful action of the octahedral group $O$ of order 24 on $V$, which
preserves $V^{\circ }$. The set $D$ is a fundamental domain for the (measure-preserving) action of this group, hence 
 $%
%TCIMACRO{\U{3bc} }%
%BeginExpansion
\mu
%EndExpansion
_{\infty }(V^{\circ }(\Bbb{R}))=24%
%TCIMACRO{\U{3bc} }%
%BeginExpansion
\mu
%EndExpansion
_{\infty }(D)$. 

We now apply \eqref{812} with $i=j=3$. Then $\omega _{3,3}={\rm d}x_{1}^{(3)}\wedge
{\rm d}x_{2}^{(3)}\wedge {\rm d}y_{1}^{(3)}\wedge {\rm d}y_{2}^{(3)}$ and 
$$\max_{1\leq k\leq 3}\big\vert \big( x_{k}^{(3)}\big) ^{2}\big\vert
\max_{1\leq l\leq 3}\big\vert y_{l}^{(3)}\big\vert =\max \left(
\big\vert y_{1}^{(3)}\big\vert , 1\right)$$
on $D$. Hence $$%
%TCIMACRO{\U{3bc} }%
%BeginExpansion
\mu
%EndExpansion
_{\infty }(D)=\int_{D}\frac{\big\vert \text{Res}\big( {\rm d}x_{1}^{(3)}\wedge
{\rm d}x_{2}^{(3)}\wedge {\rm d}y_{1}^{(3)}\wedge {\rm d}y_{2}^{(3)}\big) \big\vert }{\max
\big( \big\vert y_{1}^{(3)}\big\vert , 1\big) }$$ and we are
done.\\
\end{proof}

We are now prepared to compute $\mu_{\infty}(X(\Bbb{R}))$ explicitly. This is the counterpart to Lemma \ref{supermellin2}. 

\begin{lem}\label{lem31} We have $%
%TCIMACRO{\U{3bc} }%
%BeginExpansion
\mu
%EndExpansion
_{\infty }(X(\mathbb{R}))=96\log 2-12+4\pi ^{2}$. 
\end{lem}

\begin{proof} Set $t_{1}=x_{1}^{(3)}$, $t_{2}=x_{2}^{(3)}$, $%
u_{1}=y_{1}^{(3)}$ and $u_{2}=y_{2}^{(3)}$. Then $%
F_{3,3}=t_{1}u_{2}+t_{2}u_{1}+u_{1}u_{2}$ and $$\left\vert \text{Res}(\omega
_{3,3})\right\vert =\frac{{\rm d}t_{1}{\rm d}u_{1}{\rm d}t_{2}}{\left\vert \partial
F_{3,3}/\partial u_{2}\right\vert }=\frac{{\rm d}t_{1}{\rm d}u_{1}{\rm d}t_{2}}{\left\vert
t_{1}+u_{1}\right\vert }.$$ Moreover, we have the equivalences 
\begin{displaymath}
 \left\vert y_{1}\right\vert \geq \left\vert y_{2}\right\vert
\Longleftrightarrow \left\vert u_{1}\right\vert \geq \left\vert u_{2}\right\vert
\Longleftrightarrow \left\vert t_{2}\right\vert \leq \left\vert
t_{1}+u_{1}\right\vert 
\end{displaymath}
 as $-u_{1}t_{2}=\left( t_{1}+u_{1}\right) u_{2}$ on $V$. By the previous
lemma we conclude that
\begin{displaymath}
\begin{split}%
%TCIMACRO{\U{3bc} }%
%BeginExpansion
\frac{\mu
%EndExpansion
_{\infty }(X(\mathbf{R}))}{24}&=\int_{0}^{\infty
}\int_{-1}^{1} \int_{t_{2}=0}^{\min (1,\left\vert t_{1}+u_{1}\right\vert )}\frac{%
{\rm d}t_{2}}{\left\vert t_{1}+u_{1}\right\vert } \frac{{\rm d}t_{1}\, {\rm d}u_{1}}{\max
\left( u_{1}, 1\right) }\\
&= \int_{0}^{\infty }\int_{-1}^{1}\min \left(\frac{1}{\left\vert
t+u\right\vert },1\right)\frac{{\rm d}t\, {\rm d}u}{\max \left( u, 1\right) }\\
&=\int_{0}^{2}%
\frac{2-u+\log (u+1)}{\max \left( u, 1\right) }{\rm d}u+\int_{2}^{\infty
}\log \left( \frac{u+1}{u-1}\right) \frac{{\rm d}u}{u},
\end{split}
\end{displaymath}
and a straightforward computation now shows that this quantity equals  $4\log 2-\frac{1}{2}+\frac{\pi ^{2}}{6%
}$, as desired. \end{proof}

To compute the $p$-adic volumes $%
%TCIMACRO{\U{3bc} }%
%BeginExpansion
\mu
%EndExpansion
_{p}(X(\Bbb{Q}_{p}))$, we shall make use of the scheme $\underline{X}%
\subset \Bbb{P}_{\Bbb{Z}}^{2}\times \Bbb{P}_{\Bbb{Z}}^{2}\times 
\Bbb{P}_{\Bbb{Z}}^{2}$ with coordinates $(\mathbf{x},\mathbf{y},%
\mathbf{z})$ defined by the equations $x_{1}z_{1}+x_{2}z_{2}+x_{3}z_{3}=0$
and $y_{1}z_{1}=y_{2}z_{2}=y_{3}z_{3}$. It is smooth over $\Bbb{Z}$, and
there is an extension of $f:X\rightarrow V$ to a morphism $\underline{f}:%
\underline{X}\rightarrow \underline{V}$ with $\underline{f}(\mathbf{x}, 
\mathbf{y}, \mathbf{z})=(\mathbf{x},\mathbf{y})$ onto the subscheme $%
\underline{V}\subset \Bbb{P}_{\Bbb{Z}}^{2}\times \Bbb{P}_{\Bbb{Z}%
}^{2}$ defined by $x_{1}y_{2}y_{3}+x_{2}y_{1}y_{3}+x_{3}y_{1}y_{2}=0$. There
is also a functorial homomorphism $\ \underline{f}^{\ast }\omega _{%
\underline{V}/\Bbb{Z}}\rightarrow \omega _{\underline{X}/\Bbb{Z}}$ of
invertible $O_{X}$-modules for the relative canonical (or dualising)
invertible sheaves $\omega _{\underline{V}/\Bbb{Z}}$ and $\omega _{%
\underline{X}/\Bbb{Z}}$, which must be an isomorphism as $f$ and the base
extensions $\underline{f}_{\Bbb{F}_{p}}:\underline{X}_{\Bbb{F}%
_{p}}\rightarrow \underline{V}_{\Bbb{F}_{p}}$ are crepant (see \cite[Theorem 4]{BBS}).

Now consider the dual isomorphism from $\omega _{\underline{X}/\Bbb{Z}%
}^{-1}$ to $\underline{f}^{\ast }\omega _{\underline{V}/\Bbb{Z}}^{-1}$.
It extends the isomorphism from $\omega _{X}^{-1}$ to $f^{\ast }\omega
_{V}^{-1}$ that was used to define $\sigma _{i,j}=f^{\ast }\tau _{i,j}$. We
may thus extend $\sigma _{i,j}$ to global sections $\underline{\sigma }%
_{i,j}=f^{\ast }\underline{\tau }_{i,j}$ of $\omega _{\underline{X}/\Bbb{Z%
}}^{-1}=\underline{f}^{\ast }\omega _{\underline{V}/\Bbb{Z}}^{-1}$ as
follows. We first define $\underline{\tau }_{i,j}$ and $\underline{\sigma }%
_{i,j}$ on the principal open subsets where $x_{i}y_{j}\neq 0$ in the same
way as we defined $\sigma _{i,j}$ and $\tau _{i,j}$, and we then use an
analogue of Lemma \ref{lem28} for $\underline{V}$ and $\underline{X}$ to extend these
sections to global sections.

\begin{lem}\label{lem32} For all primes $p$ one has $$ 
%TCIMACRO{\U{3bc} }%
%BeginExpansion
\mu
%EndExpansion
_{p}(X(\Bbb{Q}_{p}))=1+\frac{5}{p}+\frac{5}{p^{2}}+\frac{1}{p^{3}}.$$%
 \end{lem}

\textbf{Proof.} In \cite[Def.\ 2.9]{Sa}  there is defined a measure $m_{p}$ on $%
\underline{X}(\Bbb{Z}_{p})$ called the model measure for which $m_{p}(%
\underline{X}(\Bbb{Z}_{p}))=|\underline{X}(\Bbb{F}_{p})|/p^{\dim X}$
(see \cite[Cor.\ 2.15]{Sa}). As $\underline{X}\times _{\Bbb{Z}}\Bbb{F}_{p}$
is a $\Bbb{P}^{1}$-bundle over a del Pezzo surface $B$ of degree 6 over $%
\Bbb{F}_{p}$, we conclude that $|\underline{X}(\Bbb{F}_{p})|=\left(
p^{2}+4p+1\right) \left( p+1\right) $ and $$m_{p}(\underline{X}(\Bbb{Z}%
_{p}))=1+\frac{5}{p}+\frac{5}{p^{2}}+\frac{1}{p^{3}}.$$
As $\underline{X}$ is proper over $\Bbb{Z}$, there is a natural bijection 
$\underline{X}(\Bbb{Z}_{p})=X(\Bbb{Q}_{p})$. To complete the proof of
the lemma, it is thus enough to show that $m_{p}=%
%TCIMACRO{\U{3bc} }%
%BeginExpansion
\mu
%EndExpansion
_{p}$. The definitions of $m_{p}$ and $%
%TCIMACRO{\U{3bc} }%
%BeginExpansion
\mu
%EndExpansion
_{p}$ are both based on Peyre's construction \cite[(2.2.1)]{Pe} of a
measure on $X(\Bbb{Q}_{v})$ associated to a $v$-adic norm on $\omega
_{X}^{-1}$. For $m_{p}$ one uses a $p$-adic norm $\|\cdot\| _{p}^{\ast }$ called the model norm, as described in \cite[(2.9)]{Sa}.  Thus it only remains to prove that this norm coincides with the $p$-adic
norm $\|\cdot\|_{p}$ in \eqref{89} used to define $%
%TCIMACRO{\U{3bc} }%
%BeginExpansion
\mu
%EndExpansion
_{p}$.

Therefore, let $\sigma $ be a local section of $\omega _{X}^{-1}$ defined at 
$x_{p}\in X(\Bbb{Q}_{p})$. To show that $\| \sigma\|
_{p}^{\ast }=\| \sigma \|_{p}$ in a neighbourhood of $x_{p}
$, we choose $i,j$ such that $x_{p}\in X_{i,j}(\Bbb{Q}_{p})$. The
restriction of $\omega _{X}^{-1}$ to $U=X_{i,j}$ is a free $O_{U}$-module
generated by $\sigma _{i,j}$ as $\sigma _{i,j}$ is the inverse to the volume
form $f^{\ast }(\text{Res}(\omega _{i,j}))$ on $X_{i,j}$. By the same argument
one obtains that the restriction of $\omega _{\underline{X}/\Bbb{Z}}^{-1}$
to $\underline{U}=\underline{X}_{i,j}$ is a free $O_{\underline{U}}$-module
generated by $\underline{\sigma }_{i,j}$. As $\underline{\sigma }_{i,j}$
restricts to $\sigma _{i,j}$ on $X_{i,j}$, we conclude from  the definition of
the model norm (see \cite[1.9 and 2.9]{Sa}) that $\| \sigma
_{i,j}\| _{p}^{\ast }=1$ on $X_{i,j}$, and by \eqref{89} that $%
|\sigma _{i,j}|_{p}=1$ on $X_{i,j}$. Hence $$\|
\sigma \| _{p}^{\ast }=|\sigma /\sigma _{i,j}|
_{p}\|\sigma _{i,j}\|_{p}^{\ast }=| \sigma
/\sigma _{i,j}|_{p}\|\sigma _{i,j}\|
_{p}=\| \sigma \| _{p}$$ in a neighbourhood of $x_{p}$, as
was to be shown.\\

Now let $$L_{p}(s, \text{Pic}\,\overline{X})= \det (1-p^{-s}\text{Fr}_{p}\mid \text{Pic }%
\left( X_{\overline{\Bbb{F}}_{p}}) \otimes \Bbb{Q}\right)^{-1}.$$ 
 Then, as Gal($\overline{\Bbb{F}}_{p}/\Bbb{F}_{p}$) acts
trivially on Pic $ ( X_{\overline{\Bbb{F}}_{p}} ) \cong\Bbb{Z}^{5}
$, we conclude that $$L_{p}(s, \text{Pic} \,\overline{X})=(1-p^{-s})^{-5},$$  so that %the global $L$-function 
$$L(s, \text{Pic} \,\overline{X})=\prod\limits_{\text{all }p}L_{p}(s,\text{Pic} \, \overline{X}%
) = \zeta(s)^5$$ 
%equals $\zeta(s)^5$, with $\zeta(s)$ the Riemann zeta function.
In particular, $$%
\lim_{s\rightarrow 1}(s-1)^{5}L(s, \text{Pic}\, \overline{X})=1 \quad \text{ and } \quad L_{p}(1,\text{Pic}\,  %
\overline{X})^{-1}=\left(\frac{p-1}{p}\right)^{5}.$$ Peyre's measure $\mu _{H}$ on $X(%
\mathbf{A})=X(\Bbb{R})\times \prod_p X(\Bbb{Q}_{p})$
is therefore  given by $$\mu _{H}=\mu _{\infty }\times
\prod_p\Big(\frac{p-1}{p}\Big)^{5}\mu _{p},$$ and it is shown in
\cite[Def.\ 4.6]{Pe2} that this gives a well defined measure on $X(\mathbf{A}%
)$. As $X(\Bbb{Q})$ is dense in $X(\mathbf{A})$, we  conclude (see \cite[Def.\ 4.8]{Pe2}) that 
\begin{equation}\label{813}
    \mu _{H}(X(\mathbf{A}))=\mu _{\infty }(X(\Bbb{R}%
))\prod_p \left(\frac{p-1}{p}\right)^{5}\mu _{p}(X(\Bbb{Q}_{p})).
\end{equation}

We may now combine the conclusions of
Lemma \ref{lem31}, Lemma \ref{lem32} and \eqref{813} to infer the following result. 
\begin{prop}\label{prop5}  Peyre's  Tamagawa constant 
 $\tau _{H}(X)=\mu _{H}(X(\mathbf{A}))$  associated to the adelic metric of all $v$%
 -adic norms in \eqref{89} is given by
\begin{equation*}\label{814}
   \tau _{H}(X)=\left( 96\log 2-12+4\pi ^{2}\right)
\prod_{p}\left( 1+\frac{5}{p}+\frac{5}{p^{2}}+\frac{1}{p^{3}}%
\right) .
\end{equation*}
\end{prop}

\subsection{The leading term of the asymptotic formula} We finally show that the asymptotic formula for $N(B)$ in Theorem \ref{thm1}
is in accordance with conjectures made in \cite{Pe2}.  As the biprojective
threefold $V$ defined by \eqref{1} is singular, we cannot refer to the original
conjectures of Manin  \cite{FMT} and Peyre \cite{Pe}. To overcome this, we make
use of the observation in Section \ref{sec81} that
\begin{equation*}\label{815}
  N(B)= \left\{ x\in X^{\circ }(\Bbb{Q}%
):(H\circ f)(x)\leq B\right\} .
\end{equation*}
 We have also seen in \eqref{89} that
\begin{equation*}\label{816}
 H(f(x))=\prod_{v}\| \sigma
(x)\|_{v}^{-1}
\end{equation*}
 for a local anticanonical section $\sigma $ with $\sigma (x)\neq 0$. We may therefore refer to the
conjectures of Peyre \cite{Pe2} for ``almost" Fano varieties instead. The
following result shows that $X$ satisfies the three conditions for being
such a variety.

\begin{lem}\label{lem33} Let $X\subset \Bbb{P}^{2}\times 
\Bbb{P}^{2}\times \Bbb{P}^{2}$ be as before. Then\\ 
{\rm (a)}  $H^{1}(X,O_{X})=H^{2}(X,O_{X})=0$;\\
{\rm (b)} The geometric Picard group  Pic$%
(X_{\overline{\Bbb{Q}}})$ is torsion-free;\\
{\rm (c$ $)} The anticanonical class is in the interior of $C_{\text{{\rm eff}}}(X)$. 
\end{lem}

\begin{proof}To prove (a), we apply the Leray spectral sequence $%
H^{i}(Y,R^{j}p_{\ast }O_{X})\Longrightarrow H^{i+j}(X,O_{X})$ to the $%
\Bbb{P}^{1}$-bundle $p:X\rightarrow Y$. Then, we obtain isomorphisms $%
H^{i}(Y,O_{Y})=H^{i}(X,O_{X})$ for all $i$, with $%
H^{1}(Y,O_{Y})=H^{2}(Y,O_{Y})=0$ for a del Pezzo surface.

For (b), we use that $X$ is a $\Bbb{P}^{1}$-bundle over a del Pezzo
surface $Y$ of degree 6. This gives $\text{Pic}(X_{\overline{\Bbb{Q}}})\cong %
\text{Pic}(Y_{\overline{\Bbb{Q}}})\oplus \Bbb{Z}\cong \Bbb{Z}^{5}$.

Finally (c$ $) follows from Lemma \ref{prop3} and the fact that $3D_{i}+E_{j}+E_{k}+F_{i}$
is an anticanonical divisor if $\{i,j,k\} =  \{1,2,3\}$ (see
Lemma \ref{lem25}).
\end{proof}

If we implicitly assume that there are no accumulating subvarieties on $X$
outside $X\setminus X^{\circ }$, then Peyre's ``empiric formula" in \cite[(5.1)]{Pe2} for
an almost Fano variety $X$ suggests that
\begin{equation*}\label{817}
  \left\{ x\in X^{\circ }(\mathbf{Q}):H(f(x))\leq B\right\}
\sim \Theta _{H}(X)B(\log B)^{r-1},
\end{equation*}
where $r= \text{rk Pic }X$ and $\Theta _{H}(X)=\alpha (X)\tau _{H}(X)$. As $r=1+\text{rk Pic } Y = 5$
  and
$$\Theta _{H}(X)=\frac{\pi
^{2}-3+24\log 2}{144}\prod_{p}\left( 1+\frac{5}{p}+\frac{5}{p^{2}}+%
\frac{1}{p^{3}}\right) $$
by Propositions \ref{prop4} and \ref{prop5}, the asymptotic formula in Theorem \ref{thm1} is therefore of
the form predicted by Peyre \cite{Pe2}.

\end{document}